\documentclass[10pt]{article}
\usepackage{hyperref}
\usepackage{amsmath}
\usepackage{amsfonts}
\usepackage{amssymb}
\usepackage{amsthm}

\usepackage[shortlabels]{enumitem}
\usepackage{verbatim}
\usepackage{graphicx}
\usepackage{microtype}
\usepackage{fullpage}
\usepackage{xparse}
\usepackage{mathrsfs}
\usepackage[capitalise]{cleveref}
\usepackage{mleftright}

\crefname{thm}{Theorem}{Theorems}
\Crefname{thm}{Theorem}{Theorems}
\crefname{conj}{Conjecture}{Conjectures}
\Crefname{conj}{Conjecture}{Conjectures}
\crefname{prop}{Proposition}{Propositions}
\Crefname{prop}{Proposition}{Propositions}
\crefname{cor}{Corollary}{Corollaries}
\Crefname{cor}{Corollary}{Corollaries}
\crefname{defn}{Definition}{Definitions}
\Crefname{defn}{Definition}{Definitions}
\crefname{rmk}{Remark}{Remarks}
\Crefname{rmk}{Remark}{Remarks}
\crefname{prob}{Problem}{Problems}
\Crefname{prob}{Problem}{Problems}
\crefname{question}{Question}{Questions}
\Crefname{question}{Question}{Questions}

\crefname{enumi}{}{}
\Crefname{enumi}{}{}

\crefname{figure}{Figure}{Figures}
\Crefname{figure}{Figure}{Figures}

\crefformat{equation}{(#2#1#3)}
\crefrangeformat{equation}{(#3#1#4) to~(#5#2#6)}
\crefmultiformat{equation}{(#2#1#3)}%
{ and~(#2#1#3)}{, (#2#1#3)}{, and~(#2#1#3)}

\Crefformat{equation}{(#2#1#3)}
\Crefrangeformat{equation}{(#3#1#4) to~(#5#2#6)}
\Crefmultiformat{equation}{(#2#1#3)}%
{ and~(#2#1#3)}{, (#2#1#3)}{, and~(#2#1#3)}

\begin{document}

\NewDocumentCommand{\C}{}{{\mathbb{C}}}
\NewDocumentCommand{\R}{}{{\mathbb{R}}}
\NewDocumentCommand{\Q}{}{{\mathbb{Q}}}
\NewDocumentCommand{\Z}{}{{\mathbb{Z}}}
\NewDocumentCommand{\N}{}{{\mathbb{N}}}
\NewDocumentCommand{\M}{}{{\mathbb{M}}}
\NewDocumentCommand{\grad}{}{\nabla}
\NewDocumentCommand{\sA}{}{\mathcal{A}}
\NewDocumentCommand{\sF}{}{\mathcal{F}}
\NewDocumentCommand{\sH}{}{\mathcal{H}}
\NewDocumentCommand{\sD}{}{\mathcal{D}}
\NewDocumentCommand{\sB}{}{\mathcal{B}}
\NewDocumentCommand{\sC}{}{\mathcal{C}}
\NewDocumentCommand{\sE}{}{\mathcal{E}}
\NewDocumentCommand{\sL}{}{\mathcal{L}}
\NewDocumentCommand{\sT}{}{\mathcal{T}}
\NewDocumentCommand{\sO}{}{\mathcal{O}}
\NewDocumentCommand{\sP}{}{\mathcal{P}}
\NewDocumentCommand{\sQ}{}{\mathcal{Q}}
\NewDocumentCommand{\sR}{}{\mathcal{R}}
\NewDocumentCommand{\sM}{}{\mathcal{M}}
\NewDocumentCommand{\sI}{}{\mathcal{I}}
\NewDocumentCommand{\sK}{}{\mathcal{K}}
\NewDocumentCommand{\Span}{}{\mathrm{span}}
\NewDocumentCommand{\fM}{}{\mathfrak{M}}
\NewDocumentCommand{\fN}{}{\mathfrak{N}}
\NewDocumentCommand{\fX}{}{\mathfrak{X}}
\NewDocumentCommand{\fY}{}{\mathfrak{Y}}
\NewDocumentCommand{\gammat}{}{\tilde{\gamma}}
\NewDocumentCommand{\ct}{}{\tilde{c}}
\NewDocumentCommand{\bt}{}{\tilde{b}}
\NewDocumentCommand{\ch}{}{\hat{c}}
\NewDocumentCommand{\hd}{}{\hat{d}}
\NewDocumentCommand{\eh}{}{\hat{e}}
\NewDocumentCommand{\Ut}{}{\tilde{U}}
\NewDocumentCommand{\Gt}{}{\widetilde{G}}
\NewDocumentCommand{\Vt}{}{\tilde{V}}
\NewDocumentCommand{\ah}{}{\hat{a}}
\NewDocumentCommand{\at}{}{\tilde{a}}
\NewDocumentCommand{\Yh}{}{\widehat{Y}}
\NewDocumentCommand{\Yt}{}{\widetilde{Y}}
\NewDocumentCommand{\Ah}{}{\widehat{A}}
\NewDocumentCommand{\Ch}{}{\widehat{C}}
\NewDocumentCommand{\At}{}{\widetilde{A}}
\NewDocumentCommand{\Bt}{}{\widetilde{B}}
\NewDocumentCommand{\Bh}{}{\widehat{B}}
\NewDocumentCommand{\Mh}{}{\widehat{M}}
\NewDocumentCommand{\sLh}{}{\widehat{\sL}}
\NewDocumentCommand{\Sh}{}{\widehat{S}}
\NewDocumentCommand{\Xt}{}{\widetilde{X}}
\NewDocumentCommand{\Lt}{}{\widetilde{L}}
\NewDocumentCommand{\Xh}{}{\widehat{X}}
\NewDocumentCommand{\Lh}{}{\widehat{L}}
\NewDocumentCommand{\Eh}{}{\widehat{E}}
\NewDocumentCommand{\Fh}{}{\widehat{F}}
\NewDocumentCommand{\wh}{}{\hat{w}}
\NewDocumentCommand{\hh}{}{\hat{h}}
\NewDocumentCommand{\Phih}{}{\widehat{\Phi}}

\NewDocumentCommand{\AMatrix}{}{\mathcal{A}}
\NewDocumentCommand{\AMatrixh}{}{\widehat{\AMatrix}}

\NewDocumentCommand{\Extend}{}{\mathscr{E}}

\NewDocumentCommand{\Zh}{}{\widehat{Z}}

\NewDocumentCommand{\Compact}{}{\mathcal{K}}


\NewDocumentCommand{\fg}{}{\mathfrak{g}}
\NewDocumentCommand{\etai}{}{\hat{\eta}}
\NewDocumentCommand{\etat}{}{\tilde{\eta}}

\NewDocumentCommand{\Deriv}{}{\mathscr{D}}
\NewDocumentCommand{\BofA}{}{\mathscr{B}}
\NewDocumentCommand{\ADeriv}{}{\mathscr{A}}

\NewDocumentCommand{\transpose}{}{\top}

\NewDocumentCommand{\ICond}{}{\sC}

\NewDocumentCommand{\LebDensity}{}{\sigma_{\mathrm{Leb}}}


\NewDocumentCommand{\Lie}{m}{\sL_{#1}}

\NewDocumentCommand{\ZygSymb}{}{\mathscr{C}}

\NewDocumentCommand{\Zyg}{m o}{\IfNoValueTF{#2}{\ZygSymb^{#1}}{\ZygSymb^{#1}(#2) }}
\NewDocumentCommand{\ZygX}{m m o}{\IfNoValueTF{#3}{\ZygSymb^{#2}_{#1}}{\ZygSymb^{#2}_{#1}(#3) }}

\NewDocumentCommand{\CSpace}{m o}{\IfNoValueTF{#2}{C(#1)}{C(#1;#2)}}

\NewDocumentCommand{\CjSpace}{m o o}{\IfNoValueTF{#2}{C^{#1}}{ \IfNoValueTF{#3}{ C^{#1}(#2)}{C^{#1}(#2;#3) } }  }

\NewDocumentCommand{\CXjSpace}{m m o}{\IfNoValueTF{#3}{C^{#2}_{#1}}{ C^{#2}_{#1}(#3) } }

\NewDocumentCommand{\ASpace}{m m o}{\mathscr{A}^{#1,#2}\IfNoValueTF{#3}{}{(#3)}}

\NewDocumentCommand{\AXSpace}{m m m o}{\mathscr{A}_{#1}^{#2,#3}\IfNoValueTF{#4}{}{(#4)}}

\NewDocumentCommand{\OSpace}{m o}{\mathscr{O}^{#1}_{b}\IfNoValueTF{#2}{}{(#2)}}
\NewDocumentCommand{\ONorm}{m m o}{\Norm{#1}[\IfNoValueTF{#3}{\OSpace{#2}}{\OSpace{#2}[#3]}]}

\NewDocumentCommand{\sBSpace}{m m m}{\mathscr{B}^{#1,#2}_{#3}}
\NewDocumentCommand{\sBNorm}{m m m m}{\Norm{#1}[\sBSpace{#2}{#3}{#4}]}

\NewDocumentCommand{\DSpace}{m m m m m}{\mathscr{D}^{#1,#2}_{#3,#4,#5}}
\NewDocumentCommand{\DNorm}{m m m m m m}{\Norm{#1}[\DSpace{#2}{#3}{#4}{#5}{#6} ]}

\NewDocumentCommand{\ComegaSpace}{m o o}{\IfNoValueTF{#2}{\CjSpace{\omega,#1}}{
\IfNoValueTF{#3}
{\CjSpace{\omega,#1}[#2]}
{\CjSpace{\omega,#1}[#2][#3]}
}
}

\NewDocumentCommand{\CXomegaSpace}{m m o o}{\IfNoValueTF{#3}{\CXjSpace{#1}{\omega,#2}}{
\IfNoValueTF{#4}
{\CXjSpace{#1}{\omega,#2}[#3]}
{\CXjSpace{#1}{\omega,#2}[#3][#4]}
}
}

\NewDocumentCommand{\ANorm}{m m m o}{\IfNoValueTF{#4}{\Norm{#1}[ \ASpace{#2}{#3} ]}{ \Norm{#1}[ \ASpace{#2}{#3}[#4] ] }}
\NewDocumentCommand{\BANorm}{m m m o}{\IfNoValueTF{#4}{\BNorm{#1}[ \ASpace{#2}{#3} ]}{ \BNorm{#1}[ \ASpace{#2}{#3}[#4] ] }}

\NewDocumentCommand{\AXNorm}{m m m m o}{\IfNoValueTF{#5}{\Norm{#1}[ \AXSpace{#2}{#3}{#4} ]}{ \Norm{#1}[ \AXSpace{#2}{#3}{#4}[#5] ] }}
\NewDocumentCommand{\BAXNorm}{m m m m o}{\IfNoValueTF{#5}{\BNorm{#1}[ \AXSpace{#2}{#3}{#4} ]}{ \BNorm{#1}[ \AXSpace{#2}{#3}{#4}[#5] ] }}

\NewDocumentCommand{\ComegaNorm}{m m o o}{\IfNoValueTF{#3}{ \Norm{#1}[\ComegaSpace{#2}] }
{
\IfNoValueTF{#4}
{\Norm{#1}[\ComegaSpace{#2}[#3]] }
{\Norm{#1}[\ComegaSpace{#2}[#3][#4]] }
}
}

\NewDocumentCommand{\CXomegaNorm}{m m m o o}{\IfNoValueTF{#4}{ \Norm{#1}[\CXomegaSpace{#2}{#3}] }
{
\IfNoValueTF{#5}
{\Norm{#1}[\CXomegaSpace{#2}{#3}[#4]] }
{\Norm{#1}[\CXomegaSpace{#2}{#3}[#4][#5]] }
}
}

\NewDocumentCommand{\HSpace}{m m o o}{\IfNoValueTF{#3}{C^{#1,#2}}{ \IfNoValueTF{#4} {C^{#1,#2}(#3)} {C^{#1,#2}(#3;#4)} }}

\NewDocumentCommand{\HXSpace}{m m m o}{\IfNoValueTF{#4}{C_{#1}^{#2,#3}}{  {C_{#1}^{#2,#3}(#4)}  }}

\NewDocumentCommand{\ZygSpace}{m o o}{\IfNoValueTF{#2}{\ZygSymb^{#1}}{ \IfNoValueTF{#3} { \ZygSymb^{#1}(#2) }{\ZygSymb^{#1}(#2;#3) } } }

\NewDocumentCommand{\ZygSpacediff}{m o o}{\IfNoValueTF{#2}{\ZygSymb^{#1}}{ \IfNoValueTF{#3} { \ZygSymb^{#1}(#2) }{\ZygSymb^{#1}(#2;#3) } } }

\NewDocumentCommand{\ZygSpaceloc}{m o o}{\IfNoValueTF{#2}{\ZygSymb^{#1}_{\mathrm{loc}}}{ \IfNoValueTF{#3} { \ZygSymb^{#1}_{\mathrm{loc}}(#2) }{\ZygSymb^{#1}_{\mathrm{loc}}(#2;#3) } } }

\NewDocumentCommand{\ZygSpacemap}{m o o}{\IfNoValueTF{#2}{\ZygSymb^{#1}_{\mathrm{loc}}}{ \IfNoValueTF{#3} { \ZygSymb^{#1}_{\mathrm{loc}}(#2) }{\ZygSymb^{#1}_{\mathrm{loc}}(#2;#3) } } }

\NewDocumentCommand{\ZygXSpace}{m m o}{\IfNoValueTF{#3}{\ZygSymb^{#2}_{#1}}{\ZygSymb^{#2}_{#1}(#3) }}

\NewDocumentCommand{\Norm}{m o}{\IfNoValueTF{#2}{\| #1\|}{\|#1\|_{#2} }}
\NewDocumentCommand{\BNorm}{m o}{\IfNoValueTF{#2}{\mleft\| #1\mright\|}{\mleft\|#1\mright\|_{#2} }}

\NewDocumentCommand{\CjNorm}{m m o o}{ \IfNoValueTF{#3}{ \Norm{#1}[\CjSpace{#2}]} { \IfNoValueTF{#4}{\Norm{#1}[\CjSpace{#2}[#3]]} {\Norm{#1}[\CjSpace{#2}[#3][#4]]}  }  }

\NewDocumentCommand{\CNorm}{m m o}{\IfNoValueTF{#3}{\Norm{#1}[\CSpace{#2}]}{\Norm{#1}[\CSpace{#2}[#3]]}}

\NewDocumentCommand{\BCNorm}{m m}{\BNorm{#1}[\CSpace{#2}]}

\NewDocumentCommand{\CXjNorm}{m m m o}{\Norm{#1}[
\IfNoValueTF{#4}
{\CXjSpace{#2}{#3}}
{\CXjSpace{#2}{#3}[#4]}
]}

\NewDocumentCommand{\BCXjNorm}{m m m o}{\BNorm{#1}[
\IfNoValueTF{#4}
{\CXjSpace{#2}{#3}}
{\CXjSpace{#2}{#3}[#4]}
]}

\NewDocumentCommand{\LpNorm}{m m o o}{
\Norm{#2}[L^{#1}
\IfNoValueTF{#3}{}{
(#3
\IfNoValueTF{#4}{}{;#4}
)
}
]
}

\NewDocumentCommand{\BCjNorm}{m m o}{ \IfNoValueTF{#3}{ \BNorm{#1}[C^{#2}]} { \BNorm{#1}[C^{#2}(#3)]  }  }

\NewDocumentCommand{\HNorm}{m m m o o}{ \IfNoValueTF{#4}{ \Norm{#1}[\HSpace{#2}{#3}]} {
\IfNoValueTF{#5}
{\Norm{#1}[\HSpace{#2}{#3}[#4]]}
{\Norm{#1}[\HSpace{#2}{#3}[#4][#5]] }
}  }

\NewDocumentCommand{\HXNorm}{m m m m o}{ \IfNoValueTF{#5}{ \Norm{#1}[\HXSpace{#2}{#3}{#4}]} {
{\Norm{#1}[\HXSpace{#2}{#3}{#4}[#5]]}
}  }

\NewDocumentCommand{\BHXNorm}{m m m m o}{ \IfNoValueTF{#5}{ \BNorm{#1}[\HXSpace{#2}{#3}{#4}]} {
{\BNorm{#1}[\HXSpace{#2}{#3}{#4}[#5]]}
}  }

\NewDocumentCommand{\ZygNorm}{m m o o}{ \IfNoValueTF{#3}{ \Norm{#1}[\ZygSpace{#2}]} {
\IfNoValueTF{#4}
{\Norm{#1}[\ZygSpace{#2}[#3]]}
{\Norm{#1}[\ZygSpace{#2}[#3][#4]]}
}  }

\NewDocumentCommand{\BZygNorm}{m m o o}{ \IfNoValueTF{#3}{ \Norm{#1}[\ZygSpace{#2}]} {
\IfNoValueTF{#4}
{\BNorm{#1}[\ZygSpace{#2}[#3]]}
{\BNorm{#1}[\ZygSpace{#2}[#3][#4]]}
}  }

\NewDocumentCommand{\ZygXNorm}{m m m o}{\Norm{#1}[
\IfNoValueTF{#4}
{\ZygXSpace{#2}{#3}}
{\ZygXSpace{#2}{#3}[#4]}
]}

\NewDocumentCommand{\BZygXNorm}{m m m o}{\BNorm{#1}[
\IfNoValueTF{#4}
{\ZygXSpace{#2}{#3}}
{\ZygXSpace{#2}{#3}[#4]}
]}

\NewDocumentCommand{\diff}{o m}{\IfNoValueTF{#1}{\frac{\partial}{\partial #2}}{\frac{\partial^{#1}}{\partial #2^{#1}} }}

\NewDocumentCommand{\dt}{o}{\IfNoValueTF{#1}{\diff{t}}{\diff[#1]{t} }}

\NewDocumentCommand{\Zygad}{m}{\{ #1\}}

\NewDocumentCommand{\Zygsonu}{}{[s_0;\nu]}
\NewDocumentCommand{\Zygomeganu}{}{[\omega;\nu]}

\NewDocumentCommand{\Had}{m}{\langle #1\rangle}

\NewDocumentCommand{\BanachSpace}{}{\mathscr{V}}
\NewDocumentCommand{\BanachAlgebra}{}{\mathscr{Y}}

\NewDocumentCommand{\Field}{}{\mathbb{F}}

\NewDocumentCommand{\Real}{}{\mathrm{Re}}
\NewDocumentCommand{\Imag}{}{\mathrm{Im}}

\NewDocumentCommand{\VectorSpace}{m}{\mathscr{#1}}
\NewDocumentCommand{\VVS}{}{\VectorSpace{V}}
\NewDocumentCommand{\ZVS}{}{\VectorSpace{Z}}
\NewDocumentCommand{\WVS}{}{\VectorSpace{W}}
\NewDocumentCommand{\LVSh}{}{\widehat{\VectorSpace{L}}}
\NewDocumentCommand{\LVSb}{o}{\IfNoValueTF{#1}{\overline{\VectorSpace{L}}}{\overline{\VectorSpace{L}_{#1}}}}
\NewDocumentCommand{\LVS}{o}{\IfNoValueTF{#1}{\VectorSpace{L}}{\VectorSpace{L}_{#1}}}
\NewDocumentCommand{\XVS}{o}{\IfNoValueTF{#1}{\VectorSpace{X}}{\VectorSpace{X}_{#1}}}
\NewDocumentCommand{\XVSb}{o}{\IfNoValueTF{#1}{\overline{\VectorSpace{X}}}{\overline{\VectorSpace{X}_{#1}}}}

\NewDocumentCommand{\Lb}{o}{\IfNoValueTF{#1}{\overline{L}}{\overline{L_{#1}}}}
\NewDocumentCommand{\wb}{o}{\IfNoValueTF{#1}{\overline{w}}{\overline{w}_{#1}}}
\NewDocumentCommand{\zb}{o}{\overline{z}\IfNoValueTF{#1}{}{_{#1}}}
\NewDocumentCommand{\Hb}{o}{\overline{H\IfNoValueTF{#1}{}{_{#1}}}}
\NewDocumentCommand{\Rb}{o}{\overline{R\IfNoValueTF{#1}{}{_{#1}}}}
\NewDocumentCommand{\Zb}{o}{\IfNoValueTF{#1}{\overline{Z}}{\overline{Z_{#1}}}}
\NewDocumentCommand{\Zbdelta}{m o}{\IfNoValueTF{#1}{\overline{Z^{#1}}}{\overline{Z^{#1}_{#2}}}}
\NewDocumentCommand{\Zhb}{o}{\IfNoValueTF{#1}{\overline{\Zh}}{\overline{\Zh_{#1}}}}

\NewDocumentCommand{\dbar}{}{\overline{\partial}}
\NewDocumentCommand{\etah}{}{\hat{\eta}}


\NewDocumentCommand{\SSFunctionSpacesSection}{}{Section 2}
\NewDocumentCommand{\SSStrangeZygSpace}{}{Remark 2.1}
\NewDocumentCommand{\SSBeyondManifold}{}{Section 2.2.1}
\NewDocumentCommand{\SSNormsAreInv}{}{Proposition 2.3}
\NewDocumentCommand{\SSDefineVectDeriv}{}{Remark 2.4}

\NewDocumentCommand{\SSSectionMoreOnAssumptions}{}{Section 4.1}
\NewDocumentCommand{\SSMainResult}{}{Theorem 4.7}
\NewDocumentCommand{\SSLemmaMoreOnAssump}{}{Proposition 4.14}

\NewDocumentCommand{\SSDivideWedge}{}{Section 5}
\NewDocumentCommand{\SSDerivWedge}{}{Lemma 5.1}

\NewDocumentCommand{\SSDensities}{}{Section 6}
\NewDocumentCommand{\SSDensitiesTheorem}{}{Theorem 6.5}
\NewDocumentCommand{\SSDensityCor}{}{Corollary 6.6}

\NewDocumentCommand{\SSScaling}{}{Section 7}
\NewDocumentCommand{\SSNSW}{}{Section 7.1}
\NewDocumentCommand{\SSHormandersCondition}{}{Section 7.1.1}
\NewDocumentCommand{\SSGenSubR}{}{Section 7.3}
\NewDocumentCommand{\SSGenSubResult}{}{Theorem 7.6}

\NewDocumentCommand{\SSCompareFunctionSpaces}{}{Lemma 8.1}
\NewDocumentCommand{\SSZygIsAlgebra}{}{Proposition 8.3}
\NewDocumentCommand{\SSBiggerNormMap}{}{Proposition 8.6}
\NewDocumentCommand{\SSCompareEuclidNorms}{}{Proposition 8.12}

\NewDocumentCommand{\SSDeriveODE}{}{Proposition 9.1}
\NewDocumentCommand{\SSExistODE}{}{Proposition 9.4}
\NewDocumentCommand{\SSExistXiOne}{}{Lemma 9.23}
\NewDocumentCommand{\SSDifferentOneAdmis}{}{Proposition 9.26}
\NewDocumentCommand{\SSCXjNormWedgeQuotient}{}{Lemma 9.32}
\NewDocumentCommand{\SSExistXiTwo}{}{Lemma 9.35}
\NewDocumentCommand{\SSComputefjzero}{}{Lemma 9.38}
\NewDocumentCommand{\SSSectionDensities}{}{Section 9.4}

\NewDocumentCommand{\SSProofInjectiveImmersion}{}{Appendix A}
\NewDocumentCommand{\SSFinerTopology}{}{Lemma A.1}

\newtheorem{thm}{Theorem}[section]
\newtheorem{cor}[thm]{Corollary}
\newtheorem{prop}[thm]{Proposition}
\newtheorem{lemma}[thm]{Lemma}
\newtheorem{conj}[thm]{Conjecture}
\newtheorem{prob}[thm]{Problem}

\theoremstyle{remark}
\newtheorem{rmk}[thm]{Remark}

\theoremstyle{remark}
\newtheorem{question}[thm]{Question}

\theoremstyle{definition}
\newtheorem{defn}[thm]{Definition}

\theoremstyle{definition}
\newtheorem{assumption}[thm]{Assumption}

\theoremstyle{remark}
\newtheorem{example}[thm]{Example}

\numberwithin{equation}{section}

\title{Sharp Regularity for the Integrability of Elliptic Structures}
\author{Brian Street\footnote{The author was partially supported by National Science Foundation Grant Nos.\ 1401671 and 1764265.}}
\date{}

\maketitle

\begin{abstract}
As part of his celebrated Complex Frobenius Theorem, Nirenberg showed that given a smooth elliptic structure (on a smooth manifold), the manifold is locally diffeomorphic to an open subset of $\mathbb{R}^r\times \mathbb{C}^n$ (for some $r$ and $n$) in such a way that the structure is locally the span of $\frac{\partial}{\partial t_1},\ldots, \frac{\partial}{\partial t_r},\frac{\partial}{\partial \overline{z}_1},\ldots, \frac{\partial}{\partial \overline{z}_n}$; where $\mathbb{R}^r\times \mathbb{C}^n$ has coordinates $(t_1,\ldots, t_r, z_1,\ldots, z_n)$.  In this paper, we give optimal regularity for the coordinate charts which achieve this realization.  Namely, if the manifold has Zygmund regularity of order $s+2$ and the structure has Zygmund regularity of order $s+1$ (for some $s>0$), then the coordinate chart may be taken to have Zygmund regularity of order $s+2$.  We do this by generalizing Malgrange's proof of the Newlander-Nirenberg Theorem to this setting.
\end{abstract}

\section{Introduction}
Fix $s\in (0,\infty]\cup \{\omega\}$ and let $M$ be a $\ZygSpace{s+2}$ manifold, where $\ZygSpace{s}$ denotes the Zygmund space\footnote{For non-integer exponents, the Zygmund space agrees with the classical H\"older space.  More precisely,
 for $m\in \N$, $a\in (0,1)$, the Zygmund space $\ZygSpace{m+a}$ is locally the same as the H\"older space $\HSpace{m}{a}$--see \cref{Rmk::FuncEudlid::ZygAndHolderTheSame}.  However,
 for $a\in \{0,1\}$, these spaces differ:  $\HSpace{m+1}{0}\subsetneq\HSpace{m}{1}\subsetneq \ZygSpace{m+1}$.}
 of order $s$, $\ZygSpace{\infty}$ denotes $C^\infty$, $\ZygSpace{\omega}$ denotes
the space of real analytic functions, and we use the convention $\infty+2=\infty+1=\infty$ and $\omega+2=\omega+1=\omega$.
Let $\LVS$ be a $\ZygSpace{s+1}$ complex elliptic structure on $M$; in particular, $\LVS$ is a complex sub-bundle of $\C TM$, is formally integrable, and
$\LVS$ satisfies $\LVS[\zeta]+\LVSb[\zeta]=\C T_\zeta M$, $\forall \zeta\in M$.  See \cref{Section::FunctionSpaces,Section::Bundles} for the full definitions.
Set $n+r:=\dim \LVS[\zeta]$ and $r:=\dim \LVS[\zeta]\cap \LVSb[\zeta]$ (by hypothesis, $n$ and $r$ do not depend on $\zeta$; see \cref{Section::Bundles}).
For a Banach space $\BanachSpace$, let
 $B_{\BanachSpace}(\delta)$ denote the ball of radius $\delta>0$, centered at $0$, in $\BanachSpace$.
The main theorem of this paper is:

\begin{thm}\label{Thm::Intro::MainThm}
For all $\zeta\in M$, there exists an open neighborhood $V\subseteq M$ of $\zeta$ and a $\ZygSpacediff{s+2}$ diffeomorphism
$\Phi:B_{\R^r\times \C^n}(1)\rightarrow V$ such that $\forall (t,z)\in B_{\R^r\times \C^n}(1)$:
\begin{equation*}
\Span_{\C} \mleft\{
\mleft(\Phi_{*}\diff{t_k}\mright) (\Phi(t,z)), \mleft(\Phi_{*} \diff{\zb[j]}\mright)(\Phi(t,z))  : 1\leq k\leq r, 1\leq j\leq n
\mright\}
=\LVS[\Phi(t,z)].
\end{equation*}
Here, we have given $\R^r\times \C^n$ coordinates $(t_1,\ldots, t_r,z_1,\ldots, z_n)$.
\end{thm}
See \cref{Thm::EMfld::MainThm} for a more abstract version of \cref{Thm::Intro::MainThm}.

When $s=\omega$, \cref{Thm::Intro::MainThm}  is classical.
When $s=\infty$, \cref{Thm::Intro::MainThm} is a result of Nirenberg \cite{NirenbergAComplexFrobeniusTheorem}; and the goal of this paper is to achieve the sharp regularity for
$\Phi$ in terms of the regularity of $M$ and $\LVS$.  When $r=0$, $\LVS$ is a complex structure, and \cref{Thm::Intro::MainThm} was proved by Malgrange \cite{MalgrangeSurLIntegbrabilite}--in this case,
the result gives the sharp regularity for the Newlander-Nirenberg Theorem \cite{NewlanderNirenbergComplexAnalyticCoordiantesInAlmostComplexManifolds}.\footnote{Another proof of the case $r=0$
was later given by Webster \cite{WebsterANewProofOfNN}.  Both \cite{MalgrangeSurLIntegbrabilite} and \cite{WebsterANewProofOfNN} state results for H\"older spaces and avoid integer exponents.  As is well-known, and described in 
the case $r=0$ of \cref{Thm::Intro::MainThm}, the results extend to integer exponents by using Zygmund spaces.}
One standard way to prove results like \cref{Thm::Intro::MainThm} for $r>0$ is to reduce the claim to the setting of $r=0$, and apply the Newlander-Nirenberg Theorem, where sharp regularity is
known due to Malgrange's result.  Unfortunately, this reduction loses a derivative (i.e., only proves \cref{Thm::Intro::MainThm} with $\Phi$ a $\ZygSpacediff{s+1}$ diffeomorphism).
Instead, we proceed by adapting Malgrange's proof to directly prove \cref{Thm::Intro::MainThm}.  

This paper is outlined as follows. In \cref{Section::FunctionSpaces} we introduce the basic function spaces we need. In \cref{Section::Bundles} we give all the relevant (standard) definitions for bundles and structures.
In \cref{Section::EMflds} we define a category of manifolds in which our results are naturally stated:  this is the category of manifolds endowed with an ``elliptic'' structure.  This category
contains both real and complex manifolds as \textit{full} sub-categories.
We use this  to state a more abstract version of our main result (\cref{Thm::EMfld::MainThm}).
In \cref{Section::Nirenbeg} we state and prove the main technical result of this paper.  As discussed in \cref{Section::Intro::MainMotivation},
with future applications in mind we keep careful track of what all the constants in \cref{Section::Nirenbeg} depend on.   This is the heart of this paper.  In \cref{Section::Proofs} we prove the main result;
i.e., \cref{Thm::Intro::MainThm} and more generally \cref{Thm::EMfld::MainThm}. 

	\subsection{Some Further Comments}
Results like \cref{Thm::Intro::MainThm} (in the smooth case, $s=\infty$) were introduced by Nirenberg to prove his more general Complex Frobenius Theorem \cite{NirenbergAComplexFrobeniusTheorem}.
There, one starts with a $C^\infty$ formally integrable  structure $\LVS$ on $M$ (see \cref{Section::Bundles}).  The classical (real) Frobenius Theorem applies to the essentially real sub-bundle $\LVS+\LVSb$ to foliate the ambient manifold into leaves, and $\LVS$ is an elliptic structure on each leaf.
Then one can apply a result like \cref{Thm::Intro::MainThm}\footnote{One needs a version of \cref{Thm::Intro::MainThm} with a parameter, which can be achieved with a similar proof in the smooth case.} to each leaf.
In this way, one can achieve a result which has the real Frobenius theorem, the Newlander-Nirenberg Theorem, and the integrability of ellipic structures as special cases (at least in the smooth setting).

In \cref{Thm::Intro::MainThm}, the coordinate chart $\Phi$ is one derivative better than the bundle $\LVS$ (i.e., $\Phi$ is $\ZygSpace{s+2}$, while $\LVS$ is $\ZygSpace{s+1}$).
This is the best one can hope for, 
since the hypotheses of \cref{Thm::Intro::MainThm} are invariant under $\ZygSpace{s+2}$ diffeomorphisms.
However, even in the classical real Frobenius theorem, one cannot obtain appropriate coordinate charts which are one derivative better than the underlying vector fields:  see \cite[Example 4.5]{GongForbeniusNirenbergTheorem}
for a very simple example involving only one vector field.  Thus, we restrict attention to the setting of \cref{Thm::Intro::MainThm} (which does not involve any kind of foliation) because this seems to be a natural generality in which we can achieve this level of regularity.

As mentioned above, one common way of proving results like \cref{Thm::Intro::MainThm} is to reduce them to the Newlander-Nirenberg theorem; though this reduction unnecessarily costs a derivative.
One can do this without losing a derivative by assuming the existence of some sufficiently regular vector fields which commute.  This is the approach taken in \cite{HillTaylorTheComplexFrobeniusTheoremForRough}
where results are proved for Lipschitz bundles.  With our approach, we do not need to assume the existence of such vector fields (and in fact, their existence is a consequence of our result).
It is possible that the methods of this paper combined with the methods of \cite{HillTaylorTheComplexFrobeniusTheoremForRough} could be used to prove results like the ones in that paper,
without assuming the existence of such commuting vector fields.

	\subsection{A Main Motivation}\label{Section::Intro::MainMotivation}
A simple consequence of \cref{Thm::Intro::MainThm} is the following:
\begin{cor}\label{Cor::IntroMotivation}
Fix $s\in (0,\infty]\cup\{\omega\}$ and
let $M$ be a $\ZygSpace{s+2}$ manifold.
Let $L_1,\ldots, L_m$ be  $\ZygSpace{s+1}$ complex vector fields on $M$ and $X_1,\ldots, X_q$ be $\ZygSpace{s+1}$ real vector fields on $M$.  Suppose:
\begin{itemize}
\item For all $\zeta\in M$, 
\begin{equation*}
	\Span_\C \mleft\{ L_1(\zeta),\ldots,L_m(\zeta), \Lb[1](\zeta),\ldots, \Lb[m](\zeta), X_1(\zeta),\ldots, X_q(\zeta) \mright\} = \C T_{\zeta}M.
\end{equation*}
\item For all $\zeta\in M$, $1\leq j,j_1,j_2\leq m$, $1\leq k, k_1,k_2\leq q$,
\begin{equation*}
	[L_{j_1}, L_{j_2}](\zeta), [L_j, X_k](\zeta), [X_{k_1},X_{k_2}](\zeta)\in \Span_{\C} \mleft\{ L_1(\zeta),\ldots, L_m(\zeta), X_1(\zeta),\ldots, X_q(\zeta) \mright\}.
\end{equation*}
\item For all $\zeta\in M$,
\begin{equation*}
\begin{split}
	&\Span_{\C}\mleft\{L_1(\zeta),\ldots, L_m(\zeta), X_1(\zeta),\ldots, X_q(\zeta)\mright\}\bigcap \Span_{\C}\mleft\{\Lb[1](\zeta),\ldots, \Lb[m](\zeta), X_1(\zeta),\ldots, X_q(\zeta)\mright\} 
	\\&=\Span_{\C} \mleft\{ X_1(\zeta),\ldots, X_q(\zeta)\mright\}.
\end{split}
\end{equation*}
\item The map $\zeta\mapsto \dim \Span_{\C} \mleft\{L_1(\zeta),\ldots, L_m(\zeta), X_1(\zeta),\ldots, X_q(\zeta)\mright\}$ is constant in $\zeta$.
\end{itemize}
Set $n+r:=\dim \Span_{\C} \mleft\{L_1(\zeta),\ldots, L_m(\zeta), X_1(\zeta),\ldots, X_q(\zeta)\mright\}$ (which does not depend on $\zeta$ by hypothesis)
and set $r:= \dim\Span_{\C}  \mleft\{ X_1(\zeta),\ldots, X_q(\zeta)\mright\}$ (which also does not depend on $\zeta$--see \cref{Lemma::Bundles::DimXConst}).
Then, $\forall \zeta \in M$, there exists a neighborhood $V$ of $\zeta$ and a $\ZygSpace{s+2}$ diffeomorphism $\Phi:B_{\R^r\times \C^n}(1)\rightarrow V$ such that
$\forall \xi\in B_{\R^r\times \C^n}(1)$
\begin{equation*}
	\Span_{\C} \mleft\{ \Phi^{*}L_1(\xi),\ldots, \Phi^{*} L_m(\xi), \Phi^{*} X_1(\xi),\ldots, \Phi^{*} X_q(\xi)  \mright\} = \Span_{\C} \mleft\{ \diff{t_1},\ldots, \diff{t_r},\diff{\zb[1]},\ldots, \diff{\zb[n]} \mright\},
\end{equation*}
where we have given $\R^r\times \C^n$ coordinates $(t_1,\ldots, t_r,z_1,\ldots, z_n)$.
\end{cor}
\begin{proof}
Apply \cref{Thm::Intro::MainThm} (see, also, \cref{Thm::EMfld::MainThm}) to the bundle 
$$\LVS[\zeta]:=\Span_{\C} \mleft\{L_1(\zeta),\ldots, L_m(\zeta), X_1(\zeta),\ldots, X_q(\zeta)\mright\};$$
$\LVS$ is easily seen to be a $\ZygSpace{s+1}$ elliptic structure on $M$.  See \cref{Section::Bundles} for this terminology.
\end{proof}

We now consider a harder question.  Let $M$ be a $C^2$ manifold, and let $L_1,\ldots, L_m$ be $C^1$ complex vector fields on $M$ and $X_1,\ldots, X_q$ be $C^1$ real vector fields on $M$.

\begin{question}\label{Question::IntroMotivation}
Fix $\zeta\in M$ and $s\in (0,\infty]\cup\{\omega\}$.  When is there a neighborhood $V$ of $\zeta$ and a $C^2$ diffeomorphism $\Phi:B_{\R^r\times \C^n} (1)\rightarrow V$ such that
$\Phi^{*}L_1,\ldots, \Phi^{*}L_m, \Phi^{*}X_1,\ldots, \Phi^{*}X_q$ are $\ZygSpace{s+1}$ vector fields on $B_{\R^r\times \C^n} (1)$ and
\begin{equation*}
	\Span_{\C} \mleft\{ \Phi^{*}L_1(\xi),\ldots, \Phi^{*} L_m(\xi), \Phi^{*} X_1(\xi),\ldots, \Phi^{*} X_q(\xi)  \mright\} = \Span_{\C} \mleft\{ \diff{t_1},\ldots, \diff{t_r},\diff{\zb[1]},\ldots, \diff{\zb[n]} \mright\}.
\end{equation*}
\end{question}

When the vector fields are already known to be $\ZygSpace{s+1}$, \cref{Question::IntroMotivation} is answered by \cref{Cor::IntroMotivation}.  But \cref{Question::IntroMotivation} asks more:  it asks
when one can pick the coordinate system $\Phi$ so that the vector fields are \textit{more regular} than they were originally.  It is not always possible to do this, but in a companion paper \cite{StreetSubH}
we give necessary and sufficient conditions under which it is possible (for $s\in (1,\infty]\cup\{\omega\}$).  By answering this question in a quantitative way we provide scaling maps
adapted to sub-Riemannian geometries, which strengthen and generalize previous results in the case $m=0$ (i.e., all the vector fields are real) by Nagel, Stein, and Wainger \cite{NagelSteinWaingerBallsAndMetrics}, Tao and Wright \cite{TaoWrightLpImproving},
and the author \cite{StreetMultiparameterCCBalls}.  The case when $m=0$  was covered in the series \cite{StovallStreetI,StovallStreetII,StovallStreetIII}.

The case when $q=0$ of \cref{Question::IntroMotivation} is particularly interesting.  In this case, the coordinate chart $\Phi$ can be thought of as a holomorphic coordinate system.  When one turns to the quantitative theory
discussed above, this allows us to create holomorphic analogs of the sub-Riemannian scaling maps introduced by Nagel, Stein, and Wainger \cite{NagelSteinWaingerBallsAndMetrics}.  In this way we can study
sub-Riemannian geometries on complex manifolds, which are adapted to the complex structure.  We call these sub-Hermitian geometries.

The main technical result of this paper (\cref{Thm::Nirenberg::MainThm}) is a key step in developing the theory in  the companion work \cite{StreetSubH}.  Because of this, it is important for our future applications
that we keep track of the dependance various constants in \cref{Thm::Nirenberg::MainThm}.  For this purpose we introduce several function spaces and definitions that we would not otherwise have to.
This makes the statement of \cref{Thm::Nirenberg::MainThm} a bit more involved than it would have to be to prove the main results of this paper; though, other than some bookkeeping, the proof is no more difficult.
Because of its quantitative nature, it is possible \cref{Thm::Nirenberg::MainThm} will be more useful in future applications than the ``main results'' of this paper.

\section{Function Spaces}\label{Section::FunctionSpaces}
Let $\Omega\subset \R^n$ be a connected, open set (we will almost always be considering the case when $\Omega$ is a ball in $\R^n$).
We have the following classical spaces of functions on $\Omega$:
\begin{equation*}
\CSpace{\Omega}=\CjSpace{0}[\Omega]:=\{f:\Omega\rightarrow \C \:\big|\: f\text{ is continuous and bounded}\},\quad \CNorm{f}{\Omega}=\CjNorm{f}{0}[\Omega]:=\sup_{x\in \Omega}|f(x)|.
\end{equation*}
For $m\in \N$, (we use the convention $0\in \N$)
\begin{equation*}
\CjSpace{m}[\Omega]:=\{f\in \CSpace{\Omega}\: \big|\: \partial_x^{\alpha}f \in \CSpace{\Omega}, \forall |\alpha|\leq m\}, \quad \CjNorm{f}{m}[\Omega]:=\sum_{|\alpha|\leq m} \CNorm{\partial_x^{\alpha} f}{\Omega}.
\end{equation*}
Next we define the classical H\"older spaces.  For $s\in [0,1]$,
\begin{equation}\label{Eqn::FSEuclid::DefnHolder}
\HNorm{f}{0}{s}[\Omega]:=\CNorm{f}{\Omega} + \sup_{\substack{x,y\in \Omega \\ x\ne y}} |x-y|^{-s} |f(x)-f(y)|, \quad \HSpace{0}{s}[\Omega]:=\{f\in \CSpace{\Omega} : \HNorm{f}{0}{s}[\Omega]<\infty\}.
\end{equation}
For $m\in \N$, $s\in [0,1]$,
\begin{equation*}
\HNorm{f}{m}{s}:=\sum_{|\alpha|\leq m} \HNorm{\partial_x^{\alpha} f}{0}{s}, \quad \HSpace{m}{s}[\Omega]:=\{f\in \CjSpace{m}[\Omega] : \HNorm{f}{m}{s}[\Omega]<\infty\}.
\end{equation*}
Next, we turn to the classical Zygmund spaces.  Given $h\in \R^n$ define $\Omega_h:=\{x\in \R^n : x,x+h,x+2h\in \Omega\}$.
For $s\in (0,1]$ set
\begin{equation*}
\ZygNorm{f}{s}[\Omega]:=\HNorm{f}{0}{s/2}[\Omega]+ \sup_{\substack{0\ne h\in \R^n \\ x\in \Omega_h}} |h|^{-s} |f(x+2h)-2f(x+h)+f(x)|,
\end{equation*}
\begin{equation*}
\ZygSpace{s}[\Omega]:=\{f\in \CSpace{\Omega} : \ZygNorm{f}{s}[\Omega]<\infty\}.
\end{equation*}
For $m\in \N$, $s\in (0,1]$, set
\begin{equation*}
\ZygNorm{f}{m+s}[\Omega]:=\sum_{|\alpha|\leq m} \ZygNorm{\partial_x^{\alpha} f}{s}[\Omega], \quad \ZygSpace{m+s}[\Omega]:=\{ f\in \CjSpace{m}[\Omega] : \ZygNorm{f}{m+s}[\Omega]<\infty\}.
\end{equation*}
We set
\begin{equation*}
\ZygSpace{\infty}[\Omega] := \bigcap_{s>0}\ZygSpace{s}[\Omega],\quad  \CjSpace{\infty}[\Omega] := \bigcap_{m\in \N} \CjSpace{m}[\Omega].
\end{equation*}
It is straightforward to verify that for a ball $B$, $\ZygSpace{\infty}[B]=\CjSpace{\infty}[B]$.

Finally, we let $\ZygSpace{\omega}[\Omega]$ be the space of real analytic functions on $\Omega$.


If $\BanachSpace$ is a Banach space, we define the same spaces taking values in $\BanachSpace$ in the obvious way, and denote these spaces
by $\CSpace{\Omega}[\BanachSpace]$, $\CjSpace{m}[\Omega][\BanachSpace]$, $\HSpace{m}{s}[\Omega][\BanachSpace]$, $\ZygSpace{s}[\Omega][\BanachSpace]$,
and $\ZygSpace{\omega}[\Omega][\BanachSpace]$.
 Given a complex vector field $X$ on $\Omega$, we identify
 $X=\sum_{j=1}^n a_j(x) \frac{\partial}{\partial x_j}$ with the function $(a_1,\ldots, a_n):\Omega\rightarrow \C^n$.  It therefore makes sense to consider quantities
 like $\ZygNorm{X}{s}[\Omega][\C^n]$.
 When $\BanachSpace$ is clear from context, we sometimes suppress it and write, e.g., $\ZygNorm{f}{s}[\Omega]$ instead of $\ZygNorm{f}{s}[\Omega][\BanachSpace]$ for readability considerations.

 \begin{rmk}
 The term $\HNorm{f}{0}{s/2}$ in the definition of $\ZygNorm{f}{s}$ is somewhat unusual, and is usually replaced by $\CjNorm{f}{0}$.  However,
 if $\Omega$ is a bounded Lipschitz domain
 these two choices yield equivalent norms:
 this is a simple consequence of \cite[Theorem 1.118 (i)]{TriebelTheoryOfFunctionSpacesIII}.  The definition we have chosen is somewhat more convenient to work with.
 \end{rmk}

 \begin{defn}
 For $s\in (0,\infty]\cup \{\omega\}$, we say $f\in \ZygSpaceloc{s}[\Omega]$ if $\forall x\in \Omega$, there exists an open ball $B\subseteq \Omega$, centered at $x$, with $f\big|_{B}\in \ZygSpace{s}[B]$.
 \end{defn}

  \begin{rmk}\label{Rmk::FuncEudlid::ZygAndHolderTheSame}
 If $\Omega$ is a bounded Lipschitz domain, $m\in \N$, $s\in (0,1)$, the spaces $\HSpace{m}{s}[\Omega]$ and $\ZygSpace{m+s}[\Omega]$ are the same--see \cite[Theorem 1.118 (i)]{TriebelTheoryOfFunctionSpacesIII}.
 However, if $s\in \{0,1\}$, these spaces differ.
 As a consequence for \textit{any} open set $\Omega\subseteq \R^n$, for $m\in \N$, $s\in (0,1)$, we have $\ZygSpaceloc{m+s}[\Omega]$ equals the space of functions which are locally
 in $\HSpace{m}{s}$.
 \end{rmk}

 \begin{rmk}\label{Rmk::FuncEudlid::TalkAboutLoc}
 $\ZygSpace{\omega}$ and $\ZygSpaceloc{\omega}$ denote the same thing.  However, for $s\in (0,\infty]$, $\ZygSpace{s}$ and $\ZygSpaceloc{s}$ are not the same.
 Since for any ball $B$ we have $\ZygSpace{\infty}[B]=\CjSpace{\infty}[B]$, the space $\ZygSpacemap{\infty}[\Omega]$ corresponds with the usual space of smooth functions on $\Omega$.
 \end{rmk} 
	
	\subsection{Manifolds}
In this paper we use $\ZygSpace{s}$ manifolds; the definition is exactly what one would expect, though a little care is needed due to the subtleties of Zygmund spaces.\footnote{For example, one must define the Zygmund maps in the right way to ensure that the composition of two Zygmund maps is again a Zygmund map.}  We present the relevant (standard) definitions here.

\begin{defn}
Let $U_1\subseteq \R^{n_1}$ and $U_2\subseteq \R^{n_2}$ be open sets.  For $s\in (0,\infty]\cup \{\omega\}$, we say $f:U_1\rightarrow U_2$ is a $\ZygSpacemap{s}$ map if $f\in \ZygSpaceloc{s}[U_1][\R^{n_2}]$.
\end{defn}

\begin{lemma}\label{Lemma::FuncManifold::ComposeEuclid}
Let $U_1\subseteq \R^{n_1}$, $U_2\subseteq \R^{n_2}$, and $U_3\subseteq \R^{n_3}$ be open sets.
For $s_1\in (0,\infty]\cup\{\omega\}$, $s_2 \geq s_1$, $s_2\in (1,\infty]\cup \{\omega\}$, if $f_1:U_1\rightarrow U_2$ is a $\ZygSpacemap{s_1}$ map and $f_2:U_2\rightarrow U_3$ is a $\ZygSpacemap{s_2}$ map, then $f_2\circ f_1:U_1\rightarrow U_3$ is a $\ZygSpacemap{s_1}$ map.
\end{lemma}
\begin{proof}
When $s_1\in \{\infty,\omega\}$, the result is obvious.  For $s_1\in (0,\infty)$,
because the notion of being a $\ZygSpacemap{s}$ map is local, it suffices to check $f_1\circ f_2$ is in $\ZygSpace{s_1}$ on sufficiently small balls.  This is described in \cref{Lemma::FuncSpaceRev::Composition}, below.
\end{proof}

\begin{lemma}\label{Lemma::FuncManifold::InverseEuclid}
For $s\in (1,\infty]\cup\{\omega\}$ if $f:U_1\rightarrow U_2$ is a $\ZygSpacemap{s}$ map which is also a $C^1$ diffeomorphism, then $f^{-1}: U_2\rightarrow U_1$ is a $\ZygSpacemap{s}$ map.
\end{lemma}
\begin{proof}
For $s\in \{\infty,\omega\}$ this is standard.  For $s\in (1,\infty)$ it suffices to check $f^{-1}$ is in $\ZygSpace{s}$ when restricted to sufficiently small balls, because the result is local.  This is described in \cref{Lemma::FuncSpaceRev::Inverse}, below.
\end{proof}

\begin{defn}\label{Defn::FuncMfld::Atlas}
Fix $s\in (1,\infty]\cup\{\omega\}$ and let $M$ be a topological space.  We say $\{(\phi_\alpha, V_\alpha) : \alpha\in \sI\}$ (where $\sI$ is some index set) is a $\ZygSpace{s}$ atlas of dimension $n$ 
if $\{V_\alpha:\alpha\in \sI\}$ is an open cover for $M$, $\phi_{\alpha}:V_\alpha\rightarrow U_\alpha$ is a homeomorphism where $U_\alpha\subseteq \R^n$ is open,
and $\phi_{\beta}\circ \phi_{\alpha}^{-1} : \phi_{\alpha} (V_\beta\cap V_\alpha)\rightarrow U_\beta$ is a $\ZygSpacemap{s}$ map.
\end{defn}

\begin{defn}
For $s\in (1,\infty]\cup\{\omega\}$ a $\ZygSpace{s}$ manifold of dimension $n$ is a paracompact\footnote{We do not use paracompactness in this paper, so the reader who wishes to define manifolds without requiring paracompactness is free to do this throughout this paper.}
topological space $M$ endowed with a $\ZygSpace{s}$ atlas of dimension $n$.
\end{defn}

\begin{rmk}\label{Rmk::FuncMfld::OpenSetsAreManifolds}
Let $U\subseteq \R^n$ be an open set.  $U$ is naturally a $\ZygSpace{\omega}$ manifold of dimension $n$; where we take the atlas consisting of a single coordinate chart (namely, the identity map $U\rightarrow U$).
We henceforth give open sets this manifold structure.
\end{rmk}

\begin{rmk}
In particular, a $\ZygSpace{s}$ manifold is a $C^m$ manifold, for any $m<s$.  In light of \cref{Rmk::FuncEudlid::TalkAboutLoc}, $\ZygSpace{\infty}$ and $\CjSpace{\infty}$ manifolds are the same.
\end{rmk}

\begin{defn}
For $s\in (0,\infty]\cup\{\omega\}$ and let $M$ and $N$ be $\ZygSpace{s+1}$ manifolds with $\ZygSpace{s+1}$ atlases $\{(\phi_{\alpha}, V_\alpha)\}$ and $\{(\psi_\beta, W_\beta)\}$, respectively.
We say $f:M\rightarrow N$ is a $\ZygSpacemap{s+1}$ map if $\psi_{\beta}\circ f\circ \phi_{\alpha}^{-1}$ is a $\ZygSpacemap{s+1}$ map, $\forall \alpha,\beta$.
\end{defn}

\begin{lemma}\label{Lemma::FuncMfld::CompositionOfMaps}
For $s\in (0,\infty]\cup \{\omega\}$, suppose $M_1$, $M_2$, and $M_3$ are $\ZygSpace{s+1}$ manifolds and $f_1:M_1\rightarrow M_2$ and $f_2:M_2\rightarrow M_3$ are $\ZygSpacemap{s+1}$ maps.
Then, $f_2\circ f_1:M_1\rightarrow M_3$ is a $\ZygSpacemap{s+1}$ map.
\end{lemma}
\begin{proof}
This follows from \cref{Lemma::FuncManifold::ComposeEuclid}.
\end{proof}

\begin{lemma}\label{Lemma::FuncManfiold::InverseMap}
Suppose $s\in (0,\infty]\cup\{\omega\}$, $M_1$ and $M_2$ are $\ZygSpace{s+1}$ manifolds, and $f:M_1\rightarrow M_2$ is a $\ZygSpacemap{s+1}$ map which is also a $C^1$ diffeomorphism.  Then
$f^{-1}:M_2\rightarrow M_1$ is a $\ZygSpacemap{s+1}$ map.
\end{lemma}
\begin{proof}
This follows from \cref{Lemma::FuncManifold::InverseEuclid}.
\end{proof}

\begin{defn}
Suppose $s\in (0,\infty]\cup\{\omega\}$, and $M_1$ and $M_2$ are $\ZygSpace{s+1}$ manifolds.  We say $f:M_1\rightarrow M_2$ is a $\ZygSpacediff{s+1}$ diffeomorphism if $f:M_1\rightarrow M_2$ is
invertible and $f:M_1\rightarrow M_2$ and $f^{-1}:M_2\rightarrow M_1$ are $\ZygSpacemap{s+1}$ maps.
\end{defn}

\begin{rmk}\label{Rmk::FuncMflds::CoordChartsAreDiffeo}
For $s\in (0,\infty]\cup\{\omega\}$, if  $M$ is a $\ZygSpace{s+1}$ manifold with with $\ZygSpace{s+1}$ atlas $\{(\phi_\alpha,V_\alpha)\}$, as described in \cref{Defn::FuncMfld::Atlas}, then the maps
$\phi_{\alpha}:V_{\alpha}\rightarrow U_\alpha$ are $\ZygSpacediff{s+1}$ diffeomorphisms, where $U_\alpha$ is given the $\ZygSpace{\omega}$ manifold structure described in \cref{Rmk::FuncMfld::OpenSetsAreManifolds}.
This follows from \cref{Lemma::FuncManfiold::InverseMap}.
\end{rmk}

Because a $\ZygSpace{s+1}$ manifold is a $C^1$ manifold, it makes sense to talk about vector fields on such a manifold.

\begin{defn}
For $s\in (0,\infty]\cup \{\omega\}$ let $M$ be a $\ZygSpace{s+1}$ manifold of dimension $n$ with $\ZygSpace{s+1}$ atlas $\{(\phi_{\alpha}, V_\alpha)\}$; here $\phi_{\alpha}:V_{\alpha}\rightarrow U_\alpha$ 
is a $\ZygSpacediff{s+1}$ diffeomorphism
and
$U_\alpha\subseteq \R^n$ is open.  We say a vector field $X$ on $M$ is a $\ZygSpace{s}$ vector field if $(\phi_{\alpha})_{*} X\in \ZygSpaceloc{s}[U_\alpha][\R^n]$, $\forall \alpha$.
\end{defn}

\section{Bundles}\label{Section::Bundles}
In this section, we include the standard definitions we use concerning bundles.  In the smooth case, these definitions are contained in \cite{TrevesHypoanalyticStructures,BerhanuCordaroHounieAnIntroductionToInvolutive},
and we follow these sources.
Fix $s\in (0,\infty]\cup\{\omega\}$, and let $M$ be a $\ZygSpace{s+2}$ manifold.
 We let $\C TM$ denote the complexified tangent space of $M$:  $\C T_{\zeta} M := TM \otimes_{\R} \C$ 
(see \cref{Appendix::LinearAlgebra} for some comments on the complexification of real vector spaces).

\begin{defn}
A $\ZygSpace{s+1}$ sub-bundle $\LVS$ of $\C TM$ of rank $m\in \N$ is a disjoint union
\begin{equation*}
	\LVS = \bigcup_{\zeta \in M} \LVS[\zeta] \subseteq \C T M
\end{equation*}
such that:
\begin{itemize}
	\item $\forall \zeta\in M$, $\LVS[\zeta]$ is an $m$-dimensional vector subspace of $\C T_\zeta M$.
	\item $\forall \zeta_0\in M$, there exists an open neighborhood $U\subseteq M$ of $\zeta_0$ and a finite collection of complex $\ZygSpace{s+1}$ vector fields
	$L_1,\ldots, L_K$ on $U$, such that $\forall \zeta\in U$,
	\begin{equation*}
		\Span_\C \{ L_1(\zeta),\ldots, L_K(\zeta) \} = \LVS[\zeta].
	\end{equation*}
\end{itemize}
\end{defn}

\begin{defn}
For a $\ZygSpace{s+1}$ sub-bundle $\LVS$ of $\C TM$, we define $\LVSb$ by $\LVSb[\zeta]= \{ \overline{z} : z\in \LVS[\zeta]\}$.  It is easy to see
that $\LVSb$ is a $\ZygSpace{s+1}$ sub-bundle of $\C TM$.
\end{defn}

\begin{defn}
Let $W\subseteq M$ be open, $L$ a complex vector field on $W$, and $\LVS$ a $\ZygSpace{s+1}$ sub-bundle of $\C TM$.
We say $L$ is a section of $\LVS$ over $W$ if $\forall \zeta\in W$, $L(\zeta)\in \LVS[\zeta]$.
We say $L$ is a $\ZygSpace{s+1}$ section of $\LVS$ over $W$ is if $L$ is a section of $\LVS$ over $W$ and $L$ is a $\ZygSpace{s+1}$ complex vector field on $W$.
\end{defn}

\begin{defn}
Let $\LVS$ be a $\ZygSpace{s+1}$ sub-bundle of $\C TM$.  We say $\LVS$ is a $\ZygSpace{s+1}$ formally integrable structure if the following holds.
For all $W\subseteq M$ open, and all $\ZygSpace{s+1}$ 
sections $L_1$ and $L_2$ of $\LVS$ over $W$,
we have $[L_1,L_2]$ is a section of $\LVS$ over $W$.
\end{defn}

\begin{defn}\label{Defn::Bundles::EllipticStructure}
Let $\LVS$ be a $\ZygSpace{s+1}$ formally integrable structure on $M$.  We say $\LVS$ is a $\ZygSpace{s+1}$ elliptic structure if 
$\LVS[\zeta]+\LVSb[\zeta]=\C T_{\zeta} M$, $\forall \zeta\in M$.
\end{defn}

\begin{lemma}\label{Lemma::Bundles::DimXConst}
Let $\LVS$ be an elliptic structure on $M$.  Then, the map $\zeta\mapsto \dim (\LVS[\zeta]\cap \LVSb[\zeta])$ is constant, $M\rightarrow \N$.
\end{lemma}
\begin{proof}
By \cref{Lemma::AppendCR::dimFormula}, $\dim (\LVS[\zeta]\cap \LVSb[\zeta])= 2\dim(\LVS[\zeta])-\dim (\LVS[\zeta]+\LVSb[\zeta])$.
The definition of a sub-bundle implies $\zeta\mapsto \dim(\LVS[\zeta])$ is constant, and the definition of an elliptic structure implies
$\dim (\LVS[\zeta]+\LVSb[\zeta])=\dim \C T_\zeta M = \dim M$, $\forall \zeta\in M$.
The result follows.
\end{proof}

Let $\LVS$ be an elliptic structure on $M$.  Set $r:=\dim  (\LVS[\zeta]\cap \LVSb[\zeta])$ and $n+r:=\dim (\LVS[\zeta])$.  By the definition of a sub-bundle
and \cref{Lemma::Bundles::DimXConst}, $n$ and $r$ are constant in $\zeta$. 

\begin{defn}
Let $\LVS$ be a elliptic structure on $M$ and let $n$ and $r$ be as above.  We say $\LVS$ is an elliptic structure of dimension $(r,n)$.
\end{defn}

\begin{rmk}\label{Rmk::Bundles::DimensionMfld}
Let $\LVS$ be an elliptic structure of dimension $(r,n)$.  Then,  $\dim M = \dim \C T_{\zeta} M = \dim (\LVS[\zeta]+\LVSb[\zeta]) =2n+r$,
where in the last equality we have used \cref{Lemma::AppendCR::dimFormula}.
\end{rmk}

\section{E-manifolds}\label{Section::EMflds}
It is convenient to state our results in a category of manifolds which contain real manifolds and complex manifolds as full sub-categories.
We define these manifolds here, and call them E-manifolds.\footnote{The manifold structure we discuss here is well-known to experts, but we could not find a name for the category of such manifolds,
and decided to call them E-manifolds for lack of a better name.}

\begin{rmk}
``E'' in the name E-manifolds stands for ``elliptic''.  Indeed, using the terminology of \cite[Definition I.2.3]{TrevesHypoanalyticStructures},
a complex manifold is a manifold endowed with a complex structure, a CR-manifold is a manifold endowed with a CR structure, and (as we will see in \cref{Thm::EMfld::MainThm}) an
E-manifold is a manifold endowed with an elliptic structure; see \cref{Defn::EMfld::sL}.  Unfortunately, the name ``elliptic manifold'' is already taken by an unrelated concept.
\end{rmk}

\begin{defn}
Let $U_1\subseteq \R^{r_1}\times \C^{n_1}$ and $U_2\subseteq \R^{r_2}\times \C^{n_2}$ be open sets.  We give $\R^{r_1}\times \C^{n_1}$ coordinates $(t,z)$
and $\R^{r_2}\times \C^{n_2}$ coordinates $(u,w)$.  We say a $C^1$ map $f:U_1\rightarrow U_2$ is an E-map if
\begin{equation*}
	df(t,z) \diff{t_k}, df(t,z)\diff{\zb[j]}\in \Span_{\C}\left\{\diff{u_1},\ldots, \diff{u_{r_2}},\diff{\wb[1]},\ldots, \diff{\wb[n_2]}\right\},\quad \forall (t,z)\in U_1, 1\leq k\leq r_1, 1\leq j\leq n_1.
\end{equation*}
For $s\in (0,\infty]\cup\{\omega\}$, we say $f:U_1\rightarrow U_2$ is a $\ZygSpaceloc{s}$ E-map if it is an E-map which is also a $\ZygSpaceloc{s}$ map. 
\end{defn}

\begin{rmk}\label{Rmk::EMfld::InverseEuclid}
Suppose $U_1,U_2\subseteq \R^{r}\times \C^{n}$ and $f:U_1\rightarrow U_2$ is an E-map which is also a $C^1$-diffeomorphism.  Then, $f^{-1}:U_2\rightarrow U_1$
is an E-map.
\end{rmk}

\begin{rmk}\label{Rmk::Emfld::Holomorphic}
Note that when $r_1=r_2=0$, if $U_1\subseteq \R^{0}\times \C^{n_1}\cong \C^{n_1}$, $U_2\subseteq \R^{0}\times \C^{n_2}\cong \C^{n_2}$, then
$f:U_1\rightarrow U_2$ is an E-map if and only if it is holomorphic.
\end{rmk}

\begin{lemma}\label{Lemma::EMfld::CompuseEuclid}
Let $U_1\subseteq \R^{r_1}\times \C^{n_1}$, $U_2\subseteq \R^{r_2}\times \C^{n_2}$, and $U_3\subseteq \R^{r_3}\times \C^{n_3}$ be open sets, and let $s\in (0,\infty]\cup\{\omega\}$.
Suppose $f_1:U_1\rightarrow U_2$ and $f_2:U_2\rightarrow U_3$ are $\ZygSpacemap{s+1}$ E-maps. Then $f_2\circ f_1:U_1\rightarrow U_3$ is a $\ZygSpacemap{s+1}$ E-map.
\end{lemma}
\begin{proof}
That $f_2\circ f_1$ is a $\ZygSpacemap{s+1}$ map follows from \cref{Lemma::FuncManifold::ComposeEuclid}.  That it is an E-map follows from the chain rule.
\end{proof}

\begin{defn}
Let $M$ be a 
topological space and fix $n, r\in \N$, $s\in (1,\infty]\cup\{\omega\}$.  We say $\{(\phi_\alpha, V_\alpha) : \alpha\in \sI\}$ (where $\sI$ is some index set)
is a $\ZygSpace{s}$ E-atlas of dimension $(r,n)$ if $\{V_\alpha: \alpha\in \sI\}$ is an open cover for $M$,
$\phi_{\alpha}:V_\alpha\rightarrow U_\alpha$ is a homeomorphism where $U_\alpha\subseteq \R^r\times \C^n$ is open, and
$\phi_\beta \circ \phi_{\alpha}^{-1}: \phi_{\alpha}(V_\beta\cap V_\alpha)\rightarrow U_\beta$ is a $\ZygSpacemap{s}$ E-map, $\forall \alpha,\beta$.
\end{defn}

\begin{defn}
A $\ZygSpace{s}$ E-manifold $M$ of dimension $(r,n)$ is a paracompact\footnote{We do not use paracompactness in this paper; so the reader who does not require that manifolds be paracompact is free to do so in this paper.}  topological space $M$ endowed with a $\ZygSpace{s}$ E-atlas
of dimension $(r,n)$.
\end{defn}

\begin{rmk}
One may analogously define $C^m$ E-manifolds in the obvious way.  $C^\infty$ E-manifolds and $\ZygSpace{\infty}$ E-manifolds are the same.
\end{rmk}

\begin{defn}
For $s\in (0,\infty]\cup\{\omega\}$,
let $M$ and $N$ be $\ZygSpace{s+1}$ E-manifolds with $\ZygSpace{s+1}$ E-atlases $\{(\phi_{\alpha},V_{\alpha})\}$ and $\{(\psi_\beta, W_\beta)\}$, respectively.
We say $f:M\rightarrow N$ is a $\ZygSpacemap{s+1}$ E-map if $\psi_\beta\circ f\circ \phi_{\alpha}^{-1}$ is a $\ZygSpacemap{s+1}$ E-map, $\forall \alpha,\beta$.
\end{defn}

\begin{lemma}
For $s\in (0,\infty]\cup\{\omega\}$, let $M_1$, $M_2$, and $M_3$ be $\ZygSpace{s+1}$ E-manifolds and $f_1:M_1\rightarrow M_2$ and $f_2:M_2\rightarrow M_3$ be
$\ZygSpacemap{s+1}$ E-maps.  Then, $f_2\circ f_1:M_1\rightarrow M_3$ is a $\ZygSpacemap{s+1}$ E-map.
\end{lemma}
\begin{proof}
This follows from \cref{Lemma::EMfld::CompuseEuclid}.
\end{proof}

\begin{lemma}
For $s\in (0,\infty]\cup\{\omega\}$, let $M_1$ and $M_2$ be $\ZygSpace{s+1}$ E-manifolds and let $f:M_1\rightarrow M_2$ be a $\ZygSpacemap{s+1}$ E-map which is
also a $C^1$ diffeomorphism.  Then, $f^{-1}:M_2\rightarrow M_1$ is a $\ZygSpacemap{s+1}$ E-map.
\end{lemma}
\begin{proof}
That $f^{-1}:M_2\rightarrow M_1$ is a $\ZygSpacemap{s+1}$ map follows from \cref{Lemma::FuncManfiold::InverseMap}.  That $f^{-1}:M_2\rightarrow M_1$ is an E-map follows from \cref{Rmk::EMfld::InverseEuclid}.
\end{proof}

\begin{defn}
Suppose $s\in (0,\infty]\cup\{\omega\}$, $M_1$ and $M_2$ are $\ZygSpace{s+1}$ E-manifolds.  We say $f:M_1\rightarrow M_2$ is a $\ZygSpacediff{s+1}$ E-diffeomorphism if $f:M_1\rightarrow M_2$ is
invertible and $f:M_1\rightarrow M_2$ and $f^{-1}:M_2\rightarrow M_1$ are $\ZygSpacemap{s+1}$ E-maps.
\end{defn}

\begin{rmk}\label{Rmk::EMfld::FullSubcategory}
For $s>1$, the category of $\ZygSpace{s}$ E-manifolds, whose objects are $\ZygSpace{s}$ E-manifolds and morphisms are $\ZygSpacemap{s}$ E-maps,
contains both $\ZygSpace{s}$ real manifolds and complex manifolds as full subcategories.  The real manifolds of dimension $r$ are those
with E-dimension $(r,0)$, while the complex manifolds of complex dimension $n$ are those with E-dimension $(0,n)$.  That complex manifolds (with morphisms given by holomorphic maps)
embed as a \textit{full} subcategory follows from \cref{Rmk::Emfld::Holomorphic}.
The isomorphisms in the category of $\ZygSpace{s}$ E-manifolds are the $\ZygSpacediff{s}$ E-diffeomorphisms.
\end{rmk}

\begin{rmk}
Note that open subsets of $\R^r\times \C^n$ are $\ZygSpace{\omega}$
E-manifolds of dimension $(r,n)$, by using the atlas consisting
of one coordinate chart (the identity map).  
Henceforth,
we give such sets this E-manifold structure.
\end{rmk}


\begin{rmk}
An E-manifold of dimension $(r,n)$ has an underlying manifold structure of dimension $2n+r$, and it therefore makes sense to talk about any of the usual objects on manifolds with respect to an E-manifold.
\end{rmk}

For $s\in (0,\infty]\cup \{\omega\}$, on a $\ZygSpace{s+2}$ E-manifold $M$ of dimension $(r,n)$, there is a naturally associated $\ZygSpace{s+1}$ elliptic structure on $M$ of dimension $(r,n)$ defined as follows.
Let $(\phi_\alpha,V_\alpha)$ be an E-atlas for $M$.  For $\zeta\in M$, we have $\zeta\in V_\alpha$ for some $\alpha$.
We set:
\begin{equation*}
\LVS_{\zeta}:= \Span_{\C}\left\{ d\Phi_{\alpha}^{-1}(\Phi_\alpha(\zeta)) \diff{t_1},\ldots, d\Phi_{\alpha}^{-1}(\Phi_\alpha(\zeta)) \diff{t_r},d\Phi_{\alpha}^{-1}(\Phi_\alpha(\zeta)) \diff{\zb[1]},\ldots, d\Phi_{\alpha}^{-1}(\Phi_\alpha(\zeta)) \diff{\zb[n]} \right\}.
\end{equation*}
It is straightforward to check that $\LVS_{\zeta}\subseteq \C T_{\zeta}M$ is well-defined\footnote{I.e., $\LVS_{\zeta}$ does not depend on which $\alpha$ we pick with $\zeta\in V_\alpha$.} and $\LVS=\bigcup_{\zeta\in M} \LVS[\zeta]$ is a $\ZygSpace{s+1}$ elliptic structure on $M$ of dimension $(r,n)$.

\begin{defn}\label{Defn::EMfld::sL}
We call $\LVS$ the elliptic structure associated to the E-manifold $M$.
\end{defn}

\begin{lemma}\label{Lemma::EMfld::RecongnizeEMap}
Suppose $M$ and $\Mh$ are $\ZygSpace{s+2}$ E-manifolds with associated elliptic structures $\LVS$ and $\LVSh$.
Then a $\ZygSpacemap{s+2}$ map $f:M\rightarrow \Mh$ is a $\ZygSpacemap{s+2}$ E-map if and only if
$df(\zeta) \LVS_\zeta \subseteq \LVSh_{f(\zeta)}$, $\forall \zeta\in M$.
\end{lemma}
\begin{proof}This follows immediately from the definitions.\end{proof}

The main result of this paper (\cref{Thm::Intro::MainThm}) can be rephrased as follows.
\begin{thm}\label{Thm::EMfld::MainThm}
Let $s\in (0,\infty]\cup \{\omega\}$ and let $M$ be a $\ZygSpace{s+2}$ manifold.  For each $\zeta\in M$, let $\LVS[\zeta]$ be a vector subspace of $\C T_\zeta M$, and let $\LVS = \bigcup_{\zeta\in M} \LVS[\zeta]$.
The following are equivalent:
\begin{enumerate}[(i)]
\item\label{Item::MainThm::EMfld} There is a $\ZygSpace{s+2}$ E-manifold structure on $M$, compatible with its $\ZygSpace{s+2}$ structure, such that $\LVS$ is the $\ZygSpace{s+1}$ elliptic structure associated to $M$.
\item\label{Item::MainThm::Bundle} $\LVS$ is a $\ZygSpace{s+1}$ elliptic structure.
\end{enumerate}
Moreover, under these conditions, the E-manfiold structure given in \cref{Item::MainThm::EMfld} is unique in the sense that if $M$ is given another $\ZygSpace{s+2}$ E-manifold structure, compatible with its $\ZygSpace{s+2}$ structure, with respect to which $\LVS$ is the associated
elliptic sub-bundle, then the identity map $M\rightarrow M$ is a $\ZygSpacediff{s+2}$ E-diffeomorphism, between these two $\ZygSpace{s+2}$ E-manifold structures on $M$.
\end{thm}

This paper is devoted to proving \cref{Thm::EMfld::MainThm}.

\begin{rmk}
In \cref{Thm::EMfld::MainThm}, following standard terminology, we have used the word ``structure'' in two different ways.  When we speak of a $\ZygSpace{s+2}$ E-manifold ``structure'' on $M$ we mean the equivalence class of $\ZygSpace{s+2}$ E-atlases (where two atlases on $M$ are equivalent if the identity map $M\rightarrow M$ is a $\ZygSpace{s+2}$ E-diffeomorphism).  When we speak of an elliptic ``structure,'' we are referring to \cref{Defn::Bundles::EllipticStructure}.
This double use of terminology is justified by \cref{Thm::EMfld::MainThm} which shows that giving an E-manifold structure is equivalent to giving an elliptic structure.
\end{rmk}

\begin{rmk}
When $s\in \{\infty,\omega\}$, \cref{Thm::EMfld::MainThm} is well-known.  Our proof yields these cases as simple corollaries, so we include them.
\end{rmk}

\begin{rmk}
In the special case $\LVS[\zeta]\cap \LVSb[\zeta]=\{0\}$, $\forall \zeta\in M$, 
\cref{Thm::EMfld::MainThm} is the Newlander-Nirenberg Theorem \cite{NewlanderNirenbergComplexAnalyticCoordiantesInAlmostComplexManifolds}, with sharp regularity as proved by Malgrange \cite{MalgrangeSurLIntegbrabilite}.
In this case E-manifolds are complex manifolds--see \cref{Rmk::EMfld::FullSubcategory}.
\end{rmk}

\section{Function Spaces Revisited}
In this section we present some basic properties of the function spaces introduced in \cref{Section::FunctionSpaces}.  Fix $\Omega\subseteq \R^n$ an open set.

\begin{prop}\label{Prop::FuncSpaceRev::Algebra}
For $s\in (0,\infty]\cup\{\omega\}$, $\ZygSpace{s}[\Omega]$ is an algebra:  if $f,g\in \ZygSpace{s}[\Omega]$, then $fg\in \ZygSpace{s}[\Omega]$.  Moreover, for $s\in (0,\infty)$ and $f,g\in \ZygSpace{s}[\Omega]$,
\begin{equation*}
\ZygNorm{fg}{s}[\Omega]\leq C_s \ZygNorm{f}{s}[\Omega]\ZygNorm{g}{s}[\Omega].
\end{equation*}
For $s\in (0,\infty]\cup \{\omega\}$, these spaces have multiplicative inverses for functions
which are bounded away from zero:   if $f\in \ZygSpace{s}[\Omega]$ with $\inf_x |f(x)|\geq c_0>0$, then $f(x)^{-1}=\frac{1}{f(x)}\in \ZygSpace{s}[\Omega]$.
Moreover if $s\in (0,\infty)$ and $\inf_x |f(x)|\geq c_0>0$ then
\begin{equation*}
	\ZygNorm{f(x)^{-1}}{s}[\Omega]\leq C,
\end{equation*}
where $C$ can be chosen to depend only on $s$, $n$, $c_0$, and an upper bound for $\ZygNorm{f}{s}[\Omega]$.
\end{prop}
\begin{proof}
For $s=\omega$, this is standard.
For $s\in (0,\infty]$, this is standard and contained in \cite[\SSZygIsAlgebra]{StovallStreetI}.
\end{proof}

\begin{rmk}\label{Rmk::FuncSpaceRev::InverseMatrix}
For $s\in (0,\infty]\cup \{\omega\}$,
suppose $A\in \ZygSpace{s}[\Omega][\M^{k\times k}]$
is such that $\inf_{t\in \Omega} |\det A(t)|>0$.  Then it follows that $A^{-1}\in \ZygSpace{s}[\Omega][\M^{k\times k}]$.
Indeed, this follows from \cref{Prop::FuncSpaceRev::Algebra} using the cofactor representation of $A^{-1}$.  
When $s\in (0,\infty)$, $\ZygNorm{A^{-1}}{s}[\Omega]$ can be bounded in terms of $s$, $k$, $n$, a lower bound for $\inf_{t\in \Omega} |\det A(t)|>0$, and an upper bound for
$\ZygNorm{A}{s}[\Omega]$.
\end{rmk}

\begin{lemma}\label{Lemma::FuncSpaceRev::Composition}
Let $D_1,D_2>0$, $s_1\in (0,\infty)$, $s_2\geq s_1$, $s_2\in (1,\infty)$, $f\in \ZygSpace{s_1}[B_{\R^n}(D_1)]$, $g\in \ZygSpace{s_2}[B_{\R^m}(D_2)][\R^n]$
with $g(B_{\R^m}(D_2))\subseteq B_{\R^n}(D_1)$.  Then, $f\circ g\in \ZygSpace{s_1}[B_{\R^m}(D_2)]$ and
$\ZygNorm{f\circ g}{s_1}[B_{\R^m}(D_2)]\leq C \ZygNorm{f}{s_1}[B_{\R^n}(D_1)]$, where $C$ can be chosen to depend only on $s_1$, $s_2$, $D_1$, $D_2$,
$m$, $n$, and an upper bound for $\ZygNorm{g}{s_2}[B_{\R^m}(D_2)]$.
\end{lemma}
\begin{proof}This is standard and proved in \cite{StovallStreetII}.\end{proof}

\begin{lemma}\label{Lemma::FuncSpaceRev::Inverse}
Fix $s\in (1,\infty)$, $D_1,D_2>0$.  Suppose $H\in \ZygSpace{s}[B_{\R^n}(D_1)][\R^n]$ is such that $B_{\R^n}(D_2)\subseteq H(B_{\R^n}(D_1))$,
$H:B_{\R^n}(D_1)\rightarrow H(B_{\R^n}(D_1))$ is a homeomorphism, and $\inf_{t\in B_{\R^n}(D_1)} |\det dH(t)|\geq c_0>0$.
Then $H^{-1}\in \ZygSpace{s}[B_{\R^n}(D_2)][\R^n]$, with $\ZygNorm{H^{-1}}{s}[B_{\R^n}(D_2)][\R^n]\leq C$, where $C$ can be chosen
to depend only on $n$, $s$, $D_1$, $D_2$, $c_0$, and an upper bound for $\ZygNorm{H}{s}[B_{\R^n}(D_1)][\R^n]$.
\end{lemma}
\begin{proof}This is standard and proved in  \cite{StovallStreetII}.\end{proof}

	\subsection{Spaces of Real Analytic Functions}\label{Section::FuncRev::RealAnal}
For the proofs that follow, it is convenient to introduce two, closely related, Banach spaces of real analytic functions.
For $s>0$, we define
$\ASpace{n}{s}$ to be the space of those $f\in \CSpace{B_{\R^n}(s)}$ such that $f(t)=\sum_{\alpha\in \N^n} \frac{c_\alpha}{\alpha!} t^{\alpha}$, $\forall t\in B_{\R^n}(s)$, where
\begin{equation*}
\ANorm{f}{n}{s}:= \sum_{\alpha\in \N^n} \frac{|c_{\alpha}|}{\alpha!} s^{|\alpha|}<\infty.
\end{equation*}

We now turn to the other Banach space of real analytic functions we use.
Let $\Omega\subset \C^N$ be a bounded, open set, and let $m\in \N$.  We set
\begin{equation*}
\OSpace{m}[\Omega]:=\left\{f:\Omega\rightarrow \C^m \: \big|\: f\text{ is holomorphic and }f\text{ extends to a continuous function }\Extend(f)\in \CSpace{\overline{\Omega}}\right\}.
\end{equation*}
With the norm
\begin{equation*}
\ONorm{f}{m}[\Omega]:=\CNorm{\Extend(f)}{\Omega},
\end{equation*}
$\OSpace{m}[\Omega]$ is a Banach space.
We set, for $\eta>0$,
\begin{equation*}
\sBSpace{N}{m}{\eta}:=\left\{
f:B_{\R^N}(\eta)\rightarrow \C^m \:\big|\: f\text{ is real analytic and extends to a holomorphic function }\Extend(f)\in \OSpace{m}[B_{\C^N}(\eta)]
\right\}.
\end{equation*}
With the norm
\begin{equation*}
\sBNorm{f}{N}{m}{\eta}
:=\ONorm{\Extend(f)}{m}[B_{\C^N}(\eta)],
\end{equation*}
$\sBSpace{N}{m}{\eta}$ is a Banach space.
Sometimes we wish to replace $\C^m$ in the above definitions with
a more general complex Banach space $\BanachSpace$.
We write this space as $\sBSpace{N}{\BanachSpace}{\eta}$ and define
the norm in the obvious way. 

\begin{lemma}\label{Lemma::FuncSpaceRevAnal::AInB}
Let $\BanachSpace$ be a Banach space.  Then $\ASpace{n}{\eta}[\BanachSpace]\subseteq \sBSpace{n}{\BanachSpace}{\eta}$ and
$\sBNorm{f}{n}{\BanachSpace}{\eta}\leq \ANorm{f}{n}{\eta}[\BanachSpace]$.
\end{lemma}
\begin{proof}
This follows immediately from the definitions.
\end{proof}

\begin{lemma}\label{Lemma::FuncSpaceRevAnal::BanachAlgebra}
Let $\BanachAlgebra$ be a Banach algebra.  Then $\ASpace{n}{s}[\BanachAlgebra]$ and $\sBSpace{N}{\BanachAlgebra}{\eta}$ are Banach algebras.  Indeed,
if $\BanachSpace$ denotes either of these spaces, then if $f,g\in \BanachSpace$, we have $fg\in \BanachSpace$ and $\Norm{fg}[\BanachSpace]\leq \Norm{f}[\BanachSpace]\Norm{g}[\BanachSpace]$.
\end{lemma}
\begin{proof}
This follow easily from the definitions.
\end{proof}

\begin{lemma}\label{Lemma::FuncSpaceReAnal::Restrict}
Fix $0<\eta_1<\eta_2$.  If $f\in \ASpace{n}{\eta_2}$ with $f(0)=0$, then
\begin{equation}\label{Eqn::FuncSpaceReAnal::RestrctA}
\ANorm{f}{n}{\eta_1}\leq \frac{\eta_1}{\eta_2} \ANorm{f}{n}{\eta_2}.
\end{equation}
Similarly, if $f\in \sBSpace{N}{m}{\eta_2}$ with $f(0)=0$, then
\begin{equation}\label{Eqn::FuncSpaceReAnal::RestrctB}
\sBNorm{f}{N}{m}{\eta_1}\leq \frac{\eta_1}{\eta_2}\sBNorm{f}{N}{m}{\eta_2}.
\end{equation}
The same results hold (with the same proofs) for functions taking values in Banach spaces.
\end{lemma}
\begin{proof}
Suppose $f\in \ASpace{n}{\eta_2}$ with $f(0)=0$.  Then, $f(t)=\sum_{|\alpha|>0} c_\alpha t^{\alpha}$ with
$\ANorm{f}{n}{\eta_2} = \sum_{|\alpha|>0} |c_{\alpha}| \eta_2^{|\alpha|}$.  We have
\begin{equation*}
\ANorm{f}{n}{\eta_1} = \sum_{|\alpha|>0} |c_\alpha| \eta_1^{|\alpha|}\leq \frac{\eta_1}{\eta_2} \sum_{|\alpha|>0} |c_{\alpha}| \eta_2^{|\alpha|}=\frac{\eta_1}{\eta_2}\ANorm{f}{n}{\eta_2}, 
\end{equation*}
completing the proof of \cref{Eqn::FuncSpaceReAnal::RestrctA}.

Let $g\in \OSpace{m}[B_{\C}(\eta_2)]$ with $g(0)=0$.  We claim
\begin{equation}\label{Eqn::FuncSpaceReAnal::RestrctO1}
\ONorm{g}{m}[B^1(\eta_1)]\leq \frac{\eta_1}{\eta_2} \ONorm{g}{m}[B_{\C}(\eta_2)].
\end{equation}
Indeed, we may write $g(z)=zg_1(z)$, where $g_1\in \OSpace{m}[B_{\C}(\eta_2)]$.  We have, by the Maximum Modulus Principle:
\begin{equation*}
\ONorm{g}{m}[B^1(\eta_1)] \leq \eta_1 \ONorm{g_1}{m}[B_{\C}(\eta_1)] \leq \eta_1 \sup_{|z|=\eta_2} |g_1(z)| = \frac{\eta_1}{\eta_2}  \sup_{|z|=\eta_2} |g(z)| = \frac{\eta_1}{\eta_2} \ONorm{g}{m}[B^1(\eta_2)],
\end{equation*}
completing the proof of \cref{Eqn::FuncSpaceReAnal::RestrctO1}.

Let $h\in \OSpace{m}[B_{\C^n}(\eta_2)]$ with $h(0)=0$.  We claim
\begin{equation}\label{Eqn::FuncSpaceReAnal::RestrctOn}
\ONorm{h}{m}[B_{\C^n}(\eta_1)]\leq \frac{\eta_1}{\eta_2} \ONorm{h}{m}[B_{\C^n}(\eta_2)].
\end{equation}
Indeed, for $0\ne w\in B_{\C^n}(\eta_1)$, apply \cref{Eqn::FuncSpaceReAnal::RestrctO1} to $g(z):= h(z w/|w|)$, to see $|h(w)|\leq \frac{\eta_1}{\eta_2}  \ONorm{h}{m}[B^n(\eta_2)]$.
Taking the supremum over all such $w$ yields \cref{Eqn::FuncSpaceReAnal::RestrctOn}.

\Cref{Eqn::FuncSpaceReAnal::RestrctB} is an immediate consequence of \cref{Eqn::FuncSpaceReAnal::RestrctOn}.
\end{proof}

\begin{lemma}\label{Lemma::FuncSpaceReAnal::AinB}
Fix $0<\eta_1<\eta_2$.  Then $\sBSpace{n}{1}{\eta_2}\subseteq \ASpace{n}{\eta_1}$ and for $f\in \sBSpace{n}{1}{\eta_2}$,
\begin{equation*}
\ANorm{f}{n}{\eta_1}\leq C\sBNorm{f}{n}{1}{\eta_2},
\end{equation*}
where $C$ can be chosen to depend only on $n$, $\eta_2$, and $\eta_1$.

Similarly for $s\in (0,\infty)$, $\sBSpace{n}{1}{\eta_2}\subseteq \ZygSpace{s}[B^n(\eta_1)]$ and for $f\in \ZygSpace{s}[B^n(\eta_1)]$,
\begin{equation*}
\ZygNorm{f}{s}[B^n(\eta_1)]\leq C\sBNorm{f}{n}{1}{\eta_2},
\end{equation*}
where $C$ can be chosen to depend only on $s$, $n$, $\eta_2$, and $\eta_1$.
\end{lemma}
\begin{proof}
It suffices to prove both results for $\eta_2=1$ and $\eta_1\in (0,1)$, by rescaling.  
When $\eta_2=1$, we extend $f$ to a holomorphic function $\Extend(f)\in \OSpace{1}[B_{\C^n}(1)]$, and use the well-known representation:
\begin{equation*}
 \Extend f(z)=\frac{(n-1)!}{2\pi^n} \int_{\partial B_{\C^n}(1)}  \Extend f(\zeta) \frac{1-\overline{z}\cdot \zeta}{|\zeta-z|^{2n}}\: d\sigma(\zeta),
\end{equation*}
where $\sigma$ denotes the surface area measure on $\partial B_{\C^n}(1)$.  From here, the results follows easily.
\end{proof}

\begin{lemma}\label{Lemma::FuncSpaceRev::Scale}
Fix $\eta_1>0$, $D>0$.  For $0<\gamma\leq  \frac{\eta_1}{D}$ and $f:B^n(\eta_1)\rightarrow \C$, define $f_\gamma:B^n(D)\rightarrow \C$ by $f_\gamma(t)=f(\gamma t)$.
\begin{enumerate}[(i)]
\item\label{Item::FuncSpaceRev::Scaling::Zygmund} Let $m\in \N$ with $m\geq 1$ and $s\in (0,1]$.
Then, for $0<\gamma \leq \min\{\frac{\eta_1}{D}, 1\}$, we have for $f\in \ZygSpace{m+s}[B^n(\eta_1)]$ with $f(0)=0$,
\begin{equation*}
	\ZygNorm{f_\gamma}{m+s}[B^n(D)]\leq \gamma (15(D+1)+1)\ZygNorm{f}{m+s}[B^n(\eta_1)].
\end{equation*}

\item\label{Item::FuncSpaceRev::Scaling::Anal} For $0<\gamma \leq \frac{\eta_1}{D}$,
we have for  $f\in \ASpace{n}{\eta_1}$ with $f(0)=0$,
\begin{equation*}
\ANorm{f_\gamma}{n}{D} \leq \frac{\gamma D}{\eta_1} \ANorm{f}{n}{\eta_1}.
\end{equation*}
\end{enumerate}
\end{lemma}
\begin{proof}
We begin with \cref{Item::FuncSpaceRev::Scaling::Zygmund}.  Using $0<\gamma \leq \min\{\frac{\eta_1}{D}, 1\}$, it follows immediately from the definitions that
\begin{equation}\label{Eqn::FuncSpaceRev::Scale0Vanish0}
\begin{split}
&\sum_{1\leq |\alpha|\leq m} \ZygNorm{\partial_x^{\alpha} f_\gamma}{s}[B^n(D)] = \sum_{1\leq |\alpha|\leq m} \gamma^{|\alpha|} \ZygNorm{(\partial_x^{\alpha} f)(\gamma\cdot)}{s}[B^n(D)]
\\&
\leq \sum_{1\leq |\alpha|\leq m} \gamma^{|\alpha|} \ZygNorm{\partial_x^{\alpha} f}{s}[B^n(\eta_1)] \leq \gamma \ZygNorm{f}{m+s}[B^n(\eta_1)].
\end{split}
\end{equation}
Since $f_\gamma(0)=f(0)=0$, we have (using the Fundamental Theorem of Calculus)
\begin{equation}\label{Eqn::FuncSpaceRev::Scale0Vanish1}
\begin{split}
&\CjNorm{f_\gamma}{1}[B^n(D)] = \CjNorm{f_\gamma}{0}[B^n(D)] +\sum_{|\alpha|=1} \CjNorm{\partial_x^{\alpha} f_\gamma}{0}[B^n(D)]
\\&\leq (D+1) \sum_{|\alpha|=1} \CjNorm{\partial_x^{\alpha} f_\gamma}{0}[B^n(D)] \leq (D+1)\gamma \CjNorm{f}{1}[B^n(\eta_1)].
\end{split}
\end{equation}
It is easy to see, directly from the definitions, that (for any function $g$ on any ball $B$),\footnote{See \cite[\SSCompareFunctionSpaces]{StovallStreetI}
for a result like \cref{Eqn::FuncSpaceRevAnal::CompareBallFS}--in the proof of that result, one can see how the constants $5$ and $15$ arise.  However, these particular constants
are not essential for what follows.}
\begin{equation}\label{Eqn::FuncSpaceRevAnal::CompareBallFS}
\ZygNorm{g}{s}[B]\leq 5 \HNorm{g}{0}{s}[B]\leq 15 \HNorm{g}{0}{1}[B]\leq 15 \CjNorm{g}{1}[B]\leq 15 \ZygNorm{g}{m+s}[B].
\end{equation}
Thus, using \cref{Eqn::FuncSpaceRev::Scale0Vanish1}, we have
\begin{equation*}
\ZygNorm{f_\gamma}{s}[B^n(D)]\leq 15 \CjNorm{f_\gamma}{1}[B^n(D)] \leq  15(D+1) \gamma \CjNorm{f}{1}[B^n(\eta_1)]\leq 15(D+1) \gamma \ZygNorm{f}{m+s}[B^n(\eta_1)].
\end{equation*}
Combining this with \cref{Eqn::FuncSpaceRev::Scale0Vanish0} completes the proof of \cref{Item::FuncSpaceRev::Scaling::Zygmund}.

We turn to \cref{Item::FuncSpaceRev::Scaling::Anal}.  Let $f\in \ASpace{n}{\eta_1}$ with $f(0)=0$, so that $f(t)=\sum_{|\alpha|>0} c_\alpha t^{\alpha}$,
and $\ANorm{f}{n}{\eta_1} =\sum_{|\alpha|>0} |c_\alpha|\eta_1^{|\alpha|}$.  For $0<\gamma \leq \frac{\eta_1}{D}$ we have
$f_\gamma(t) = \sum_{|\alpha|>0} c_\alpha \gamma^{|\alpha|} t^{\alpha}$, and therefore $f_\gamma\in \ASpace{n}{D}$ and we have
\begin{equation*}
\ANorm{f_\gamma}{n}{D}=\sum_{|\alpha|>0} |c_\alpha| (\gamma D)^{|\alpha|} = \sum_{|\alpha|>0} |c_\alpha| (\eta_1)^{|\alpha|} \mleft(\frac{\gamma D}{\eta_1}\mright)^{|\alpha|}
\leq \mleft(\frac{\gamma D}{\eta_1}\mright) \ANorm{f}{n}{\eta_1},
\end{equation*}
completing the proof of \cref{Item::FuncSpaceRev::Scaling::Anal}.
\end{proof}

\begin{lemma}\label{Lemma::FuncSpaceRev::ComposeAnal}
Let $\eta_1,\eta_2>0$, $n_1,n_2\in \N$, and let $\BanachSpace$ be a Banach space.
Suppose $f\in \ASpace{n_1}{\eta_1}[\BanachSpace]$, $g\in \ASpace{n_2}{\eta_2}[\R^{n_1}]$ with
$\ANorm{g}{n_2}{\eta_2}[\R^{n_1}]\leq \eta_1$.  Then, $f\circ g\in \ASpace{n_2}{\eta_2}[\BanachSpace]$
with $\ANorm{f\circ g}{n_2}{\eta_2}\leq \ANorm{f}{n_1}{\eta_1}$.
\end{lemma}
\begin{proof}
This is immediate from the definitions.
\end{proof}

\begin{lemma}\label{Lemma::FuncSpaceRev::DerivOfAnal}
Fix $0<\eta_2<\eta_1$, and suppose $f\in \ASpace{n}{\eta_1}[\BanachSpace]$, where $\BanachSpace$ is a Banach space.
Then, for each $j=1,\ldots, n$, $\diff{t_j} f(t)\in \ASpace{n}{\eta_2}[\BanachSpace]$ and
$\ANorm{\diff{t_j} f}{n}{\eta_2}[\BanachSpace] \leq C \ANorm{f}{n}{\eta_1}[\BanachSpace]$, where $C$ can be chosen to depend only on $\eta_1$ and $\eta_2$.
\end{lemma}
\begin{proof}
Without loss of generality, we prove the result for $j=1$.  We let $e_1$ denote the first standard basis element:  $e_1=(1,0,\ldots, 0)\in \R^n$.
Suppose $f(t) = \sum c_{\alpha} \frac{t^{\alpha}}{\alpha!}$.  Then, $\diff{t_j} f(t) = \sum_{\alpha_1>0} c_\alpha \frac{t^{\alpha-e_1}}{(\alpha-e_1)!}$.
Hence,
\begin{equation*}
\BANorm{\diff{t_1} f}{n}{\eta_2} = \sum_{\alpha_1>0} \frac{\Norm{c_\alpha}[\BanachSpace]}{(\alpha-e_1)!} \eta_2^{|\alpha-e_1|}
=\sum_{\alpha} \frac{\Norm{c_\alpha}[\BanachSpace]}{\alpha!} \eta_1^{|\alpha|} \mleft(\frac{\eta_2}{\eta_1}\mright)^{|\alpha|} \frac{\alpha_1}{\eta_1}
\leq \mleft( \sup_{\alpha}\mleft(\frac{\eta_2}{\eta_1}\mright)^{|\alpha|} \frac{\alpha_1}{\eta_1} \mright) \ANorm{f}{n}{\eta_1},
\end{equation*}
completing the proof.
\end{proof}

\section{Some Additional Notation}\label{Section::Notation}
If $f:M\rightarrow N$ is a $C^1$ map between $C^1$ manifolds, we write $df(x):T_xM\rightarrow T_xN$ for the usual differential.  We extend this to be a complex linear map
$df(x):\C T_xM\rightarrow \C T_x N$, where $\C T_x M = T_x M\otimes_{\R} \C$ denotes the complexified tangent space.
Even if the manifold $M$ has additional structure (e.g., in the case of a complex manifold), $df(x)$ is defined in terms of the underlying real manifold structure.

When working on $\R^r\times \C^n$ we will often use coordinates $(t,z)$ where $t=(t_1,\ldots, t_r)\in \R^r$ and $z=(z_1,\ldots, z_n)\in \C^n$.
We write
\begin{equation*}
\diff{t}=
\begin{bmatrix}
\diff{t_1}\\
\diff{t_2}\\
\vdots\\
\diff{t_r}
\end{bmatrix},
\quad
\diff{z}
=\begin{bmatrix}
\diff{z_1}\\
\diff{z_2}\\
\vdots\\
\diff{z_n}
\end{bmatrix},
\quad
\diff{\zb}
=\begin{bmatrix}
\diff{\zb[1]}\\
\diff{\zb[2]}\\
\vdots\\
\diff{\zb[n]}
\end{bmatrix}.
\end{equation*}
At times we will instead use coordinates $(u,w)$ where $u\in \R^r$ and $w\in \C^n$ and define $\diff{u}$, $\diff{w}$, and $\diff{\wb}$ similarly.

For a function $F(t,z)=(F_1(t,z),\ldots, F_m(t,z)):\R^r\times \C^n\rightarrow \C^m$ we write
\begin{equation*}
d_t F =
\begin{bmatrix}
\frac{\partial F_1}{\partial t_1} & \cdots & \frac{\partial F_1}{\partial t_r} \\
\vdots &\ddots&\vdots\\
\frac{\partial F_m}{\partial t_1} & \cdots & \frac{\partial F_m}{\partial t_r}
\end{bmatrix},
\: d_z F =
\begin{bmatrix}
\frac{\partial F_1}{\partial z_1} & \cdots & \frac{\partial F_1}{\partial z_n} \\
\vdots &\ddots&\vdots\\
\frac{\partial F_m}{\partial z_1} & \cdots & \frac{\partial F_m}{\partial z_n}
\end{bmatrix},
\: d_{\zb} F =
\begin{bmatrix}
\frac{\partial F_1}{\partial \zb[1]} & \cdots & \frac{\partial F_1}{\partial \zb[n]} \\
\vdots &\ddots&\vdots\\
\frac{\partial F_m}{\partial \zb[1]} & \cdots & \frac{\partial F_m}{\partial \zb[n]}
\end{bmatrix},
\end{equation*}
We identify $\R^{r}\times \R^{2n}\cong \R^{r}\times \C^n$ via the map $(t_1,\ldots, t_r, x_1,\ldots,x_{2n})\mapsto (t_1,\ldots, t_r, x_1+ix_{n+1}, \ldots, x_n+ix_{2n})$.
Thus, given a function $G(t,z):\R^r\times \C^n\rightarrow \R^s\times \C^m$, we may also think of $F$ as function $G(t,x)=(G_1(t,x),\ldots, G_{s+2m}(t,x)):\R^{r}\times \R^{2n}\rightarrow \R^{s}\times \R^{2m}$.
For such a function, we write
\begin{equation*}
d_{(t,x)} G =
\begin{bmatrix}
\frac{\partial G_1}{\partial t_1} & \cdots & \frac{\partial G_1}{\partial t_r} & \frac{\partial G_1}{\partial x_1} & \cdots & \frac{\partial G_1}{\partial x_{2n}}\\
\vdots &\ddots &\vdots&\vdots&\ddots& \vdots\\
\vdots &\ddots &\vdots&\vdots&\ddots& \vdots\\
\frac{\partial G_{s+2m}}{\partial t_1} & \cdots & \frac{\partial G_{s+2m}}{\partial t_r} & \frac{\partial G_{s+2m}}{\partial x_1} & \cdots & \frac{\partial G_{s+2m}}{\partial x_{2n}}
\end{bmatrix}.
\end{equation*} 
We write $I_{N\times N}\in \M^{N\times N}$ to denote the $N\times N$ identity matrix, and $0_{a\times b}\in\M^{a\times b}$ to denote the $a\times b$ zero matrix.

\section{The Main Technical Result}\label{Section::Nirenbeg}
In this section, we state and prove the main technical result needed to prove \cref{Thm::EMfld::MainThm}.

Fix $s_0\in (0,\infty)\cup \{\omega\}$ and let
$X_1,\ldots, X_r, L_1,\ldots, L_n$ be complex vector fields
on $B_{\R^r\times \C^n}(1)$ with:
\begin{itemize}
\item If $s_0\in (0,\infty)$, $X_k, L_j\in \ZygSpace{s_0+1}[B_{\R^r\times \C^n}(1)][\C^{r+2n}]$.
\item If $s_0=\omega$, $X_k, L_j\in \ASpace{r+2n}{1}[\C^{r+2n}]$.
\end{itemize}
We suppose:
\begin{itemize}
\item $X_k(0)=\diff{t_k}$, $L_j(0)=\diff{\zb[j]}$.
\item $\forall \zeta\in B_{\R^r\times \C^n}(1)$,
$[X_{k_1},X_{k_2}](\zeta), [X_k, L_j](\zeta), [L_{j_1}, L_{j_2}](\zeta)
\in \Span_{\C}\left\{X_1(\zeta),\ldots, X_r(\zeta), L_1(\zeta),\ldots, L_n(\zeta)\right\}$.
\end{itemize}

Under these hypotheses, Nirenberg's theorem on the integrability of elliptic structures\footnote{Originally, Nirenberg considered only the case of $C^\infty$ vector fields and worked in the case when $X_1,\ldots, X_r$ were real.} implies
that there exists a map $\Phi_4:B_{\R^r\times \C^n}(1)\rightarrow B_{\R^r\times \C^n}(1)$, with $\Phi_4(0)=0$, $\Phi_4$ is a diffeomorphism
onto its image (which is an open neighborhood of $0\in B_{\R^r\times \C^n}(1)$), and such that
$\Phi_4^{*}X_k (u,w), \Phi_4^{*}L_j(u,w)\in \Span_{\C}\{\diff{u_1},\ldots, \diff{u_r}, \diff{\wb[1]},\ldots, \diff{\wb[r]}\}$, $\forall (u,w)$ (here
we are giving the domain space $\R^r\times \C^n$ coordinates $(u,w)$).
Our goal in this section is to give a quantitative version of
this result which gives $\Phi_4$ the optimal regularity (namely,
when $s_0\in (0,\infty)$, $\Phi_4$ is in $\ZygSpace{s_0+2}$,
and when $s_0=\omega$, $\Phi_4$ is real analytic).

As discussed in \cref{Section::Intro::MainMotivation}, for future applications we need keep track of what the constants depend on in this section, and need to make the statement of the results more precise than would be required just
for the main results of this paper.  To ease notation, we introduce notions of ``admissible'' constants.
These are constants which only depend on certain parameters.  The use of these constants greatly simplifies notation in both the statements of the results and the proofs.


\begin{defn}
If $s_0\in (0,\infty)$, for $s\geq s_0$ if we say $C$ is an
$\Zygad{s}$-admissible constant, it means that we assume
$X_k, L_j\in \ZygSpace{s+1}[B_{\R^r\times \C^n}(1)][\C^{r+2n}]$, $\forall j,k$.  $C$ can then be chosen to depend only on
$n$, $r$, $s$, $s_0$, and upper bounds for
$\ZygNorm{X_k}{s+1}[B_{\R^r\times \C^n}(1)]$ and
$\ZygNorm{L_j}{s+1}[B_{\R^r\times \C^n}(1)]$, $1\leq k\leq r$, $1\leq j\leq n$.
For $s\leq s_0$, we define $\Zygad{s}$-admissible constants
to be $\Zygad{s_0}$-admissible constants.
\end{defn}

\begin{defn}
If $s_0=\omega$, we say $C$ is an $\Zygad{\omega}$-admissible constant
if $C$ can be chosen to depend only on $n$, $r$, and upper bounds
for $\ANorm{X_k}{2n+r}{1}$, $\ANorm{L_j}{2n+r}{1}$, $1\leq k\leq r$, $1\leq j\leq n$.
\end{defn}

We write $A\lesssim_{\Zygad{s}} B $ to mean $A\leq CB$, where $C$ is a positive $\Zygad{s}$-admissible constant.  We write $A\approx_{\Zygad{s}} B$ for $A\lesssim_{\Zygad{s}} B$ and $B\lesssim_{\Zygad{s}} A$.

\begin{thm}\label{Thm::Nirenberg::MainThm}
There exists an $\Zygad{s_0}$-admissible constant $K_2\geq 1$
and a map $\Phi_4:B_{\R^r\times \C^n}(1)\rightarrow B_{\R^r\times \C^n}(1)$ such that
\begin{enumerate}[(i)]
\item\label{Item::Nirenberg::PhiRegularity}
\begin{itemize}
\item If $s_0\in (0,\infty)$, $\Phi_4\in \ZygSpace{s_0+2}[B_{\R^r\times \C^n}(1)][\R^r\times \C^n]$ and
    $\ZygNorm{\Phi_4}{s+2}[B_{\R^r\times \C^n}(1)]\lesssim_{\Zygad{s}} 1$, $\forall s>0$.
\item If $s_0=\omega$, $\Phi_4\in \ASpace{2n+r}{2}[\R^r\times \C^n]$
and $\ANorm{\Phi_4}{2n+r}{2}\leq 1$.  In particular, $\Phi_4$ extends to a real analytic function on $B_{\R^r\times \C^n}(2)$.
\end{itemize}
\item\label{Item::Nirenberg::Phiof0} $\Phi_4(0)=0$ and $d_{(t,x)} \Phi_4(0) = K_2^{-1} I_{(r+2n)\times (r+2n)}$.  See \cref{Section::Notation} for the notation $d_{(t,x)}$.

\item\label{Item::Nirenberg::JacobianConst} $\forall \zeta\in B_{\R^r\times\C^n}(1)$, $\det d_{(t,x)} \Phi_4 (\zeta)\approx_{\Zygad{s_0}} 1$.

\item\label{Item::Nirenberg::Phi4Open} $\Phi_4(B_{\R^r\times \C^n}(1))\subseteq B_{\R^r\times \C^n}(1)$ is an open set and $\Phi_4:B_{\R^r\times \C^n}(1)\rightarrow \Phi_4(B_{\R^r\times \C^n}(1))$ is a diffeomorphism\footnote{Here, and in the rest of the paper, we say $F:U_1\rightarrow U_2$
is a diffeomorphism if $F$ is a bijection and $dF$ is everywhere nonsingular.}.
\item\label{Item::Nirenberg::AMatrix}
\begin{equation*}
\begin{bmatrix}
\diff{u} \\ \diff{\wb}
\end{bmatrix}
= K_2^{-1}(I+\AMatrix)
\begin{bmatrix}
\Phi_4^{*} X \\ \Phi_4^{*} L
\end{bmatrix},
\end{equation*}
where $\AMatrix:B_{\R^r\times\C^n}(1)\rightarrow \M^{(n+r)\times (n+r)}(\C)$, $\AMatrix(0)=0$ and
\begin{itemize}
    \item If $s_0\in (0,\infty)$, $\ZygNorm{\AMatrix}{s+1}[B_{\R^r\times\C^n}(1)][\M^{(n+r)\times (n+r)}]\lesssim_{\Zygad{s}} 1$, $\forall s>0$ and $$\ZygNorm{\AMatrix}{s_0+1}[B_{\R^r\times\C^n}(1)][\M^{(n+r)\times (n+r)}]\leq \frac{1}{4}.$$
    \item If $s_0=\omega$, $\ANorm{\AMatrix}{2n+r}{1}[\M^{(n+r)\times (n+r)}]\leq \frac{1}{4}$.
\end{itemize}
In either case, note that this implies $(I+\AMatrix)$ is an invertible matrix on $B_{\R^r\times\C^n}(1)$.
\item\label{Item::Nirenberg::MainThm::ZEst} Suppose $Z$ is another complex vector field on $B_{\R^r\times\C^n}(1)$.  Then,
    \begin{itemize}
        \item If $s_0\in (0,\infty)$,
        $\ZygNorm{\Phi_4^{*} Z}{s+1}[B_{\R^r\times\C^n}(1)]\lesssim_{\Zygad{s}} \ZygNorm{Z}{s+1}[B_{\R^r\times\C^n}(1)]$, $\forall s>0$.
        \item If $s_0=\omega$,
        $\ANorm{\Phi_4^{*} Z}{2n+r}{1}\lesssim_{\Zygad{\omega}} \ANorm{Z}{2n+r}{1}$.
    \end{itemize}
\end{enumerate}
\end{thm}

\begin{rmk}
In \cref{Thm::Nirenberg::MainThm} (and in the rest of this section), we have written $s>0$ to mean $s\in (0,\infty)$ and similarly for other such inequalities.
For example if $s_0\in (0,\infty)$ and we write $s\geq s_0$, it means $s\in [s_0,\infty)$.
\end{rmk}

\begin{rmk}\label{Rmk::Nirenberg::ReductionCosts}
Proofs of results like \cref{Thm::Nirenberg::MainThm} in the literature
only prove that $\Phi_4$ is $\ZygSpace{s_0+1}$ (instead of $\ZygSpace{s_0+2}$); and
each of the estimates is similarly off by a derivative.\footnote{However, the results in this section concerning real analytic vector fields are standard.}
\end{rmk} 

\begin{rmk}
When $s_0=\omega$, the hypothesis $X_k, L_j\in \ASpace{r+2n}{1}[\C^{r+2n}]$ can be replaced with the slightly weaker hypothesis 
$X_k, L_j\in \sBSpace{r+2n}{r+2n}{1}$; one can achieve the same result with the same proof.  However, our applications use $X_k, L_j\in \ASpace{r+2n}{1}[\C^{r+2n}]$,
so we use this space instead.  In any case, it is straightforward to see (using \cref{Lemma::FuncSpaceRevAnal::AInB,Lemma::FuncSpaceReAnal::AinB}) that either choice yields an equivalent theorem.
\end{rmk}

    \subsection{A Reduction}
To prove \cref{Thm::Nirenberg::MainThm}, we prove the following
proposition.  For it we use the same notation and setting
as \cref{Thm::Nirenberg::MainThm}.

\begin{prop}\label{Prop::Nirenberg::MainProp}
There exist $\Zygad{s_0}$-admissible constants $K_1\geq 1$
and $\eta_3\in (0,1]$ and a map $\Phi_3:B_{\R^r\times \C^n}(\eta_3)\rightarrow B_{\R^r\times \C^n}(1)$ such that:
\begin{enumerate}[(i)]
\item\label{Item::Nirenberg::Phi3Reg}
\begin{itemize}
    \item If $s_0\in (0,\infty)$, $\Phi_3\in \ZygSpace{s_0+2}[B_{\R^r\times \C^n}(\eta_3)][\R^r\times \C^n]$ and $\ZygNorm{\Phi_3}{s+2}[B_{\R^r\times \C^n}(\eta_3)]\lesssim_{\Zygad{s}} 1$, $\forall s>0$.
    \item If $s_0=\omega$, $\Phi_3\in \ASpace{2n+r}{2\eta_3}[\R^r\times\C^n]$ and
    $\ANorm{\Phi_3}{2n+r}{2\eta_3}\leq 1$.  In particular, $\Phi_3$ extends to a real analytic function on $B_{\R^r\times \C^n}(2\eta_3)$.
\end{itemize}
\item\label{Item::Nirenberg::MainProp::0anddet} $\Phi_3(0)=0$ and $d_{(t,x)}\Phi_3(0)=K_1^{-1} I_{(r+2n)\times (r+2n)}$.
\item\label{Item::Nirenberg::MainProp::Diffeo} $\Phi_3(B_{\R^r\times \C^n}(\eta_3))\subseteq B_{\R^r\times \C^n}(1)$ is open
and $\Phi_3:B_{\R^r\times \C^n}(\eta_3)\rightarrow \Phi_3(B_{\R^r\times \C^n}(\eta_3))$ is a diffeomorphism.

\item\label{Item::Niremberg::MainProp::AEst}
\begin{equation*}
\begin{bmatrix}
\diff{u} \\ \diff{\wb}
\end{bmatrix}
= K_1^{-1}(I+\AMatrix_3)
\begin{bmatrix}
\Phi_3^{*} X \\ \Phi_3^{*} L
\end{bmatrix},
\end{equation*}
where $\AMatrix_3:B_{\R^r\times\C^n}(\eta_3)\rightarrow \M^{(n+r)\times (n+r)}(\C)$, $\AMatrix_3(0)=0$ and
\begin{itemize}
    \item If $s_0\in (0,\infty)$, $\ZygNorm{\AMatrix_3}{s+1}[B_{\R^r\times\C^n}(\eta_3)]\lesssim_{\Zygad{s}} 1$, $\forall s>0$.
    \item If $s_0=\omega$, $\ANorm{\AMatrix_3}{n}{\eta_3}\lesssim_{\Zygad{\omega}} 1$.
\end{itemize}

\end{enumerate}
\end{prop}

First we see how \cref{Thm::Nirenberg::MainThm} follows from
\cref{Prop::Nirenberg::MainProp}.

\begin{proof}[Proof of \cref{Thm::Nirenberg::MainThm}]
Let $\Phi_3$, $K_1$, $\eta_3$, and $\AMatrix_3$ be
as in \cref{Prop::Nirenberg::MainProp}.
It follows from \cref{Prop::Nirenberg::MainProp} \cref{Item::Nirenberg::MainProp::0anddet}
that
$\det d_{(t,x)}\Phi_3(0)\approx_{\Zygad{s_0}} 1$.
Next we claim that if $\etah\approx_{\Zygad{s_0}} 1$ is chosen sufficiently small (with $\etah\leq \eta_3$), then $\det d_{(t,x)}\Phi_3(\zeta)\approx_{\Zygad{s_0}} 1$, $\forall \zeta\in B_{\R^r\times \C^n}(\etah)$.  Indeed:
\begin{itemize}
\item Suppose $s_0\in (0,\infty)$.  Note $\CjNorm{\Phi_3}{2}[B_{\R^r\times \C^n}(\eta_3)]\leq \ZygNorm{\Phi_3}{s_0+2}[B_{\R^r\times \C^n}(\eta_3)]\lesssim_{\Zygad{s_0}} 1$.
Using the fact that $\det d_{(t,x)}\Phi_3(0)\approx_{\Zygad{s_0}} 1$,
if $\etah\approx_{\Zygad{s_0}} 1$ is chosen sufficiently
small (with $\etah\leq \eta_3$), we have
$\det d_{(t,x)} \Phi_3(\zeta) \approx_{\Zygad{s_0}} 1$, $\forall
\zeta\in B_{\R^r\times \C^n}(\etah)$.

\item Suppose $s_0=\omega$.  By \cref{Lemma::FuncSpaceRevAnal::AInB,Lemma::FuncSpaceReAnal::AinB}, $\CjNorm{\Phi_3}{2}[B_{\R^r\times \C^n}(\eta_3)] \lesssim_{\Zygad{s_0}} \sBNorm{\Phi_3}{2n+r}{2n+r}{2\eta_3}\leq \ANorm{\Phi_3}{2n+r}{\eta_3}\lesssim_{\Zygad{s_0}} 1$.  Thus, using the fact that $\det d_{(t,x)}\Phi_3(0)\approx_{\Zygad{s_0}} 1$, if $\etah\approx_{\Zygad{s_0}} 1$ is chosen sufficiently
small (with $\etah\leq \eta_3$), we have
$\det d_{(t,x)} \Phi_3(\zeta) \approx_{\Zygad{s_0}} 1$, $\forall
\zeta\in B_{\R^r\times \C^n}(\etah)$.
\end{itemize}

For $\gamma\leq \etah$, set $\Psi_\gamma:B_{\R^r\times \C^n}(1)\rightarrow B_{\R^r\times \C^n}(\etah)$ by $\Psi_\gamma(\zeta)=\gamma\zeta$.
We will set $\Phi_4:=\Phi_3\circ \Psi_\gamma$ for appropriately
chosen $\gamma$.  Consider,
\begin{equation*}
\frac{1}{\gamma}\begin{bmatrix}
\diff{u}\\ \diff{\wb}
\end{bmatrix}
=\Psi_\gamma^{*}
\begin{bmatrix}
\diff{u}\\ \diff{\wb}
\end{bmatrix}
=\Psi_\gamma^{*} K_1^{-1} (I+\AMatrix_3)
\begin{bmatrix}
\Phi_3^{*} X \\ \Phi_3^{*} L
\end{bmatrix}
=K_1^{-1} (I+\AMatrix_3\circ \Psi_\gamma)
\begin{bmatrix}
(\Phi_3\circ \Psi_\gamma)^{*} X\\
(\Phi_3\circ \Psi_\gamma)^{*} L
\end{bmatrix}.
\end{equation*}
Since $\AMatrix_3(0)=0$, using \cref{Prop::Nirenberg::MainProp} \cref{Item::Niremberg::MainProp::AEst} and \cref{Lemma::FuncSpaceRev::Scale}, we have:
\begin{itemize}
\item If $s_0\in (0,\infty)$,
$\ZygNorm{\AMatrix_3\circ \Psi_\gamma}{s_0+1}[B_{\R^r\times \C^n}(1)]\lesssim_{\Zygad{s_0}} \gamma \ZygNorm{\AMatrix_3}{s_0+1}[B_{\R^r\times \C^n}(\eta_3)]\lesssim_{\Zygad{s_0}} \gamma$.
Thus, by taking $\gamma$ to be a sufficiently small $\Zygad{s_0}$-admissible
constant, we have
$\ZygNorm{\AMatrix_3\circ \Psi_\gamma}{s_0+1}[B_{\R^r\times \C^n}(1)]\leq \frac{1}{4}$.
\item If $s_0=\omega$, we have
$\ANorm{\AMatrix_3\circ \Psi_\gamma}{2n+r}{1}\leq \frac{\gamma}{\eta_3} \ANorm{\AMatrix_3}{2n+r}{\eta_3}\lesssim_{\Zygad{\omega}} \gamma$.  
Also, set $R(t,z):= d\Phi_3(t,z)- K_1^{-1} I_{(2n+r)\times (2n+r)}$, so that $R(0,0)=0$, and by \cref{Lemma::FuncSpaceRev::DerivOfAnal}, $R\in \ASpace{2n+r}{\eta_3}[\M^{(2n+r)\times (2n+r)}]$ and $\ANorm{R}{2n+r}{\eta_3}\lesssim_{\Zygad{\omega}} 1$.
We have $\ANorm{R\circ \Psi_\gamma}{2n+r}{1}\leq \frac{\gamma}{\eta_3}\ANorm{R}{2n+r}{\eta_3}\lesssim_{\Zygad{\omega}} \gamma$. 
Thus,
by taking $\gamma$ to be a sufficiently small $\Zygad{\omega}$-admissible
constant, we have $\ANorm{\AMatrix_3\circ \Psi_\gamma}{2n+r}{1}\leq \frac{1}{4}$ and $\ANorm{R\circ \Psi_\gamma}{2n+r}{1}\leq (2K_1)^{-1}$.
\end{itemize}
Taking $\gamma$ as above and setting $\Phi_4=\Phi_3\circ \Psi_\gamma$,
\cref{Thm::Nirenberg::MainThm} \cref{Item::Nirenberg::PhiRegularity}, \cref{Item::Nirenberg::Phiof0}, \cref{Item::Nirenberg::JacobianConst}, \cref{Item::Nirenberg::Phi4Open}, and \cref{Item::Nirenberg::AMatrix} follow
with $K_2=\gamma^{-1} K_1$ and $\AMatrix=\AMatrix_3\circ \Psi_\gamma$.
%

We turn to \cref{Item::Nirenberg::MainThm::ZEst}.  Recall, 
\begin{equation}\label{Eqn::Nirenberg::FormulaForPullBack}
\Phi_4^{*} Z (u,w) = d\Phi_4(u,w)^{-1} Z(\Phi_4(u,w)).
\end{equation}
If $s_0\in (0,\infty)$, we have from \cref{Item::Nirenberg::PhiRegularity} and \cref{Lemma::FuncSpaceRev::Composition} that $\ZygNorm{Z\circ \Phi_4}{s+1}[B_{\R^r\times \C^n}(1)]\lesssim_{\Zygad{s}} \ZygNorm{Z}{s+1}[B_{\R^r\times \C^n}(1)]$.  Also, by \cref{Item::Nirenberg::PhiRegularity}, \cref{Item::Nirenberg::JacobianConst}, and \cref{Rmk::FuncSpaceRev::InverseMatrix} we have
$\ZygNorm{(d\Phi_4)^{-1}}{s+1}[B_{\R^r\times \C^n}(1)]\lesssim_{\Zygad{s}} 1$.  Using these estimates, \cref{Eqn::Nirenberg::FormulaForPullBack}, and \cref{Prop::FuncSpaceRev::Algebra},
\cref{Item::Nirenberg::MainThm::ZEst} follows in the case $s_0\in (0,\infty)$.

If $s_0=\omega$, \cref{Item::Nirenberg::PhiRegularity} and \cref{Lemma::FuncSpaceRev::ComposeAnal} show $\ANorm{Z\circ \Phi_4}{2n+r}{1}\lesssim_{\Zygad{\omega}} \ANorm{Z}{2n+r}{1}$.
Letting $R$ be as above, we have $d\Phi_4 = \gamma K_1^{-1} (I + K_1 R\circ \Psi_{\gamma})$.  Since $\ANorm{K_1 R\circ \Psi_\gamma}{2n+r}{1}[\M^{(2n+r)\times (2n+r)}]\leq 1/2$, and since $\ASpace{2n+r}{1}[\M^{(2n+r)\times (2n+r)}]$ is a Banach Algebra
(\cref{Lemma::FuncSpaceRevAnal::BanachAlgebra}), it follows (by using the Neumann series for $(I + K_1 R\circ \Psi_{\gamma})^{-1}$) that $\ANorm{(d\Phi_4)^{-1}}{2n+r}{1}\leq 2 K_1\gamma^{-1}\lesssim_{\Zygad{\omega}} 1$.
Using these estimates, \cref{Eqn::Nirenberg::FormulaForPullBack}, and \cref{Prop::FuncSpaceRev::Algebra}, \cref{Item::Nirenberg::MainThm::ZEst} follows in the case $s_0=\omega$, completing the proof.
\end{proof}

%
%

We now turn to the proof of \cref{Prop::Nirenberg::MainProp}, which encompasses
the rest of  \cref{Section::Nirenbeg}.  We do this by presenting
a series of increasingly general versions of the proposition,
and reducing each to the previous; eventually culminating
with the full \cref{Prop::Nirenberg::MainProp}.
The outline of this proof is:
\begin{itemize}
\item In \cref{Section::Nirenberg::HoloFrob} we present a quantitative version of the holomorphic Frobenius theorem; this result is standard.
\item In \cref{Section::Nirenberg::RA} we prove the special case of \cref{Prop::Nirenberg::MainProp} when the vector fields are all assumed to be real analytic and commute.
We do this by reducing it to the holomorphic case.  This procedure is standard.
\item In \cref{Section::Nirenberg::Additional} we present an easily checkable special case of the real analytic setting using elliptic PDEs.  This is a generalization of part of Malgrange's approach \cite{MalgrangeSurLIntegbrabilite}.
\item In \cref{Section::Nirenberg::SmallError} we use elliptic PDEs to reduce the case of vector fields which are a small perturbation of $\diff{t}$ and $\diff{\zb}$ to the previous case.  This is a generalization of part of
Malgrange's approach \cite{MalgrangeSurLIntegbrabilite}.
\item In \cref{Section::Nirenberg::Commute} we use a simple scaling argument to study vector fields which might be a large perturbation of $\diff{t}$ and $\diff{\zb}$; we do this by reducing to the previous case.
\item In \cref{Section::Nirenberg::Final} we complete the proof by using some simple linear algebra.
\end{itemize}

\begin{rmk}\label{Rmk::Nirenberg::AdmissibleOkay}
In each subsection which follows we use notions of admissible constants which are specific to that section; we are explicit about what we mean in each subsection.  In each subsection, we prove progressively stronger results, eventually
culminating in the proof of \cref{Prop::Nirenberg::MainProp}; we do this by reducing each result to the weaker results which proceed it.  The admissible constants in each result are defined so that constants which are admissible
in the result we are proving are admissible in the weaker results which we reduce it to.  So that, for example, the main result in \cref{Section::Nirenberg::SmallError} is reduced to the main result in \cref{Section::Nirenberg::Additional};
and in this application of the main result in \cref{Section::Nirenberg::Additional}, each constant which is admissible in the sense of \cref{Section::Nirenberg::Additional} is admissible in the sense of \cref{Section::Nirenberg::SmallError}.
Thus the various notions of admissible constants seamlessly glue together to yield  \cref{Prop::Nirenberg::MainProp}.
\end{rmk}

    \subsection{The Holomorphic Frobenius Theorem}\label{Section::Nirenberg::HoloFrob}
Fix $\eta_0>0$.  In this section we work on $\C^r\times \C^{2n}$
with complex coordinates $(\sigma,\zeta)=(\sigma_1,\ldots, \sigma_r,\zeta_1,\ldots, \zeta_{2n})$.
We are given holomorphic vector fields:
\begin{equation*}
X_k = \diff{\sigma_k} +\sum_{l=1}^{2n} b_{1,k}^l(\sigma,\zeta) \diff{\zeta_l} + \sum_{l=1}^{r} b_{2,k}^l(\sigma,\zeta) \diff{\sigma_l}, \quad 1\leq k\leq r,
\end{equation*}
\begin{equation*}
L_j = \frac{1}{2} \left( \diff{\zeta_j} + i \diff{\zeta_{j+n}} \right) + \sum_{l=1}^{2n} b_{3,j}^l(\sigma,\zeta) \diff{\zeta_l} + \sum_{l=1}^r b_{4,j}^l(\sigma,\zeta)\diff{\sigma_l}, \quad 1\leq j\leq n,
\end{equation*}
where $b_{c,d}^e\in \OSpace{1}[B_{\C^{r+2n}}(\eta_0)]$, $\forall c,d,e$ (see \cref{Section::FuncRev::RealAnal} for the definition of the space $\OSpace{1}[B_{\C^{r+2n}}(\eta_0)]$).
We also assume $b_{c,d}^e(0,0)=0$, $\forall c,d,e$.

We assume $[L_{j_1},L_{j_2}]=0$, $[L_j, X_k]=0$, $[X_{k_1},X_{k_2}]=0$,
$\forall j_1, j_2, k_1, k_2, j, k$.
Take $C_1$ so that $\ONorm{b_{c,d}^e}{1}[B_{\C^{r+2n}}(\eta_0)]\leq C_1$, $\forall c,d,e$.

\begin{defn}
We say $C$ is an admissible constant if $C$ can be chosen
to depend only on $\eta_0$, $n$, $r$, and $C_1$.
\end{defn}

We write $A\lesssim B$ for $A\leq CB$, where $C$ is an admissible constant.  We write $A\approx B$ for $A\lesssim B$ and $B\lesssim A$.\footnote{We use similar notation
in the following sections without explicitly defining it.}  

\begin{prop}\label{Prop::NirenbergHolo::MainProp}
There exists an admissible constant $\eta_1>0$ and $w_1,\ldots, w_n\in \OSpace{1}[B_{\C^{r+2n}}(\eta_1)]$ such that:
\begin{itemize}
\item $w_l(0)=0$ and $d w_l(0) = d\zeta_l+id\zeta_{l+n}$.
\item $\ONorm{w_l}{1}[B_{\C^{r+2n}}(\eta_1)]\lesssim 1$, $\forall l$.
\item $L_j w_l=0$, $X_k w_l=0$, $\forall j,k,l$.
\end{itemize}
\end{prop}

In what follows, we use the exponentiation of holomorphic vector fields.
So that if $V$ is a holomorphic vector field on an open set $\Omega\subseteq \C^N$,
it makes sense to define $(t,z)\mapsto e^{tV}z$, for
$z\in \Omega$ and
$t$ in a neighborhood
of $0\in \C$ (depending on $z$).  If $\Omega'\Subset\Omega$ is a relatively
compact open set, then the map $e^{tV}z$ exists for $z\in \Omega'$
and $t\in B_{\C}(\delta)$, where $\delta$ can be chosen to depend only on upper bounds for $\mathrm{dist}(\Omega', \partial \Omega)^{-1}$,
$\ONorm{Z}{N}[\Omega]$, and $N$.  Furthermore,
$\ONorm{e^{tZ}z-z}{N}[B_{\C}(\delta)\times \Omega']$ can be bounded
in terms of upper bounds for $\ONorm{Z}{N}[\Omega]$, $\delta$, and $N$.  This is all proved
using the standard Contraction Mapping Principle argument.  See Chapter I, Section 1 of \cite{IlyashenkoYakovenkoLecturesOnAnalyticDifferentialEquations} for a proof of this
standard fact.

\begin{proof}[Proof of \cref{Prop::NirenbergHolo::MainProp}]
Let $Z_1,\ldots, Z_n$ be given by $Z_j=\frac{1}{2}\left(\diff{\zeta_j} -  i\diff{\zeta_{j+n}}\right)$, and set
\begin{equation*}
\Psi(t_1,\ldots, t_r, u_1,\ldots, u_n, v_1,\ldots, v_n):=e^{t_1 X_1} e^{t_2 X_2} \cdots e^{t_r X_r} e^{u_1 L_1} e^{u_2 L_2} \cdots e^{u_n L_n} e^{v_1 Z_1} e^{v_2 Z_2}\cdots e^{v_n Z_n} 0.
\end{equation*}
By the above discussion, there exists an admissible constant $\eta'>0$
with $\Psi\in \OSpace{r+2n}[B_{\C^{r+2n}}(\eta')]$
and $\ONorm{\Psi}{r+2n}[B_{\C^{r+2n}}(\eta')]\lesssim 1$.

Since $\diff{t_k}\big|_{t=0,u=0,v=0} \Psi(t,u,v) = X_k(0) = \diff{\sigma_k}$,
$\diff{u_j}\big|_{t=0,u=0,v=0} \Psi(t,u,v) =L_j(0) = \frac{1}{2} \left( \diff{\zeta_j} + i \diff{\zeta_{j+n}} \right)$, and
$\diff{v_j}\big|_{t=0,u=0,v=0} \Psi(t,u,v) = Z_j(0) = \frac{1}{2}\left(\diff{\zeta_j} - i\diff{\zeta_{j+n}}\right)$,
we have (where $I_{a\times a}$ denotes the $a\times a$ idenitity matrix
and $0_{a\times b}$ denotes the $a\times b$ zero matrix):
\begin{equation*}
d_{t,u,v} \Psi(0,0,0)
=
\mleft[\def\arraystretch{1.5}
\begin{array}{c|c|c}
I_{r\times r} & 0_{r\times n} & 0_{r\times n}\\
\hline
0_{n\times r} & \frac{1}{2} I_{n\times n} & \frac{1}{2} I_{n\times n}\\
\hline
0_{n\times r} & \frac{i}{2} I_{n\times n} & -\frac{i}{2} I_{n\times n}
\end{array}
\mright].
\end{equation*}
In particular, $d_{t,u,v} \Psi(0,0,0)$ is invertible and
\begin{equation}\label{Eqn::NirenbergHolo::dPsiinv}
\mleft( d_{t,u,v} \Psi(0,0,0) \mright)^{-1} =
\mleft[\def\arraystretch{1.5}
\begin{array}{c|c|c}
I_{r\times r} & 0_{r\times n} & 0_{r\times n}\\
\hline
0_{n\times r} & I_{n\times n} & -i I_{n\times n}\\
\hline
0_{n\times r} & I_{n\times n} & i I_{n\times n}
\end{array}
\mright].
\end{equation}
Since $\Psi(0,0,0)=(0,0,0)$, the holomorphic Inverse Function Theorem
applies to show that there exist admissible constants
$\eta'',\eta_1>0$ such that
\begin{equation*}
\Psi : B_{\C^{r+2n}}(\eta'')\rightarrow \Psi(B_{\C^{r+2n}}(\eta''))
\end{equation*}
is a biholomorphism,
$B_{\C^{r+2n}}(\eta_1)\subseteq \Psi(B_{\C^{r+2n}}(\eta''))$,
and $\ONorm{\Psi^{-1}}{2n+r}[B_{\C^{r+2n}}(\eta_1)]\lesssim 1$.

We give $B_{\C^{r+2n}}(\eta'')$ holomorphic coordiantes
$(t_1,\ldots, t_r, u_1,\ldots, u_n, v_1\ldots, v_n)$.
Set $V_j(t_1,\ldots, t_r, u_1,\ldots, u_n, v_1,\ldots, v_n) = v_j$.
Define, for $1\leq j\leq n$, $w_j\in \OSpace{1}[B_{\C^{r+2n}}(\eta_1)]$
by
\begin{equation*}
w_j(\sigma, \zeta):=V_j\circ \Psi^{-1}(\sigma, \zeta).
\end{equation*}
Note $\ONorm{w_j}{1}[B_{\C^{r+2n}}(\eta_1)]\lesssim 1$ and $w_j(0)=0$.

Because the $X_k$s and $L_j$s commute, we have
$\Psi_{*} \diff{t_k} =X_k$ and $\Psi_{*} \diff{u_j} =L_j$.
Thus, since $\diff{t_k} V_l=0$ and $\diff{u_j} V_l=0$, we have
$X_k w_l=0$ and $L_j w_l=0$.

Finally, we compute $dw_j(0) = dV_j(0)d\Psi^{-1}(0)$.
$d V_j(0)$ is the row vector which has $1$ in the $r+n+j$ component
and $0$ in all other components and $d\Psi^{-1}(0)$ is given
in \cref{Eqn::NirenbergHolo::dPsiinv}.  Thus,
$dw_j(0)$ is the vector which equals $1$ in the $r+j$ component,
$i$ in the $r+n+j$ component, and $0$ in all other components.
I.e., $dw_j(0) = d\zeta_j+ i d\zeta_{j+n}$.
\end{proof} 

    \subsection{Real Analytic Vector Fields}\label{Section::Nirenberg::RA}
Fix $\eta_0>0$.  Let $X_1,\ldots, X_r, L_1,\ldots, L_n$
be real analytic vector fields on $\R^r\times \C^n\cong \R^{r+2n}$
of the form
\begin{equation*}
X_k = \diff{t_k} + A_k \diff{t} + B_k \diff{z} + E_k \diff{\zb}, \quad 1\leq k\leq r,
\end{equation*}
\begin{equation*}
L_j = \diff{\zb[j]} + C_j \diff{t} + D_j \diff{z} + F_j \diff{\zb}, \quad 1\leq j\leq n.
\end{equation*}
Here we are thinking of $A_k$, $B_k$, $C_j$, $D_j$, $E_k$, and $F_j$ as
real analytic row vectors
vectors:  $A_k, C_j\in \sBSpace{r+2n}{r}{\eta_0}$, $B_k, D_j, E_k, F_j\in \sBSpace{r+2n}{n}{\eta_0}$ (see \cref{Section::FuncRev::RealAnal} for the definition of $\sBSpace{r+2n}{\cdot}{\eta_0}$).
We assume $A_k(0)=0$, $B_k(0)=0$, $C_j(0)=0$, $D_j(0)=0$, $E_k(0)=0$, and $F_j(0)=0$,
and we assume the $X$s and $L$s all commute:
$[X_{k_1},X_{k_2}]=0$, $[X_k, L_j]=0$, $[L_{j_1},L_{j_2}]=0$,
$\forall j_1, j_2, k_1, k_2, j, k$.

\begin{defn}
We say $K$ is an admissible constant if $K$ can be chosen to depend only
on $\eta_0$, $n$, $r$, and upper bounds
for $\sBNorm{A_k}{r+2n}{r}{\eta_0}$, $\sBNorm{B_k}{r+2n}{n}{\eta_0}$, $\sBNorm{C_j}{r+2n}{r}{\eta_0}$,  $\sBNorm{D_j}{r+2n}{n}{\eta_0}$, $\sBNorm{E_k}{r+2n}{n}{\eta_0}$, and $\sBNorm{F_j}{r+2n}{n}{\eta_0}$, $\forall j,k$.
\end{defn}

\begin{prop}\label{Prop::NirenbergRA::MainProp}
There exists an admissible constant $\eta_2>0$ and a map
\begin{equation*}
\Phi_1 : B_{\R^r\times \C^n}(\eta_2)\rightarrow B_{\R^r\times \C^n}(\eta_0)
\end{equation*}
such that:
\begin{itemize}
\item $\Phi_1\in \sBSpace{r+2n}{r+2n}{\eta_2}$ with $\sBNorm{\Phi_1}{r+2n}{r+2n}{\eta_2}\lesssim 1$.
\item $\Phi_1(B_{\R^r\times \C^n}(\eta_2))\subseteq B_{\R^r\times \C^n}(\eta_0)$ is open and $\Phi_1:B_{\R^r\times \C^n}(\eta_2)\rightarrow \Phi_1(B_{\R^r\times \C^n}(\eta_2))$ is a real analytic diffeomorphism.
\item $\Phi_1(0)=0$ and $d_{t,x} \Phi_1 (0) = I_{(r+2n)\times (r+2n)}$.
\item
\begin{equation*}
\begin{bmatrix}
\diff{u} \\ \diff{\wb}
\end{bmatrix}
= (I+\AMatrix_1)
\begin{bmatrix}
\Phi_1^{*} X \\ \Phi_1^{*} L
\end{bmatrix},
\end{equation*}
where $\AMatrix_1(0)=0$,
$\AMatrix_1\in \sBSpace{r+2n}{\M^{(n+r)\times (n+r)}}{\eta_2}$,
and $\sBNorm{\AMatrix_1}{r+2n}{\M^{(n+r)\times (n+r)}}{\eta_2}\leq 1$.
\end{itemize}
\end{prop}

To prove \cref{Prop::NirenbergRA::MainProp} we start with a conditional lemma.

\begin{lemma}\label{Lemma::NirenbergRA::MainLemma}
We take the same setting as \cref{Prop::NirenbergRA::MainProp}.
Suppose there is an admissible constant $\eta_1>0$ and functions
$w_1,\ldots, w_n\in \sBSpace{r+2n}{1}{\eta_1}$ such that:
$w_l(0)=0$, $dw_l(0)=dz_l$, $\sBNorm{w_l}{r+2n}{1}{\eta_1}\lesssim 1$,
and $L_j w_l=0$, $X_k w_l=0$, $\forall j,k,l$.  Then,
the conclusions of \cref{Prop::NirenbergRA::MainProp} hold.
\end{lemma}
\begin{proof}
We define $\Psi:B_{\R^r\times \C^n}(\eta_1)\rightarrow \R^r\times \C^n$
by
\begin{equation*}
\Psi(t,z) = (t, w_1(t,z),\ldots, w_n(t,z)).
\end{equation*}
I.e., by identifying $\R^{2n}\cong \C^n$ via the map
$(x_1,\ldots, x_{2n})\mapsto (x_1+ix_{n+1},\ldots, x_n+ix_{2n})$,
we have
\begin{equation*}
\Psi(t,x) = (t, \Real(w_1)(t,x),\ldots, \Real(w_n)(t,x), \Imag(w_1)(t,x),\ldots, \Imag(w_n)(t,x)).
\end{equation*}
Note that $\Psi(0,0)=0$.
Since $dw_j(0)=dz_j$ it follows that
$d_{t,x} \Psi(0)= I_{(r+2n)\times (r+2n)}$.
Thus, the Inverse Function Theorem applies to $\Psi$
to show that there exists an admissible constants $\eta',\eta''>0$
such that $\Psi:B_{\R^r\times\C^n}(\eta')\rightarrow \Psi(B_{\R^r\times \C^n}(\eta'))$ is a real analytic diffeomorphism,
$B_{\R^r\times \C^n}(\eta'')\subseteq \Psi(B_{\R^r\times \C^n}(\eta'))$,
and $\Phi_1:=\Psi^{-1}:B_{\R^r\times \C^n}(\eta'')\rightarrow B_{\R^r\times \C^n}(\eta')$ satisfies $\Phi_1\in \sBSpace{r+2n}{r+2n}{\eta''}$ with $\sBNorm{\Phi_1}{r+2n}{r+2n}{\eta''}\lesssim 1$.

Using coordinates $(u,w)$ on $\R^r\times \C^n$, since
$L_j w_l=0$ and $X_kw_l=0$, $\forall j,k,l$, we have
\begin{equation*}
\Phi_1^{*} X_k(u,w), \Phi_1^{*} L_j(u,w)\in \Span_{\C}\mleft\{ \diff{u_1},\ldots, \diff{u_r}, \diff{\wb[1]},\ldots, \diff{\wb[n]} \mright\}, \quad \forall (u,w)\in B_{\R^r\times \C^n}(\eta'').
\end{equation*}
Since $d_{t,x}\Phi_1(0)=d_{t,x}\Psi(0)^{-1} = I_{(r+2n)\times (r+2n)}$
and $X_k(0)=\diff{t_k}$, $L_j(0)=\diff{\wb[j]}$ we have
\begin{equation*}
\begin{bmatrix}
\Phi_1^{*} X \\ \Phi_1^{*} L
\end{bmatrix}
=(I+M) \begin{bmatrix}
\diff{u}\\ \diff{\wb}
\end{bmatrix},
\end{equation*}
where $M$ is a real analytic matrix, $M(0)=0$, and
$\sBNorm{M}{2n+r}{\M^{(n+r)\times (n+r)}}{\eta''/2}\lesssim 1$ (here $\eta''/2$ can be replaced with any fixed number in $(0,\eta'')$).
Since $M(0)=0$, for $\eta_2\in (0,\eta''/2]$ we have, using \cref{Lemma::FuncSpaceReAnal::Restrict},
\begin{equation*}
\sBNorm{M}{2n+r}{\M^{(n+r)\times (n+r)}}{\eta_2}\lesssim \eta_2.
\end{equation*}
By taking $\eta_2>0$ to be a sufficiently small admissible constant,
we have
\begin{equation*}
\sBNorm{M}{2n+r}{\M^{(n+r)\times (n+r)}}{\eta_2}\leq \frac{1}{2}.
\end{equation*}
We define $I+\AMatrix_1:=(I+M)^{-1}$.  Then we have
$\AMatrix_1(0)=0$,
$\sBNorm{\AMatrix_1}{2n+r}{\M^{(n+r)\times (n+r)}}{\eta_2}\leq 1$ (since $\sBSpace{2n+r}{\M^{(n+r)\times (n+r)}}{\eta_2}$ is a Banach algebra and we have used the Neumann series for $(1+M)^{-1}$),
and
\begin{equation*}
\begin{bmatrix}
\diff{u}\\ \diff{\wb}
\end{bmatrix}
=(I+\AMatrix_1) \begin{bmatrix}
\Phi_1^{*} X \\ \Phi_1^{*} L
\end{bmatrix},
\end{equation*}
as desired, completing the proof.
\end{proof}

\begin{proof}[Proof of \cref{Prop::NirenbergRA::MainProp}]
We need to show that there exist functions $w_1,\ldots, w_n$
as in \cref{Lemma::NirenbergRA::MainLemma}.
By the definition of $\sBSpace{r+2n}{\cdot}{\eta_0}$,  the functions $A_k$, $B_k$, $C_j$, $D_j$, $E_k$, and $F_j$ extend to holomorphic
functions
$\Extend(A_k), \Extend(C_j)\in \OSpace{r}[B_{\C^{r+2n}}(\eta_0)]$,
$\Extend(B_k), \Extend(D_j), \Extend(E_k), \Extend(F_j)\in \OSpace{n}[B_{\C^{r+2n}}(\eta_0)]$,
with
$$\ONorm{\Extend(A_k)}{r},\ONorm{\Extend(C_j)}{r},\ONorm{\Extend(B_k)}{n},\ONorm{\Extend(D_j)}{n},\ONorm{\Extend(E_k)}{n},\ONorm{\Extend(F_j)}{n}\lesssim 1.$$

We give $\C^r\times \C^{2n}$ coordinates $(\sigma, \zeta)$.
Let
$$\diff{\zeta_{\cdot}}:=\begin{bmatrix}
\diff{\zeta_1}\\
\vdots\\
\diff{\zeta_n}
\end{bmatrix}, \quad
\diff{\zeta_{\cdot+n}}:=\begin{bmatrix}
\diff{\zeta_{n+1}}\\
\vdots\\
\diff{\zeta_{2n}}
\end{bmatrix}. $$
We extend $X_k$ and $L_j$ to holomorphic vector fields on $\C^{r}\times \C^{2n}$,
 by setting
\begin{equation*}
\Extend(X_k) = \diff{\sigma_j} + \Extend(A_k) \diff{\sigma} + \Extend(B_k) \frac{1}{2} \mleft( \diff{\zeta_{\cdot}} -i \diff{\zeta_{\cdot+n}} \mright)  + \Extend(E_k) \frac{1}{2}\mleft( \diff{\zeta_j} + i \diff{\zeta_{j+n}} \mright) ,
\end{equation*}
\begin{equation*}
\Extend(L_j) = \frac{1}{2}\mleft( \diff{\zeta_j} + i \diff{\zeta_{j+n}} \mright) + \Extend(C_j) \diff{\sigma} + \Extend(D_j) \frac{1}{2} \mleft( \diff{\zeta_{\cdot}} -i \diff{\zeta_{\cdot+n}} \mright)  + \Extend(F_j) \frac{1}{2}\mleft( \diff{\zeta_j} + i \diff{\zeta_{j+n}} \mright) .
\end{equation*}
I.e., we have extended each $t_k$ to the complex variable $\sigma_k$
and each $x_j$ to the complex variable $\zeta_j$.
Since the $X$s and $L$s commute, the same is true of the $\Extend(X)$s and $\Extend(L)$s by analytic continuation:
$[\Extend(X_{k_1}), \Extend(X_{k_2})]=0$, $[\Extend(L_{j_1}), \Extend(L_{j_2})]=0$, $[\Extend(X_k), \Extend(L_j)]=0$, $\forall k_1,k_2,j_1,j_1,j,k$.

\Cref{Prop::NirenbergHolo::MainProp} applies to $\Extend(X_1),\ldots, \Extend(X_r), \Extend(L_1),\ldots, \Extend(L_n)$ and each constant which is admissible in the sense of \cref{Prop::NirenbergHolo::MainProp} is admissible in the sense of this section.
This shows that there exists
an admissible constant $\eta_1>0$ and functions
$\wh_1,\ldots, \wh_n\in \OSpace{1}[B_{\C^{r+2n}}(\eta_1)]$,
with $\ONorm{\wh_l}{1}\lesssim 1$, $\wh_l(0)=0$, $d\wh_l(0) = d\zeta_l+id\zeta_{l+n}$, and
$\Extend(L_j) \wh_l=0$, $\Extend(X_k)\wh_l=0$, $\forall j,k,l$.

Define, for $(t,x)\in B_{\R^r\times \R^{2n}}(\eta_1)$,
\begin{equation*}
w_l(t_1,\ldots, t_r, x_1,\ldots, x_{2n}):= \wh_l(t_1+i0,\ldots, t_r+i0, x_1+i0,\ldots, x_{2n}+i0).
\end{equation*}
Note that $\wh_l$ is the analytic extension of $w_l$ and therefore
$\sBNorm{w_l}{r+2n}{1}{\eta_1}\lesssim 1$.
Also, $dw_l(0) = dx_l+idx_{l+n}= dz_l$, $w_l(0)=\wh_l(0)=0$.
Finally, since $\Extend(X_k)\wh_l=0$ and $\Extend(L_j) \wh_l=0$ we have
$X_k w_l=0$ and $L_j w_l=0$, $\forall j,k,l$.
Thus \cref{Lemma::NirenbergRA::MainLemma} applies, completing the proof.
\end{proof} 

    \subsection{Vector fields satisfying an additional equation}\label{Section::Nirenberg::Additional}
Fix $s_0\in (0,\infty)$.  We let $X_1,\ldots, X_r, L_1,\ldots, L_n$
be $\ZygSpace{s_0+1}$ complex vector fields on $B_{\R^r\times \C^n}(1)$
of the following form:
\begin{equation*}
X= \diff{t} + A\diff{t} + B\diff{z}+E\diff{\zb},
\quad
L=\diff{\zb}+C\diff{t}+D\diff{z}+ F\diff{\zb}.
\end{equation*}
Here we are using matrix notation; so that
$X$ is the column vector $\begin{bmatrix} X_1,\ldots, X_r\end{bmatrix}^{\transpose}$, $\diff{t}=\begin{bmatrix}\diff{t_1},\ldots, \diff{t_r}\end{bmatrix}^{\transpose}$,
similarly for $L$, $\diff{z}$, and $\diff{\zb}$,
and $A$, $B$, $C$,  $D$, $E$, and $F$ are matrices of the appropriate size.
Thus, if we let $A_k$ denote the $k$th row of $A$, and similarly for $B$, $C$, $D$, $E$, and $F$ we have
\begin{equation*}
X_k = \diff{t_k} + A_k \diff{t}+B_k\diff{z} + E_k \diff{\zb}, \quad L_j=\diff{\zb[j]} + C_j \diff{t} + D_j\diff{z} + F_j \diff{\zb}.
\end{equation*}
We assume:
\begin{itemize}
\item $A\in \ZygSpace{s_0+1}[B_{\R^r\times \C^n}(1)][\M^{r\times r}(\C)]$,
    $B, E\in \ZygSpace{s_0+1}[B_{\R^r\times \C^n}(1)][\M^{r\times n}(\C)]$,
    $C\in \ZygSpace{s_0+1}[B_{\R^r\times \C^n}(1)][\M^{n\times r}(\C)]$,
    $D, F\in \ZygSpace{s_0+1}[B_{\R^r\times \C^n}(1)][\M^{n\times n}(\C)]$.
\item $A(0)=0_{r\times r}$, $B(0)=0_{r\times n}$, $C(0)=0_{n\times r}$, $D(0)=0_{n\times n}$, $E(0)=0_{r\times n}$, and $F(0)=0_{n\times n}$.
\item The $X$s and $L$s commute: $[X_{k_1},X_{k_2}]=0$, $[L_{j_1},L_{j_2}]=0$, and $[X_k, L_j]=0$, $\forall j_1,j_2,k_1,k_2,j,k$.
\item
\begin{equation}\label{Eqn::NirenbergAdditional::Bil0}
\sum_{k=1}^r \frac{\partial A_k}{\partial t_k} +\sum_{j=1}^n \frac{\partial C_j}{\partial z_j}=0, \quad
\sum_{k=1}^r \frac{\partial B_k}{\partial t_k} +\sum_{j=1}^n \frac{\partial D_j}{\partial z_j}=0, \quad
\sum_{k=1}^r \frac{\partial E_k}{\partial t_k} +\sum_{j=1}^n \frac{\partial F_j}{\partial z_j}=0.
\end{equation}
\end{itemize}

\begin{defn}
We say $K$ is an admissible constant if $K$ can be chosen
to depend only on $n$, $r$, and $s_0$.
\end{defn}

\begin{prop}\label{Prop::NirenbergAddtional::MainProp}
There exists an admissible constant $\gamma>0$ such that if
\begin{equation*}
\begin{split}
&\ZygNorm{A}{s_0+1}[B_{\R^r\times \C^n}(1)], \ZygNorm{B}{s_0+1}[B_{\R^r\times \C^n}(1)], \ZygNorm{C}{s_0+1}[B_{\R^r\times \C^n}(1)],
\\&
 \ZygNorm{D}{s_0+1}[B_{\R^r\times \C^n}(1)], 
 \ZygNorm{E}{s_0+1}[B_{\R^r\times \C^n}(1)],  \ZygNorm{F}{s_0+1}[B_{\R^r\times \C^n}(1)]\leq \gamma,
\end{split}
\end{equation*} 
then there exists an admissible
constant $\eta_2>0$ and a map $\Phi_1:B_{\R^r\times \C^n}(\eta_2)\rightarrow B_{\R^r\times \C^n}(1)$ such that:
\begin{itemize}
\item $\Phi_1\in \sBSpace{r+2n}{r+2n}{\eta_2}$ with $\sBNorm{\Phi_1}{r+2n}{r+2n}{\eta_2}\lesssim 1$.
\item $\Phi_1(B_{\R^r\times \C^n}(\eta_2))\subseteq B_{\R^r\times \C^n}(1)$ is open and $\Phi_1:B_{\R^r\times \C^n}(\eta_2)\rightarrow \Phi_1(B_{\R^r\times \C^n}(\eta_2))$ is a real analytic diffeomorphism.
\item $\Phi_1(0)=0$ and $d_{t,x} \Phi_1 (0) = I_{(r+2n)\times (r+2n)}$.
\item
\begin{equation*}
\begin{bmatrix}
\diff{u} \\ \diff{\wb}
\end{bmatrix}
= (I+\AMatrix_1)
\begin{bmatrix}
\Phi_1^{*} X \\ \Phi_1^{*} L
\end{bmatrix},
\end{equation*}
where $\AMatrix_1(0)=0$,
$\AMatrix_1\in \sBSpace{r+2n}{\M^{(n+r)\times (n+r)}}{\eta_2}$,
and $\sBNorm{\AMatrix_1}{r+2n}{\M^{(n+r)\times (n+r)}}{\eta_2}\leq 1$.
\end{itemize}
\end{prop}
\begin{proof}
To prove the proposition, we will show that if $\gamma>0$ is a sufficiently small admissible constant, then $A$, $B$, $C$, $D$, $E$, and $F$ are real analytic, and there
exists an admissible constant $\eta_0>0$ such that
\begin{equation}\label{Eqn::NirernbergAdditional::ToShowRegEqn}
\sBNorm{A_k}{r+2n}{r}{\eta_0},\sBNorm{B_k}{r+2n}{n}{\eta_0}, \sBNorm{C_j}{r+2n}{r}{\eta_0},\sBNorm{D_j}{r+2n}{n}{\eta_0},\sBNorm{E_k}{r+2n}{n}{\eta_0}, \sBNorm{F_j}{r+2n}{n}{\eta_0}\lesssim 1,\quad \forall j,k.
\end{equation}
The result will then follow immediately from \cref{Prop::NirenbergRA::MainProp}.

The equation $[X_{k_1}, X_{k_2}]=0$ can be equivalently
rewritten as the following three equations:
\begin{equation}\label{Eqn::NirenbergAdditional::FirstA}
\frac{\partial A_{k_2}}{\partial t_{k_1}} - \frac{\partial A_{k_1}}{\partial t_{k_2}} = A_{k_2} \diff{t} A_{k_1} - A_{k_1} \diff{t} A_{k_2} + B_{k_2} \diff{z} A_{k_1}-B_{k_1} \diff{z}A_{k_2} +E_{k_2}\diff{\zb} A_{k_1} - E_{k_1} \diff{\zb} A_{k_2},
\end{equation}
\begin{equation}\label{Eqn::NirenbergAdditional::SecondA}
\frac{\partial B_{k_2}}{\partial t_{k_1}} - \frac{\partial B_{k_1}}{\partial t_{k_2}} = A_{k_2}\diff{t} B_{k_1} -A_{k_1} \diff{t} B_{k_2} + B_{k_2} \diff{z} B_{k_1} - B_{k_1} \diff{z} B_{k_2} +E_{k_2} \diff{\zb} B_{k_1} - E_{k_1} \diff{\zb} B_{k_2},
\end{equation}
\begin{equation}\label{Eqn::NirenbergAdditional::ThirdA}
\frac{\partial E_{k_2}}{\partial t_{k_1}} - \frac{\partial E_{k_1}}{\partial t_{k_2}} = A_{k_2} \diff{t} E_{k_1} - A_{k_1} \diff{t} E_{k_2} + B_{k_2}\diff{z} E_{k_1} - B_{k_1} \diff{z} E_{k_2} +E_{k_2}\diff{\zb} E_{k_1} -E_{k_1} \diff{\zb} E_{k_2}.
\end{equation}
We write \cref{Eqn::NirenbergAdditional::FirstA}, \cref{Eqn::NirenbergAdditional::SecondA}, and \cref{Eqn::NirenbergAdditional::ThirdA} as the following equation:
\begin{equation}\label{Eqn::NirenbergAdditional::Bil1}
\begin{split}
&\mleft( 
\mleft(
\frac{\partial A_{k_2}}{\partial t_{k_1}} - \frac{\partial A_{k_1}}{\partial t_{k_2}}
\mright)_{1\leq k_1<k_2\leq r},
\mleft(
\frac{\partial B_{k_2}}{\partial t_{k_1}} - \frac{\partial B_{k_1}}{\partial t_{k_2}}
\mright)_{1\leq k_1<k_2\leq r},
\mleft(
\frac{\partial E_{k_2}}{\partial t_{k_1}} - \frac{\partial E_{k_1}}{\partial t_{k_2}}
\mright)_{1\leq k_1<k_2\leq r}
 \mright) 
 \\&= \Gamma_1((A,B, E), \grad(A,B,E)),
 \end{split}
\end{equation}
where $\Gamma_1$ is an explicit constant coefficient bilinear form
depending only on $n$ and $r$.
Similarly, $[L_{j_1},L_{j_2}]=0$ can be written as:
\begin{equation}\label{Eqn::NirenbergAdditional::Bil2}
\begin{split}
&\mleft(
\mleft(
\frac{\partial C_{j_2}}{\partial \zb[j_1]} - \frac{\partial C_{j_1}}{\partial \zb[j_2]}
\mright)_{1\leq j_1<j_2\leq n},
\mleft(
\frac{\partial D_{j_2}}{\partial \zb[j_1]} - \frac{\partial D_{j_1}}{\partial \zb[j_2]}
\mright)_{1\leq j_1<j_2\leq n},
\mleft(
\frac{\partial F_{j_2}}{\partial \zb[j_1]} - \frac{\partial F_{j_1}}{\partial \zb[j_2]}
\mright)_{1\leq j_1<j_2\leq n}
\mright)
\\&=\Gamma_2((C,D,F), \grad(C,D,F)).
\end{split}
\end{equation}
Finally, $[X_k,L_j]=0$ can be written as:
\begin{equation}\label{Eqn::NirenbergAdditional::Bil3}
\begin{split}
&\mleft(
\mleft(
\frac{\partial C_j}{\partial t_k} - \frac{\partial A_k}{\partial \zb[j]}
\mright)_{\substack{1\leq j\leq n\\ 1\leq k\leq r}},
\mleft(
\frac{\partial D_j}{\partial t_k} - \frac{\partial B_k}{\partial \zb[j]}
\mright)_{\substack{1\leq j\leq n\\ 1\leq k\leq r}},
\mleft(
\frac{\partial F_j}{\partial t_k} - \frac{\partial E_k}{\partial \zb[j]}
\mright)_{\substack{1\leq j\leq n\\ 1\leq k\leq r}}
\mright)
\\&=\Gamma_3((A,B,C,D,E,F), \grad(A,B,C,D,E,F)).
\end{split}
\end{equation}

Combining \cref{Eqn::NirenbergAdditional::Bil1}, \cref{Eqn::NirenbergAdditional::Bil2}, \cref{Eqn::NirenbergAdditional::Bil3}, and \cref{Eqn::NirenbergAdditional::Bil0} we see that
$(A,B,C,D,E,F)$ satisfies the following equation:
\begin{equation}\label{Eqn::NirenbergAdditional::BilF}
\sE (A,B,C,D,E,F)= \Gamma((A,B,C,D,E,F), \grad(A, B, C,D,E,F)),
\end{equation}
where $\Gamma$ is an explicit constant coefficient, bilinear form, depending
only on $n$ and $r$, and $\sE$ is the following explicit operator (which depends only on $n$ and $r$):
\begin{equation*}
\begin{split}
\sE (A,B,C,D,E,F)
=
\Bigg(&
\mleft(
\frac{\partial A_{k_2}}{\partial t_{k_1}} - \frac{\partial A_{k_1}}{\partial t_{k_2}}
\mright)_{1\leq k_1<k_2\leq r},
\mleft(
\frac{\partial C_{j_2}}{\partial \zb[j_1]} - \frac{\partial C_{j_1}}{\partial \zb[j_2]}
\mright)_{1\leq j_1<j_2\leq n},
\\&
\mleft(
\frac{\partial C_j}{\partial t_k} - \frac{\partial A_k}{\partial \zb[j]}
\mright)_{\substack{1\leq j\leq n\\ 1\leq k\leq r}},
\sum_{k=1}^r \frac{\partial A_k}{\partial t_k} +\sum_{j=1}^n \frac{\partial C_j}{\partial z_j},
\\&
\mleft(
\frac{\partial B_{k_2}}{\partial t_{k_1}} - \frac{\partial B_{k_1}}{\partial t_{k_2}}
\mright)_{1\leq k_1<k_2\leq r},
\mleft(
\frac{\partial D_{j_2}}{\partial \zb[j_1]} - \frac{\partial D_{j_1}}{\partial \zb[j_2]}
\mright)_{1\leq j_1<j_2\leq n},
\\&
\mleft(
\frac{\partial D_j}{\partial t_k} - \frac{\partial B_k}{\partial \zb[j]}
\mright)_{\substack{1\leq j\leq n\\ 1\leq k\leq r}},
\sum_{k=1}^r \frac{\partial B_k}{\partial t_k} +\sum_{j=1}^n \frac{\partial D_j}{\partial z_j},
\\&
\mleft(
\frac{\partial E_{k_2}}{\partial t_{k_1}} - \frac{\partial E_{k_1}}{\partial t_{k_2}}
\mright)_{1\leq k_1<k_2\leq r},
\mleft(
\frac{\partial F_{j_2}}{\partial \zb[j_1]} - \frac{\partial F_{j_1}}{\partial \zb[j_2]}
\mright)_{1\leq j_1<j_2\leq n},
\\&
\mleft(
\frac{\partial F_j}{\partial t_k} - \frac{\partial E_k}{\partial \zb[j]}
\mright)_{\substack{1\leq j\leq n\\ 1\leq k\leq r}},
\sum_{k=1}^r \frac{\partial E_k}{\partial t_k} +\sum_{j=1}^n \frac{\partial F_j}{\partial z_j}
\Bigg).
\end{split}
\end{equation*}
\Cref{Lemma::AppendElliptic::OneOperator} shows that $\sE$ is elliptic.

\Cref{Prop::AppendRealAnal::MainProp}, applied to \cref{Eqn::NirenbergAdditional::BilF}, shows that there is an admissible
$\gamma>0$ such that if 
$$\ZygNorm{A}{s_0+1}, \ZygNorm{B}{s_0+1}, \ZygNorm{C}{s_0+1}, \ZygNorm{D}{s_0+1},  \ZygNorm{E}{s_0+1},  \ZygNorm{F}{s_0+1}\leq \gamma,$$
then there exists an admissible $\eta_0>0$ such that \cref{Eqn::NirernbergAdditional::ToShowRegEqn} holds. 
Now the result follows from \cref{Prop::NirenbergRA::MainProp}.
\end{proof}

\begin{rmk}
We only use \cref{Prop::NirenbergAddtional::MainProp} in the special case $A\equiv 0$, $C\equiv 0$, $E\equiv 0$, and $F\equiv 0$; however the proof in this special case is no easier than the more general
case covered in \cref{Prop::NirenbergAddtional::MainProp}.
\end{rmk}

    \subsection{Vector fields with small error}\label{Section::Nirenberg::SmallError}
Fix $s_0\in (0,\infty)$.  We consider $\ZygSpace{s_0+1}$ complex vector fields,
$X_1,\ldots, X_r, L_1,\ldots, L_n$, on
$B_{\R^r\times \C^n}(2)$ of the following form:
\begin{equation*}
X=\diff{t} + E\diff{z}, \quad L=\diff{\zb}+F\diff{z}.
\end{equation*}
Here we are again using the matrix notation from
\cref{Section::Nirenberg::Additional}.

We assume:
\begin{enumerate}[(I)]
\item $E\in \ZygSpace{s_0+1}[B_{\R^r\times \C^n}(2)][\M^{r\times n}(\C)]$, $F\in \ZygSpace{s_0+1}[B_{\R^r\times \C^n}(2)][\M^{n\times n}(\C)]$.
    \item $E(0)=0$, $F(0)=0$.
    \item\label{Item::NirenbergPerterb::Commute} $\forall \zeta\in B_{\R^r\times \C^n}(2)$,
    $[X_{k_1}, X_{k_2}](\zeta), [L_{j_1}, L_{j_2}](\zeta), [X_k,L_j](\zeta)\in \Span_{\C}\mleft\{X_1(\zeta),\ldots, X_r(\zeta), L_1(\zeta),\ldots, L_n(\zeta)\mright\}$, $\forall j,k,l$.
\end{enumerate}

\begin{rmk}\label{Rmk::NirenbergPerterm::Commute}
Assumption \cref{Item::NirenbergPerterb::Commute} is equivalent to assuming
$X_1,\ldots, X_r, L_1,\ldots, L_n$ commute.  Indeed, under \cref{Item::NirenbergPerterb::Commute} and because
of the form of $X$ and $L$, we have $\forall \zeta\in B_{\R^r\times \C^n}(2)$,
\begin{equation*}
\begin{split}
&[X_{k_1}, X_{k_2}](\zeta), [L_{j_1}, L_{j_2}](\zeta), [X_k,L_j](\zeta)
\\&\in \Span_{\C}\mleft\{X_1(\zeta),\ldots, X_r(\zeta), L_1(\zeta),\ldots, L_n(\zeta)\mright\}\bigcap \Span_{\C}\mleft\{\diff{z_1},\ldots, \diff{z_n}\mright\}=\{0\}.
\end{split}
\end{equation*}
\end{rmk}

\begin{defn}
For $s>s_0$ if we say $C$ is an $\Zygad{s}$-admissible constant,
it means that we assume
$E\in \ZygSpace{s+1}[B_{\R^r\times \C^n}(2)][\M^{r\times n}(\C)]$ and $F\in \ZygSpace{s+1}[B_{\R^r\times \C^n}(2)][\M^{n\times n}(\C)]$.
$C$ can be chosen to depend only on $n$, $r$, $s_0$, $s$,
and upper bounds for $\ZygNorm{E}{s+1}[B_{\R^r\times \C^n}(2)]$ and $\ZygNorm{F}{s+1}[B_{\R^r\times \C^n}(2)]$.
For $0<s\leq s_0$ we say $C$ is an $\Zygad{s}$-admissible constant
if $C$ can be chosen to depend only on $n$, $r$, and $s_0$.
\end{defn}

\begin{prop}\label{Prop::NirenbergPertrub::MainProp}
There exists $\sigma=\sigma(n,r,s_0)>0$ such that if
$\ZygNorm{E}{s_0+1}[B_{\R^r\times \C^n}(2)], \ZygNorm{F}{s_0+1}[B_{\R^r\times \C^n}(2)]\leq \sigma$,
then there exists an $\Zygad{s_0}$-admissible constant $\eta_3>0$
and a map $\Phi_2:B_{\R^r\times \C^n}(\eta_3)\rightarrow B_{\R^r\times \C^n}(2)$ such that:
\begin{itemize}
\item $\Phi_2\in \ZygSpace{s_0+2}[B_{\R^r\times \C^n}(\eta_3)][\R^r\times \C^n]$ and $\ZygNorm{\Phi_2}{s+2}[B_{\R^r\times \C^n}(\eta_3)]\lesssim_{\Zygad{s}} 1$, $\forall s>0$.
\item $\Phi_2(0)=0$, $d_{t,x} \Phi_2(0) = I_{(r+2n)\times (r+2n)}$.
\item $\Phi_2(B_{\R^r\times \C^n}(\eta_3))\subseteq B_{\R^r\times \C^n}(2)$ is open and $\Phi_2:B_{\R^r\times \C^n}(\eta_3)\rightarrow  \Phi_2(B_{\R^r\times \C^n}(\eta_3))$ is a $\ZygSpace{s_0+2}$ diffeomorphism.
\item \begin{equation*}
\begin{bmatrix}
\diff{u} \\ \diff{\wb}
\end{bmatrix}=
(I+\AMatrix_2)
\begin{bmatrix}
\Phi_2^{*} X\\ \Phi_2^{*} L
\end{bmatrix},
\end{equation*}
where $\AMatrix_2:B_{\R^r\times \C^n}(\eta_3)\rightarrow \M^{(r+n)\times (r+n)}(\C)$, $\AMatrix_2(0)=0$, and $\ZygNorm{\AMatrix}{s+1}[B_{\R^r\times \C^n}(\eta_3)]\lesssim_{\Zygad{s}} 1$.
\end{itemize}
\end{prop}

To prove \cref{Prop::NirenbergPertrub::MainProp}, we prove the following lemma.

\begin{lemma}\label{Lemma::NirenbergPerturb::MainLemma}
Fix $\gamma>0$.  There exists $\sigma=\sigma(n,r,s_0,\gamma)>0$ such that if $\ZygNorm{E}{s_0+1}[B_{\R^r\times \C^n}(2)], \ZygNorm{F}{s_0+1}[B_{\R^r\times \C^n}(2)]\leq \sigma$,
there exists $H\in \ZygSpace{s_0+2}[B_{\R^r\times \C^n}(2)][\R^r\times \C^n]$ such that
\begin{enumerate}[(i)]
\item $H(t,z)=(t,z)+R(t,z)$, $R(0,0)=0$, $d_{t,x}R(0,0)=0_{(r+2n)\times (r+2n)}$. 
\item $\ZygNorm{H}{s+2}[B_{\R^r\times \C^n}(3/2)]\lesssim_{\Zygad{s}} 1$, $\forall s>0$.
\item\label{Item::NirenbergPlus::Diffeo} $H:B_{\R^r\times \C^n}(2)\rightarrow B_{\R^r\times \C^n}(3)$ is injective, $H(B_{\R^r\times \C^n}(2))\subseteq B_{\R^r\times \C^n}(3)$ is open,
and $H:B_{\R^r\times \C^n}(2)\rightarrow H(B_{\R^r\times \C^n}(2))$ is a diffeomorphism.
\item\label{Item::NirenbergPlus::BigImage} $B_{\R^r\times \C^n}(1)\subseteq H(B_{\R^r\times \C^n}(3/2))$.
\item $\ZygNorm{H^{-1}}{s+2}[B_{\R^r\times \C^n}(1)]\lesssim_{\Zygad{s}} 1$, $\forall s>0$.
\item Let $V_k:=H_{*} X_k$ and $W_j=H_{*} L_j$.  Then there exists a matrix
$M\in \ZygSpace{s_0+1}[B_{\R^r\times \C^n}(1)][\M^{(r+n)\times (r+n)}(\C)]$ with $M(0)=0$ and such that:
	\begin{itemize}
		\item $\ZygNorm{M}{s+1}[B_{\R^r\times \C^n}(1)]\lesssim_{\Zygad{s}} 1$, $\forall s>0$.
		\item If
			\begin{equation*}
				\begin{bmatrix}
				\Xt \\
				\Lt
				\end{bmatrix}:=
				(I+M)
				\begin{bmatrix}
				V\\
				W
				\end{bmatrix},
			\end{equation*}
			then
			\begin{equation*}
			\Xt=\diff{u}+B\diff{w}, \quad \Lt=\diff{\wb}+D\diff{w},
			\end{equation*}
			where we are using the matrix notation from \cref{Section::Nirenberg::Additional}.  We have
			\begin{equation*}
			\ZygNorm{B}{s_0+1}[B_{\R^r\times \C^n}(1)],\ZygNorm{D}{s_0+1}[B_{\R^r\times \C^n}(1)]\leq \gamma,
			\end{equation*}
			\begin{equation*}
			\ZygNorm{B}{s+1}[B_{\R^r\times \C^n}(1)],\ZygNorm{D}{s+1}[B_{\R^r\times \C^n}(1)]\lesssim_{\Zygad{s}} 1,\quad \forall s>0,
			\end{equation*}
			and $B(0)=0$, $D(0)=0$.
			Finally, if we let $B_k$ denote the $k$th row of $B$, and similarly for $D_j$,
			we have
			\begin{equation}\label{Eqn::NirenbergPerterb::NewEqn}
\sum_{k=1}^r \frac{\partial B_k}{\partial u_k} +\sum_{j=1}^n \frac{\partial D_j}{\partial w_j}=0.
\end{equation}
		\item $\Xt_1,\ldots, \Xt_r, \Lt_1,\ldots, \Lt_n$ commute on $B_{\R^r\times \C^n}(1)$.
	\end{itemize}
\end{enumerate}
\end{lemma}

First we see why \cref{Lemma::NirenbergPerturb::MainLemma} gives \cref{Prop::NirenbergPertrub::MainProp}
\begin{proof}[Proof of \cref{Prop::NirenbergPertrub::MainProp}]
Take $\gamma=\gamma(n,r,s_0)>0$ as in \cref{Prop::NirenbergAddtional::MainProp}.
We take $\sigma=\sigma(n,r,s_0,\gamma)>0$ as in \cref{Lemma::NirenbergPerturb::MainLemma}.  With this choice of $\sigma$ and $\gamma$, \cref{Lemma::NirenbergPerturb::MainLemma}
shows that \cref{Prop::NirenbergAddtional::MainProp} applies to the vector fields $\Xt_1,\ldots, \Xt_r,\Lt_1,\ldots, \Lt_n$ from \cref{Lemma::NirenbergPerturb::MainLemma} (and constants which are admissible in the sense of \cref{Prop::NirenbergAddtional::MainProp}
are $\Zygad{s_0}$-admissible in the sense of this section)--here we are taking $A\equiv 0$, $C\equiv 0$, $E\equiv 0$, and $F\equiv 0$ in \cref{Prop::NirenbergAddtional::MainProp}.

Thus, we obtain an $\Zygad{s_0}$-admissible constant $\eta_2>0$ and a map $\Phi_1:B_{\R^r\times\C^n}(\eta_2)\rightarrow B_{\R^r\times \C^n}(1)$ as in  \cref{Prop::NirenbergAddtional::MainProp}.
Set $\eta_3:=\eta_2/2$.
For each $s>0$, we have, using \cref{Lemma::FuncSpaceReAnal::AinB},
\begin{equation*}
\ZygNorm{\Phi_1}{s+2}[B_{\R^r\times \C^n}(\eta_3)]\leq C_{s,\eta_2} \sBNorm{\Phi_1}{r+2n}{r+2n}{\eta_2}\lesssim_{\Zygad{s_0}} C_{s,\eta_2},
\end{equation*}
where $C_{s,\eta_2}$ can be chosen to depend only on $s$ and $\eta_2$.  We conclude,
\begin{equation*}
\ZygNorm{\Phi_1}{s+2}[B_{\R^r\times \C^n}(\eta_3)]\lesssim_{\Zygad{s}} 1,\quad \forall s>0.
\end{equation*}
Similarly, if $\AMatrix_1$ is as in  \cref{Prop::NirenbergAddtional::MainProp}, we have $\AMatrix_1(0)=0$,
and using \cref{Lemma::FuncSpaceReAnal::AinB},
$\ZygNorm{\AMatrix_1}{s+1}[B_{\R^r\times \C^n}(\eta_3)]\lesssim_{\Zygad{s}} \sBNorm{\AMatrix_1}{r+2n}{\M^{(n+r)\times (n+r)}}{\eta_2}\leq 1$, $\forall s>0$, and
\begin{equation*}
	\begin{bmatrix}
	\diff{u}\\\diff{\wb}
	\end{bmatrix}
	=(I+\AMatrix_1)\begin{bmatrix}
	\Phi_1^{*} \Xt\\ \Phi_1^{*} \Lt
	\end{bmatrix}.
\end{equation*}
We have, with $M$ and $H$ as in \cref{Lemma::NirenbergPerturb::MainLemma},
\begin{equation*}
	\begin{bmatrix}
	\diff{u}\\\diff{\wb}
	\end{bmatrix}
	=(I+\AMatrix_1)(I+M\circ \Phi_1)\begin{bmatrix}
	\Phi_1^{*} V\\ \Phi_1^{*} W
	\end{bmatrix}
	=(I+\AMatrix_1)(I+M\circ \Phi_1)
	\begin{bmatrix}
	(H^{-1}\circ \Phi_1)^{*} X \\
	(H^{-1}\circ \Phi_1)^{*} L
	\end{bmatrix}
	=:(I+\AMatrix_2)
	\begin{bmatrix}
	\Phi_2^{*} X \\
	\Phi_2^{*} L
	\end{bmatrix},
\end{equation*}
where $\Phi_2=H^{-1}\circ \Phi_1$ and $I+\AMatrix_2 = (I+\AMatrix_1)(I+M\circ \Phi_1)$.
Since we have already noted $\ZygNorm{\Phi_1}{s+2}[B_{\R^r\times \C^n}(\eta_3)]\lesssim_{\Zygad{s}} 1$ and $\ZygNorm{\AMatrix_1}{s+1}[B_{\R^r\times \C^n}(\eta_3)]\lesssim_{\Zygad{s}} 1$, $\forall s>0$, and \cref{Lemma::NirenbergPerturb::MainLemma} gives $\ZygNorm{M}{s+1}[B_{\R^r\times \C^n}(1)]\lesssim_{\Zygad{s}} 1$, $\forall s>0$,
it follows from \cref{Prop::FuncSpaceRev::Algebra,Lemma::FuncSpaceRev::Composition} that
$\ZygNorm{\AMatrix_2}{s+1}[B_{\R^r\times \C^n}(\eta_3)]\lesssim_{\Zygad{s}} 1$, $\forall s>0$.
Since $\AMatrix_1(0)=0$, $M(0)=0$, and $\Phi_1(0)=0$, we have $\AMatrix_2(0)=0$.
Since $\ZygNorm{H^{-1}}{s+2}[B_{\R^r\times \C^n}(1)]\lesssim_{\Zygad{s}} 1$ (by \cref{Lemma::NirenbergPerturb::MainLemma}) and $\ZygNorm{\Phi_1}{s+2}[B_{\R^r\times \C^n}(\eta_3)]\lesssim_{\Zygad{s}} 1$,
for all $s>0$, it follows from \cref{Lemma::FuncSpaceRev::Composition} that
$\ZygNorm{\Phi_2}{s+2}[B_{\R^r\times \C^n}(\eta_3)]\lesssim_{\Zygad{s}} 1$, $\forall s>0$.
$\Phi_2(0)=H^{-1}(\Phi_1(0))=H^{-1}(0)=0$, since $H(0)=0$.  $d_{t,x} \Phi_2(0)= (d_{t,x} H(0))^{-1} d_{t,x} \Phi_1(0)= I \cdot I = I$.
Finally, that $\Phi_2$ is a diffeomorphism onto its image follows from the corresponding results for $H$ and $\Phi_1$ in \cref{Lemma::NirenbergPerturb::MainLemma} and \cref{Prop::NirenbergAddtional::MainProp}.
\end{proof}

\begin{proof}[Proof of \cref{Lemma::NirenbergPerturb::MainLemma}]
Let $\sigma_0=\sigma_0(n,r,s_0,\gamma)>0$ be a small constant (depending only on $n$, $r$, $s_0$, and $\gamma$), to be chosen later.
We will find $H$ of the form $H(t,z)=(t,z)+R(t,z)$, where $R(t,z)=(0,R_2(t,z))$, $R_2\in \ZygSpace{s_0+2}[B_{\R^r\times \C^n}(2)][\C^n]$, $R_2(0,0)=0$, $dR_2(0,0)=0$,
and $\ZygNorm{R_2}{s_0+2}[B_{\R^r\times \C^n}(2)]\leq \sigma_0$.  Note that if $\sigma_0>0$ is sufficiently small (depending only on $n$ and $r$),
\cref{Item::NirenbergPlus::Diffeo} and \cref{Item::NirenbergPlus::BigImage} follow immediately from the inverse function theorem.
Moreover, we will also
have
\begin{equation}\label{Eqn::NirenbergPerterb::BoundDet}
\inf_{(t,z)\in B_{\R^r\times \C^n}(2)} |\det dH(t,z)|\geq \frac{1}{2}.
\end{equation}
Henceforth, we take $\sigma_0>0$ so small that these consequences hold.

We begin by studying an arbitrary $H(t,z)$ of the form $H(t,z)=(t,z)+(0,R_2(t,z))$ with $\ZygNorm{R_2}{s_0+2}[B_{\R^r\times \C^n}(2)]\leq \sigma_0$, $R_2(0,0)=0$, $dR_2(0,0)=0$ (we will later
specialize to a specific choice of $R_2$).  In what follows, for $s>0$ if we write $A\lesssim_s B$, it means that we assume $R_2\in \ZygSpace{s+2}[B_{\R^r\times \C^n}(3/2)][\C^n]$ and $A\leq CB$
where $C$ is a positive $\Zygad{s}$-admissible constant which is also allowed to depend on an upper bound for $\ZygNorm{R_2}{s+2}[B_{\R^r\times \C^n}(3/2)]$.  At the end of the proof,
we will chose a particular $R_2$ with $\ZygNorm{R_2}{s+2}[B_{\R^r\times \C^n}(3/2)]\lesssim_{\Zygad{s}} 1$; once we do this, $\lesssim_s$ and $\lesssim_{\Zygad{s}}$ will denote the same thing.

For such $H$, by the above remarks, it makes sense to consider $H^{-1}:B_{\R^r\times \C^n}(1)\rightarrow B_{\R^r\times \C^n}(3/2)$.  Moreover,
it follows from 
\cref{Lemma::FuncSpaceRev::Inverse} (using \cref{Eqn::NirenbergPerterb::BoundDet})
that
\begin{equation}\label{Eqn::NirenbergPerterb::BoundHinv}
\ZygNorm{H^{-1}}{s+2}[B_{\R^r\times \C^n}(1)]\lesssim_{s} 1.
\end{equation}

Set $H_1(t,z)=t$, $H_2(t,z)=z+R_2(t,z)$ so that $H(t,z)=(H_1(t,z), H_2(t,z))$.  We have the following obvious equalities:
\begin{equation}\label{Eqn::NirenbergPerterb::Basicds}
\begin{split}
&d_t H_1 = I, d_z H_1=0, d_{\zb} H_1=0, d_t H_2 = d_t R_2, d_z H_2 = I+d_z R_2, 
\\&d_{\zb} H_2 = d_{\zb} R_2, d_t \Hb[2] = d_t \Rb[2], d_z\Hb[2] = d_z \Rb[2], d_{\zb} \Hb[2] = I+ d_{\zb} \Rb[2].
\end{split}
\end{equation}
Using the notation from \cref{Section::Notation}, we have (thinking of $H$ mapping the $(t,z)$ variable to the $(u,w)$ variable):
\begin{equation*}
H_{*} \diff{t} =(d_t H_1(t,z))^{\transpose} \diff{u} + (d_t H_2(t,z))^{\transpose} \diff{w} +(d_t \Hb[2](t,z))^{\transpose} \diff{\wb}\bigg|_{(t,z)=H^{-1}(u,w)},
\end{equation*}
\begin{equation*}
H_{*} \diff{z} =(d_z H_1(t,z))^{\transpose} \diff{u} + (d_z H_2(t,z))^{\transpose} \diff{w} +(d_z \Hb[2](t,z))^{\transpose} \diff{\wb}\bigg|_{(t,z)=H^{-1}(u,w)},
\end{equation*}
\begin{equation*}
H_{*} \diff{\zb} =(d_{\zb} H_1(t,z))^{\transpose} \diff{u} + (d_{\zb} H_2(t,z))^{\transpose} \diff{w} +(d_{\zb} \Hb[2](t,z))^{\transpose} \diff{\wb}\bigg|_{(t,z)=H^{-1}(u,w)}.
\end{equation*}
Thus, if $V=H_{*}X$ and $W=H_{*}L$, using \cref{Eqn::NirenbergPerterb::Basicds} we have
\begin{equation*}
\begin{split}
V(u,w) = \diff{u} &+ \bigg[ (d_tR_2(t,z))^{\transpose} + E(t,z) (I+d_zR_2(t,z))^{\transpose} \bigg]\diff{w}
\\&+\bigg[ (d_t \Rb[2](t,z) )^{\transpose} +E(t,z) (d_z \Rb[2](t,z))^{\transpose} \bigg]\diff{\wb}\bigg|_{(t,z)=H^{-1}(u,w)},
\end{split}
\end{equation*}
\begin{equation*}
\begin{split}
W(u,w) = \diff{\wb} &+ \bigg[ (d_{\zb} R_2(t,z))^{\transpose} + F(t,z) (I+d_z R_2(t,z))^{\transpose}  \bigg]\diff{w}
\\& + \bigg[ (d_{\zb} \Rb[2](t,z))^{\transpose} + F(t,z) (d_z \Rb[2](t,z))^{\transpose} \bigg]\diff{\wb}\bigg|_{(t,z)=H^{-1}(u,w)}.
\end{split}
\end{equation*}

Our goal is to pick $\sigma=\sigma(n,r,s_0,\gamma)>0$ so that the conclusions of the lemma hold for
$$\ZygNorm{E}{s_0+1}[B_{\R^r\times \C^n}(2)], \ZygNorm{F}{s_0+1}[B_{\R^r\times \C^n}(2)]\leq \sigma.$$
  We will choose $\sigma$ at the end of the proof;
but we will ensure $\sigma\leq 1$, so that we may henceforth assume $\ZygNorm{E}{s_0+1}[B_{\R^r\times \C^n}(2)], \ZygNorm{F}{s_0+1}[B_{\R^r\times \C^n}(2)]\leq 1$.  Using this and
the assumption $\ZygNorm{R_2}{s_0+2}[B_{\R^r\times \C^n}(2)]\leq \sigma_0$,  we have, by taking $\sigma_0>0$ sufficiently small (depending only on $n$ and $r$),
\begin{equation*}
\inf_{(t,z)\in B_{\R^r\times \C^n}(2)} \mleft| \det \bigg[ I+ (d_{\zb} \Rb[2](t,z))^{\transpose} +   F(t,z) (d_z \Rb[2](t,z))^{\transpose}\bigg] \mright|\geq \frac{1}{2}.
\end{equation*}
Thus, $I+d_{\zb}\Rb[2]^{\transpose} + F d_z\Rb[2]^{\transpose}$ is invertible on $B_{\R^r\times \C^n}(2)$ and
\cref{Rmk::FuncSpaceRev::InverseMatrix} implies
\begin{equation}\label{Eqn::NirenbergPerterb::BoundInverse}
\BZygNorm{\mleft(I+d_{\zb}\Rb[2]^{\transpose} + F d_z\Rb[2]^{\transpose}\mright)^{-1}}{s+1}[B_{\R^r\times \C^n}(3/2)][\M^{n\times n}]\lesssim_{s} 1.
\end{equation}
We define a matrix 
$M(u,w):H(B_{\R^r\times \C^n}(2))\rightarrow \M^{(r+n)\times (r+n)}(\C)$
by
\begin{equation*}
I+M\circ H :=
{\mleft[\def\arraystretch{1.5}
\begin{array}{c | c}
I_{r\times r} &  -\mleft( d_t \Rb[2]^{\transpose} +E d_z\Rb[2]^{\transpose}\mright)\mleft( I+d_{\zb} \Rb[2]^{\transpose} +F d_z \Rb[2]^{\transpose} \mright)^{-1}\\
\hline
0_{n\times r} & \mleft( I+ d_{\zb} \Rb[2]^{\transpose}+ Fd_z \Rb[2]^{\transpose}  \mright)^{-1}
\end{array}
\mright],}
\end{equation*}
where each part of the above equation is evaluated at $(t,z)$ and we are using notation like $d_t\Rb[2]^{\transpose}$ to mean $(d_t \Rb[2](t,z))^{\transpose}$.
In particular, since $B_{\R^r\times \C^n}(1)\subseteq H(B_{\R^r\times \C^n}(2))$ (by \cref{Item::NirenbergPlus::BigImage} which we have already verified), $M$ is defined on $B_{\R^r\times \C^n}(1)$.

By \cref{Eqn::NirenbergPerterb::BoundInverse} and \cref{Prop::FuncSpaceRev::Algebra}, we have
$\ZygNorm{I+M\circ H}{s+1}[B_{\R^r\times \C^n}(3/2)]\lesssim_s 1$.  Combining this with \cref{Eqn::NirenbergPerterb::BoundHinv}, \cref{Lemma::FuncSpaceRev::Composition} shows
\begin{equation*}
\ZygNorm{M}{s+1}[B_{\R^r\times \C^n}(1)]\lesssim_s 1, \quad \forall s>0.
\end{equation*}
Also, since $d R_2(0)=0$ and $H(0)=0$, we have $M(0)=0$.  Set
\begin{equation}\label{Eqn::NirenberPerterb::DefineXtLt}
\begin{bmatrix}
\Xt \\ \Lt
\end{bmatrix}:=(I+M)
\begin{bmatrix} V\\W\end{bmatrix}.
\end{equation}
Note that
\begin{equation}\label{Eqn::NirenbergPerterb::FormXtLt}
\Xt = \diff{u} + B\diff{w}, \quad \Lt=\diff{\wb}+D\diff{w},
\end{equation}
where $B$ and $D$ depend on $E$, $F$, and $R_2$, and (in what follows each function is evaluated at $(t,z)$ unless otherwise mentioned):
\begin{equation}\label{Eqn::NirenberPerterb::FormB}
\begin{split}
&B(u,w) = B[E,F,R_2](u,w)
\\&=
\mleft( d_t R_2^{\transpose} + E (I+d_z R_2^{\transpose}) \mright) - \mleft( d_t \Rb[2]^{\transpose} + E d_z \Rb[2]^{\transpose}   \mright)\mleft( I+ d_{\zb} \Rb[2]^{\transpose} + F d_z \Rb[2]^{\transpose} \mright)^{-1}
\mleft( d_{\zb} R_2^{\transpose} + F(I+d_z R_2^{\transpose}) \mright)\bigg|_{(t,z)=H^{-1}(u,w)},
\end{split}
\end{equation}
\begin{equation}\label{Eqn::NirenberPerterb::FormD}
\begin{split}
&D(u,w)=D[E,F,R_2](u,w)=
\mleft( I+ d_{\zb} \Rb[2]^{\transpose}+F d_z\Rb[2]^{\transpose} \mright)^{-1}
\mleft( d_{\zb} R_2^{\transpose} + F(I+d_z R_2^{\transpose}) \mright)\bigg|_{(t,z)=H^{-1}(u,w)}.
\end{split}
\end{equation}
Note that since $E(0)=0$, $F(0)=0$,  $dR_2(0)=0$, and $H^{-1}(0)=0$, we have $B(0)=0$ and $D(0)=0$.
Let $\sigma_1=\sigma_1(n,r,s_0,\gamma)\in (0,1]$ be a small constant to be chosen later.
At the end of the proof, we will take $\sigma\leq \sigma_1$ so we may assume $\ZygNorm{E}{s_0+1}[B_{\R^r\times \C^n}(2)], \ZygNorm{F}{s_0+1}[B_{\R^r\times \C^n}(2)]\leq \sigma_1$.
We have, using \cref{Eqn::NirenbergPerterb::BoundHinv}, \cref{Eqn::NirenbergPerterb::BoundInverse}, \cref{Prop::FuncSpaceRev::Algebra}, and \cref{Lemma::FuncSpaceRev::Composition},
\begin{equation*}
\ZygNorm{D}{s_0+1}[B_{\R^r\times \C^n}(1)] \lesssim_{\Zygad{s_0}} \ZygNorm{d_{\zb} R_2^{\transpose} + F(I+d_z R_2^{\transpose}) }{s_0+1}[B_{\R^r\times \C^n}(2)] \lesssim_{\Zygad{s_0}} \sigma_0+\sigma_1,
\end{equation*}
and
\begin{equation*}
\ZygNorm{D}{s+1}[B_{\R^r\times \C^n}(1)] \lesssim_{s} \ZygNorm{d_{\zb} R_2^{\transpose} + F(I+d_z R_2^{\transpose}) }{s+1}[B_{\R^r\times \C^n}(3/2)] \lesssim_{s} 1,\quad \forall s>0.
\end{equation*}
Similarly, we have
\begin{equation*}
\ZygNorm{B}{s_0+1}[B_{\R^r\times \C^n}(1)] \lesssim_{\Zygad{s_0}} \sigma_0+\sigma_1, \quad 
\ZygNorm{B}{s+1}[B_{\R^r\times \C^n}(1)] \lesssim_s 1, \quad \forall s>0.
\end{equation*}
In particular, if we take $\sigma_0$ and $\sigma_1$ sufficiently small (depending only on $n$, $r$, $s_0$, and $\gamma$), we have
\begin{equation*}
\ZygNorm{B}{s_0+1}[B_{\R^r\times \C^n}(1)], \ZygNorm{D}{s_0+1}[B_{\R^r\times \C^n}(1)] \leq \gamma.
\end{equation*}

Next we claim that $\Xt_1,\ldots, \Xt_r,\Lt_1,\ldots, \Lt_n$ commute.  We are given that $X_1,\ldots, X_r,L_1,\ldots, L_n$ commute (see \cref{Rmk::NirenbergPerterm::Commute}), and it follows that
$V_1,\ldots, V_r, W_1,\ldots, W_n$ commute.  Since $I+M(u,w)$ is clearly an invertible matrix by its definition, \cref{Eqn::NirenberPerterb::DefineXtLt} shows $\forall (u,w)\in H(B_{\R^r\times \C^n}(1))$,
$\forall j,k, j_1, j_2, k_1,k_2$,
\begin{equation*}
\begin{split}
&[\Xt_{k_1}, \Xt_{k_2}](u,w), [\Lt_{j_1}, \Lt_{j_2}](u,w), [\Xt_k, \Lt_j](u,w)
\in \Span_\C\{ V_1(u,w),\ldots, V_r(u,w),W_1(u,w),\ldots, W_n(u,w)\} 
\\&= \Span_{\C}\{ \Xt_1(u,w),\ldots, \Xt_r(u,w), \Lt_1(u,w),\ldots, \Lt_n(u,w)\}.
\end{split}
\end{equation*}
Because of the form of $\Xt$ and $\Lt$ given in \cref{Eqn::NirenbergPerterb::FormXtLt} this implies $\Xt_1,\ldots, \Xt_r, \Lt_1,\ldots, \Lt_n$ commute (just as in \cref{Rmk::NirenbergPerterm::Commute}).

So far we have shown that if we have $R_2$ as above with $R_2(0)=0$, $d R_2(0)=0$,  $\ZygNorm{R_2}{s_0+2}[B_{\R^r\times \C^n}(2)]\leq \sigma_0$,
 and have $\ZygNorm{R_2}{s+2}[B_{\R^r\times \C^n}(3/2)]\lesssim_{\Zygad{s}} 1$, then all of the conclusions of the lemma hold,
except possibly for \cref{Eqn::NirenbergPerterb::NewEqn}.  Thus all that remains to show is that we can pick such an $R_2$ so that \cref{Eqn::NirenbergPerterb::NewEqn} holds
(provided $\sigma$ is small enough).  To do this we use \cref{Prop::AppendExist::MainProp}.

Given $E$, $F$, and $R_2$, we define $B=B[E,F,R_2]$ and $D=D[E,F,R_2]$ by \cref{Eqn::NirenberPerterb::FormB} and \cref{Eqn::NirenberPerterb::FormD}.
We let $B_{k,l}$ denote the $(k,l)$ component of the matrix $B$, and similarly for $D$.
For $(t,z)\in B_{\R^r\times \C^n}(2)$ and $1\leq m\leq n$,
\begin{equation*}
\Psi_{m}(E,F,R_2)(t,z):=\sum_{k=1}^r \frac{\partial B_{k,m}(u,w)}{\partial u_k} + \sum_{j=1}^n \frac{\partial D_{j,m}(u,w)}{\partial w_j}\bigg|_{(u,w)=H(t,z)}.
\end{equation*}
Set $\Psi(E,F,R_2):=(\Psi_{1}(E,F,R_2),\ldots, \Psi_{n}(E,F,R_2))$.
Note that \cref{Eqn::NirenbergPerterb::NewEqn} follows from $\Psi(E,F,R_2)=0$, so our goal is to solve for $R_2$ (in terms of $E$ and $F$)
so that $\Psi(E,F,R_2)=0$.

Letting $R(t,z) = (0, R_2(t,z))$, 
for any function $K(t,x)$ we have
\begin{equation*}
\diff{u_k} K(H^{-1}(u,w))\bigg|_{(u,w)=H(t,z)} = d K(t,z) (I + d R(t,z))^{-1} e_k,
\end{equation*}
where $e_k$ is the $k$th standard basis element--what is important is that the right hand side is a function of $dK(t,z)$ and $dR_2(t,z)$.  Similar comments hold
for $\diff{y_l}K(H^{-1}(s,w))$ where $w_j=y_{j}+iy_{j+n}$.
Thus, using the formulas for $B$ and $D$ in  \cref{Eqn::NirenberPerterb::FormB} and \cref{Eqn::NirenberPerterb::FormD}, using the notation of \cref{Prop::AppendExist::MainProp}, and writing $z_j=x_j+ix_{j+n}$ we see that there is a smooth function $g$, taking values in in $\C^n$,
which vanishes at the origin, such that
\begin{equation*}
\Psi(E,F,R_2)(t,x) = g( \Deriv^1 E(t,x), \Deriv^1 F(t,x), \Deriv^2 R_2(t,x)).
\end{equation*}
Furthermore, the function $g$ depends only on $n$ and $r$.  Also it is easy to see that $g$ is quasi-linear in $R_2$ in the sense of \cref{Eqn::AppendExist::Quasi}.\footnote{It is not necessary for what follows that $g$ be quasi-linear; though the proof of \cref{Prop::AppendExist::MainProp} is simpler in the quasi-linear case.}

To apply \cref{Prop::AppendExist::MainProp}, we wish to show that $g$ is elliptic in $R_2$ at $E=0$, $F=0$, $R_2=0$, in the sense of that proposition.  I.e., define $\sE_2$ as in \cref{Prop::AppendExist::MainProp};
we wish to show $\sE_2$ is elliptic.  Note the map
\begin{equation*}
R_2\mapsto \frac{d}{d\epsilon}\bigg|_{\epsilon=0} \Psi(0,0,\epsilon R_2)
\end{equation*}
is a second order, constant coefficient differential operator acting on $R_2$ whose principal symbol is $\sE_2$.  Thus we wish to show that this operator is elliptic.

To make the dependance of $H$ on $R_2$ explicit, we write $H_{R_2}$ in place of $H$.  I.e., $H_{R_2}(t,z)=(t,z)+(0,R_2(t,z))$.  It suffices to compute
$\frac{d}{d\epsilon}\big|_{\epsilon=0} \Psi(0,0,\epsilon R_2)$ in the case $R_2\in C^\infty$.  In that case, we have
$H_{\epsilon R_2}^{-1}(u,w) = (u,w)-\epsilon (0, R_2(u,w))+O(\epsilon^2)$, and for example,
\begin{equation}\label{Eqn::NirenbergPerterb::DerivError}
\epsilon (d_t R_2)(H_{\epsilon R_2}^{-1}(u,w))= \epsilon (d_t R_2)(u,w)+O(\epsilon^2)\text{ and }\epsilon (d_tR_2)(H_{\epsilon R_2}(t,z)) = \epsilon (d_t R_2)(t,z)+O(\epsilon^2),
\end{equation}
and similarly for $d_t$ replaced by $d_{\zb}$.  Here, $O(\epsilon^{2})$ it denotes a term which is $C^\infty$ in the variables $(t,z)$ or $(u,w)$ and every derivative, of every order $\geq 0$, in these variables is $O(\epsilon^2)$ as $\epsilon\rightarrow 0$.

Thus, using the formulas \cref{Eqn::NirenberPerterb::FormB} and \cref{Eqn::NirenberPerterb::FormD}, we have
\begin{equation}\label{Eqn::NirenbergPererb::AError}
B[0,0,\epsilon R_2](s,w) = \epsilon d_t R_2(s,w)^{\transpose} +O(\epsilon^2),
\quad  D[0,0,\epsilon R_2](s,w) = \epsilon d_{\zb} R_2(s,w)^{\transpose} + O(\epsilon^2).
\end{equation}
We write  $R_2(t,z)=(R_{2,1}(t,z),\ldots, R_{2,n}(t,z))$.
We also write $d_t R_2(s,w)_{l,k}$ for the $(l,k)$ component of the matrix $d_t R_2$, and similarly for $d_{\zb} R_2$ (see the discussion of this notation in \cref{Section::Notation}).
  Using this notation, plugging \cref{Eqn::NirenbergPererb::AError} into the definition of $\Psi$,
and using \cref{Eqn::NirenbergPerterb::DerivError}, we have for $1\leq m\leq n$,
\begin{equation*}
\begin{split}
\Psi_{m}(0,0,\epsilon R_2)(t,z) &= \epsilon \sum_{k=1}^r \diff{u_k} (d_t R_2)(u,w)_{m,k} + \epsilon \sum_{j=1}^n \diff{w_j} (d_{\zb} R_2)(u,w)_{m,j} + O(\epsilon^2)\bigg|_{(u,w)=H_{\epsilon R_2}(t,z)}
\\&=\epsilon\sum_{k=1}^r \frac{\partial^2}{\partial t_k^2} R_{2,m}(t,z) + \epsilon \sum_{j=1}^n \frac{\partial^2}{\partial z_j \partial \zb[j]} R_{2,m}(t,z) + O(\epsilon^2).
\end{split}
\end{equation*}
We conclude
\begin{equation*}
\frac{d}{d\epsilon}\bigg|_{\epsilon=0} \Psi(0,0,\epsilon R_2) = \mleft( \sum_{k=1}^r \frac{\partial^2}{\partial t_k^2}  + \sum_{j=1}^n \frac{\partial^2}{\partial z_j \partial \zb[j]}  \mright)R_2,
\end{equation*}
and is therefore elliptic, as desired.

We apply \cref{Prop::AppendExist::MainProp} with $D=2$, $\eta=3/2$, and
$$N=\{ R_2\in \ZygSpace{s_0+1}[B_{\R^r\times \C^n}(2)][ \C^n] :  \ZygNorm{R_2}{s_0+2}[B_{\R^r\times \C^n}(2)][\C^n]< \sigma_0\}.$$
We conclude that there exists $\sigma_2>0$ (depending only on $n$, $r$, $s_0$, and $\sigma_0$--since $g$ depends only on $n$ and $r$) so that if $\ZygNorm{E}{s_0+1}[B_{\R^r\times \C^n}(2)], \ZygNorm{F}{s_0+1}[B_{\R^r\times \C^n}(2)]\leq \sigma_2$,
then we may find $R_2=R_2(E,F)\in N$ so that $\Psi(E,F,R_2)=0$.  The conclusions of \cref{Prop::AppendExist::MainProp} show that this $R_2$ satisfies $R_2(0)=0$, $dR_2(0)=0$, and
$\ZygNorm{R_2}{s+2}[B_{\R^r\times \C^n}(3/2)]\lesssim_{\Zygad{s}} 1$, $\forall s>0$.
Setting $\sigma:=\min\{\sigma_1, \sigma_2\}$ completes the proof.
\end{proof}


%

%

%

%
%

%


    \subsection{Commuting vector fields}\label{Section::Nirenberg::Commute}
Fix $\eta_0>0$, $s_0\in (0,\infty)\cup \{\omega\}$, and let $X_1,\ldots, X_r,L_1,\ldots, L_n$ be complex $\ZygSpace{s_0+1}$ vector fields on $B_{\R^r\times\C^n}(\eta_0)$ of the form
\begin{equation*}
X=\diff{t}+E\diff{z}, \quad L=\diff{\zb}+F\diff{z},
\end{equation*}
where $E(0)=0$, $F(0)=0$, we are using the matrix notation from \cref{Section::Nirenberg::Additional}, and:
\begin{itemize}
	\item If $s_0\in (0,\infty)$, $E\in \ZygSpace{s_0+1}[B_{\R^r\times \C^n}(\eta_0)][\M^{r\times n}(\C)]$ and $F \in \ZygSpace{s_0+1}[B_{\R^r\times \C^n}(\eta_0)][\M^{n\times n}(\C)]$.
	\item If $s_0=\omega$, $E\in \ASpace{r+2n}{\eta_0}[\M^{r\times n}(\C)]$ and $F\in \ASpace{r+2n}{\eta_0}[\M^{n\times n}(\C)]$.
\end{itemize}
We suppose $\forall \zeta\in B_{\R^r\times \C^n}(\eta_0)$,
\begin{equation*}
[X_{k_1},X_{k_2}](\zeta), [L_{j_1},L_{j_2}](\zeta), [X_k,L_j](\zeta)\in \Span_{\C}\{ X_1(\zeta),\ldots, X_r(\zeta),L_1(\zeta),\ldots, L_n(\zeta)\}, \quad \forall j,k,j_1,j_2,k_1,k_2.
\end{equation*}
As in \cref{Rmk::NirenbergPerterm::Commute}, this is the same as assuming the vector fields commute.

\begin{defn}
If $s_0\in (0,\infty)$, for $s\in [s_0,\infty)$, if we say $C$ is an $\Zygad{s}$-admissible constant, it means that we assume $E,F\in \ZygSpace{s+1}[B_{\R^r\times \C^n}(\eta_0)]$.
$C$ can then be chosen to depend only on $s$, $s_0$, $n$, $r$, $\eta_0$, and upper bounds for $\ZygNorm{E}{s+1}[B_{\R^r\times \C^n}(\eta_0)]$ and
$\ZygNorm{F}{s+1}[B_{\R^r\times \C^n}(\eta_0)]$.  For $s\in (0,s_0)$, we define $\Zygad{s}$-admissible constants to be $\Zygad{s_0}$-admissible constants.
\end{defn}

\begin{defn}
If $s_0=\omega$, we say $C$ is an $\Zygad{\omega}$-admissible constant if $C$ can be chosen to depend only on $n$, $r$, $\eta_0$,
and upper bounds for $\ANorm{E}{r+2n}{\eta_0}$ and $\ANorm{F}{r+2n}{\eta_0}$.
\end{defn}

\begin{prop}\label{Prop::NirenbergCommute::MainProp}
There exist $\Zygad{s_0}$-admissible constants $\eta_3>0$, $K_1\geq 1$ and a map
$\Phi_3:B_{\R^r\times \C^n}(\eta_3)\rightarrow B_{\R^r\times \C^n}(\eta_0)$ such that:
\begin{enumerate}[(i)]
\item \begin{itemize}
\item If $s_0\in (0,\infty)$, $\Phi_3\in \ZygSpace{s_0+2}[B_{\R^r\times \C^n}(\eta_3)][\R^r\times \C^n]$ and $\ZygNorm{\Phi_3}{s+2}[B_{\R^r\times \C^n}(\eta_3)]\lesssim_{\Zygad{s}} 1$, $\forall s>0$.
\item If $s_0=\omega$, $\Phi_3\in \ASpace{r+2n}{2\eta_3}[\R^r\times \C^n]$ and $\ANorm{\Phi_3}{r+2n}{2\eta_3}\leq \eta_0$.
\end{itemize}
\item $\Phi_3(0)=0$ and $d_{t,x} \Phi_3(0)= K_1^{-1} I_{(r+2n)\times (r+2n)}$.
\item $\Phi_3(B_{\R^r\times \C^n}(\eta_3))\subseteq B_{\R^r\times \C^n}(\eta_0)$ is open and $\Phi_3:B_{\R^r\times \C^n}(\eta_3)\rightarrow \Phi_3(B_{\R^r\times \C^n}(\eta_3))$ is a diffeomorphism.
\item \begin{equation*}
\begin{bmatrix}
\diff{u} \\ \diff{\wb}
\end{bmatrix}
=K_1^{-1} (I+\AMatrix_2)
\begin{bmatrix}
\Phi_3^{*} X\\
\Phi_3^{*} L
\end{bmatrix},
\end{equation*}
where $\AMatrix_2:B_{\R^r\times \C^n}(\eta_3)\rightarrow \M^{(r+n)\times (r+n)}(\C)$, $\AMatrix_2(0)=0$, and:
\begin{itemize}
	\item If $s_0\in (0,\infty)$, $\ZygNorm{\AMatrix_2}{s+1}[B_{\R^r\times \C^n}(\eta_3)]\lesssim_{\Zygad{s}} 1$, $\forall s>0$.
	\item If $s_0=\omega$, $\ANorm{\AMatrix_2}{2n+r}{\eta_3}\lesssim 1$.
\end{itemize}
\end{enumerate}
\end{prop}

\begin{rmk}
If $s_0\in (0,\infty)$ we will show $\eta_3$ depends only on $n$, $r$, and $s_0$.  For $s_0=\omega$, we will take $K_1=1$.  This is not important
in the sequel, however.
\end{rmk}

\begin{proof}[Proof of \cref{Prop::NirenbergCommute::MainProp} when $s_0=\omega$]
Since $E\in \ASpace{r+2n}{\eta_0}[\M^{r\times n}]\subseteq \sBSpace{r+2n}{\M^{r\times n}}{\eta_0}$, $F\in \ASpace{r+2n}{\eta_0}[\M^{n\times n}]\subseteq \sBSpace{r+2n}{\M^{n\times n}}{\eta_0}$,
and
\begin{equation*}
\sBNorm{E}{r+2n}{\M^{r\times n}}{\eta_0} \leq \ANorm{E}{r+2n}{\eta_0}[\M^{r\times n}]\lesssim_{\Zygad{\omega}} 1\text{ and }\sBNorm{F}{r+2n}{\M^{n\times n}}{\eta_0}\leq \ANorm{F}{r+2n}{\eta_0}[\M^{n\times n}]\lesssim_{\Zygad{\omega}} 1,
\end{equation*}
we see that \cref{Prop::NirenbergRA::MainProp} applies to the vector fields $X_1,\ldots, X_r,L_1,\ldots, L_n$ and every constant which is admissible in the sense of that
proposition is $\Zygad{\omega}$-admissible here.

Thus, we obtain an $\Zygad{\omega}$-admissible constant $\eta_2>0$ and a map $\Phi_1:B_{\R^r\times \C^n}(\eta_2)\rightarrow B_{\R^r\times \C^n}(\eta_0)$
as in that proposition.  Letting $\AMatrix_1$ be the matrix from that proposition, and setting $\eta_3:=\eta_2/4$, we have (using \cref{Lemma::FuncSpaceReAnal::AinB})
\begin{equation*}
\ANorm{\Phi_1}{r+2n}{2\eta_3}\lesssim_{\Zygad{\omega}} \sBNorm{\Phi_1}{r+2n}{r+2n}{\eta_2}\lesssim_{\Zygad{\omega}} 1, \quad
\ANorm{\AMatrix_1}{r+2n}{\eta_3}\lesssim_{\Zygad{\omega}} \sBNorm{\AMatrix_1}{r+2n}{\M^{(n+r)\times(n+r)}}{\eta_2}\leq 1.
\end{equation*}

Taking $\Phi_3:=\Phi_1$, $\AMatrix_2:=\AMatrix_1$, and $K_1:=1$, all of the conclusions of \cref{Prop::NirenbergCommute::MainProp} now follow from the corresponding
conclusions in \cref{Prop::NirenbergRA::MainProp}.
\end{proof}

We now turn to the proof of \cref{Prop::NirenbergCommute::MainProp} when $s_0\in (0,\infty)$.  Because of the definition of $\Zygad{s}$-admissible constants,
it suffices to prove the result just for $s\in [s_0,\infty)$, and that is how we will proceed.  We begin with a lemma.

\begin{lemma}\label{Lemma::NirenbergCommute::MainLemma}
Define for $\gamma>0$, $\Psi_\gamma:B_{\R^r\times \C^n}(\eta_0/\gamma)\rightarrow B_{\R^r\times \C^n}(\eta_0)$ by
$\Psi_\gamma(t,z) = (\gamma t, \gamma z)$.  Let $X_k^{\gamma}:=\gamma \Psi_{\gamma}^{*}  X_k$ and $L_j^{\gamma}:=\gamma \Psi_{\gamma}^{*} L_j$.
Then,
\begin{equation}\label{Eqn::NirenbergCommute::XgammaForm}
X^{\gamma}= \diff{t} + E_\gamma \diff{z},\quad L^{\gamma} = \diff{\zb}+F_{\gamma} \diff{z},
\end{equation}
where $E_\gamma(0)=0$, $F_\gamma(0)=0$, and for $0<\gamma\leq \min\{\eta_0/2,1\}$, $s\in [s_0,\infty)$,
\begin{equation}\label{Eqn::NirenbergCommute::EstEgamma}
\ZygNorm{E_\gamma}{s+1}[B_{\R^r\times \C^n}(2)][\M^{r\times n}], \ZygNorm{F_\gamma}{s+1}[B_{\R^r\times \C^n}(2)][\M^{n\times n}]\lesssim_{\Zygad{s}} \gamma.
\end{equation}
Finally, $X_1^{\gamma},\ldots, X_r^{\gamma},L_1^{\gamma},\ldots, L_n^{\gamma}$ commute.
\end{lemma}
\begin{proof}
That $X_1^{\gamma},\ldots, X_r^{\gamma},L_1^{\gamma},\ldots, L_n^{\gamma}$ commute follows immediately from the same property of $X_1,\ldots, X_r, L_1,\ldots, L_n$.
Note that \cref{Eqn::NirenbergCommute::XgammaForm} holds with $E_\gamma(t,z) = E(\gamma t, \gamma z)$ and $F_\gamma(t,z)=F(\gamma t, \gamma z)$.
Thus, since $E(0)=0$ and $F(0)=0$, the same is true for $E_\gamma$ and $F_\gamma$, and we have for $0<\gamma\leq \min\{\eta_0/2,1\}$, using \cref{Lemma::FuncSpaceRev::Scale},
\begin{equation*}
\ZygNorm{E_{\gamma}}{s+1}[B_{\R^r\times \C^n}(2)]\leq 46 \gamma \ZygNorm{E}{s+1}[B_{\R^r\times \C^n}(\eta_0)]\lesssim_{\Zygad{s}} \gamma,
\end{equation*}
and similarly for $F_\gamma$.  This completes the proof.
\end{proof}

\begin{proof}[Proof of \cref{Prop::NirenbergCommute::MainProp} when $s_0\in (0,\infty)$]
Let $\sigma=\sigma(n,r,s_0)>0$ be the constant from \cref{Prop::NirenbergPertrub::MainProp}.
For $\gamma\leq \eta_0/2$, define $\Psi_\gamma$, $X^\gamma$, $L^\gamma$, $E_{\gamma}$, and $F_{\gamma}$ as in \cref{Lemma::NirenbergCommute::MainLemma}.
By \cref{Eqn::NirenbergCommute::EstEgamma}, if $\gamma\in (0,\eta_0/2]$ is a sufficiently small $\Zygad{s_0}$-admissible constant (without loss of generality, $\gamma\leq 1$), we
have
\begin{equation*}
\ZygNorm{E_\gamma}{s_0+1}[B_{\R^r\times \C^n}(2)], \ZygNorm{F_\gamma}{s_0+1}[B_{\R^r\times \C^n}(2)]\leq \sigma.
\end{equation*}
With this choice of $\gamma$,  \cref{Prop::NirenbergPertrub::MainProp} applies to the vector fields $X^\gamma$ and $L^\gamma$
to yield a constant $\eta_3=\eta_3(n,r,s_0)>0$ and a map $\Phi_2:B_{\R^r\times \C^n}(\eta_3)\rightarrow B_{\R^r\times \C^n}(2)$ as in that result,
and any constant which is $\Zygad{s}$-admissible in that proposition is $\Zygad{s}$-admissible in the sense of this section.
Set $\Phi_3:=\Psi_\gamma\circ \Phi_2:B_{\R^r\times \C^n}(\eta_3)\rightarrow B_{\R^r\times \C^n}(\eta_0)$.
We take $K_1:=\gamma^{-1}\geq 1$.
Since $\gamma$ is $\Zygad{s_0}$-admissible and $\ZygNorm{\Phi_2}{s+2}[B_{\R^r\times \C^n}(\eta_3)]\lesssim_{\Zygad{s}} 1$, $\forall s$ (by \cref{Prop::NirenbergPertrub::MainProp}),
we have $\ZygNorm{\Phi_3}{s+2}[B_{\R^r\times \C^n}(\eta_3)]\lesssim_{\Zygad{s}} 1$, $\forall s$.  Also, $\Phi_3(0)=\Psi_\gamma(\Phi_2(0))=\Psi_\gamma(0)=0$,
and $d_{t,x} \Phi_3(0) = \gamma d_{t,x} \Phi_2(0)= K_1^{-1} I_{(2n+r)\times (2n+r)}$.
That $\Phi_3$ is a diffeomorphism onto its image follows from the corresponding result about $\Phi_2$ in \cref{Prop::NirenbergPertrub::MainProp}.
Finally, if $\AMatrix_2$ is as in \cref{Prop::NirenbergPertrub::MainProp}, we have,
\begin{equation*}
\begin{bmatrix}
\diff{u} \\ \diff{\wb}
\end{bmatrix}
=(I+\AMatrix_2)
\begin{bmatrix}
\Phi_2^{*} X^{\gamma}\\
\Phi_2^{*} L^{\gamma}
\end{bmatrix}
=(I+\AMatrix_2) K_1^{-1} \begin{bmatrix}
\Phi_2^{*} \Psi_\gamma^{*} X\\
\Phi_2^{*} \Psi_\gamma^{*} L
\end{bmatrix}
=(I+\AMatrix_2) K_1^{-1} \begin{bmatrix}
\Phi_3^{*} X\\
\Phi_3^{*}  L
\end{bmatrix}.
\end{equation*}
All of the desired estimates for $\AMatrix_2$ are stated in \cref{Prop::NirenbergPertrub::MainProp} and this completes the proof.
\end{proof}

    \subsection{Proof of \texorpdfstring{\cref{Prop::Nirenberg::MainProp}}{the main proposition}}\label{Section::Nirenberg::Final}
Using the matrix notation of \cref{Section::Nirenberg::Additional} we may write
\begin{equation*}
X=\diff{t} + B_1 \diff{t}+B_2\diff{z}+B_3\diff{\zb},\quad L=\diff{\zb}+B_4\diff{t}+B_5\diff{z}+B_6\diff{\zb},
\end{equation*}
where each $B_l$ takes values in matrices of an appropriate size, $B_l(0)=0$ for each $l$, and
\begin{itemize}
	\item If $s_0\in (0,\infty)$, $\ZygNorm{B_l}{s+1}[B_{\R^r\times \C^n}(1)]\lesssim_{\Zygad{s}} 1$, $\forall s>0$.
	\item If $s_0=\omega$, $\ANorm{B_l}{r+2n}{1}\lesssim_{\Zygad{\omega}} 1$.
\end{itemize}
Define $M$ to be the $(r+n)\times (r+n)$ matrix:
\begin{equation*}
M:=\mleft[
\begin{array}{c|c}
B_1 & B_3 \\
\hline
B_4 & B_6
\end{array}
\mright].
\end{equation*}
We have
\begin{equation*}
\begin{bmatrix}
X \\ L
\end{bmatrix}
=
(I+M)
\begin{bmatrix}
\diff{t} \\ \diff{\zb}
\end{bmatrix}
+\begin{bmatrix}
B_2 \\ B_5
\end{bmatrix}
\diff{z},
\end{equation*}
and $M(0)=0$.
\begin{itemize}
\item If $s_0\in (0,\infty)$, we have
$\ZygNorm{M}{s_0+1}[B_{\R^r\times \C^n}(1)] \lesssim_{\Zygad{s_0}} 1$.
Thus,
by taking $\eta_0>0$ to be a sufficiently small $\Zygad{s_0}$-admissible constant and using that $M(0)=0$,
we have
\begin{equation*}
\inf_{\zeta\in B_{\R^r\times \C^n}(\eta_0)} | \det (I+M(\zeta))|\geq \frac{1}{2}.
\end{equation*}
\Cref{Rmk::FuncSpaceRev::InverseMatrix} shows that
$\ZygNorm{(I+M)^{-1}}{s+1}[B_{\R^r\times \C^n}(\eta_0)]\lesssim_{\Zygad{s}} 1$.
\item If $s_0=\omega$, we have
$\ANorm{M}{2n+r}{1}\lesssim_{\Zygad{\omega}} 1$.
Since $M(0)=0$, \cref{Lemma::FuncSpaceReAnal::Restrict} implies
$\ANorm{M}{2n+r}{\eta_0}
\lesssim_{\Zygad{\omega}} \eta_0$, for $\eta_0\in (0,1]$.
Thus, by taking $\eta_0>0$ to be a sufficiently
small $\Zygad{\omega}$-admissible constant we have 
\begin{equation*}
\ANorm{M}{2n+r}{\eta_0}[\M^{(r+n)\times (r+n)}]
\leq \frac{1}{2}.
\end{equation*}
Since $\ASpace{2n+r}{\eta_0}[\M^{(r+n)\times (r+n}]$
is a Banach algebra
(\cref{Prop::FuncSpaceRev::Algebra}) it follows
that
$\ANorm{(I+M)^{-1}}{2n+r}{\eta_0}[\M^{(r+n)\times (r+n}]\leq 2$; here we have used the Neumann series for $(1+M)^{-1}$.
\end{itemize}
In either case we have an $\Zygad{s_0}$-admissible constant $\eta_0>0$ so that $(I+M)^{-1}$
satisfies
good estimates on $B_{\R^r\times \C^n}(\eta_0)$.

Define vector fields $\Xh_1,\ldots, \Xh_r,\Lh_1,\ldots, \Lh_n$ on $B_{\R^r\times \C^n}(\eta_0)$ by
\begin{equation*}
\begin{bmatrix} \Xh \\ \Lh\end{bmatrix}
=(I+M)^{-1} \begin{bmatrix} X \\L\end{bmatrix}
=\begin{bmatrix} \diff{t} \\ \diff{\zb} \end{bmatrix} + (I+M)^{-1} \begin{bmatrix} B_2 \\ B_5 \end{bmatrix}\diff{z}.
\end{equation*}
Thus, we have
\begin{equation*}
\Xh = \diff{t} +\Eh \diff{z}, \quad \diff{\zb}+\Fh\diff{z},
\end{equation*}
where $\Eh(0)=0$, $\Fh(0)=0$ and using \cref{Prop::FuncSpaceRev::Algebra} and the bounds for $(I+M)^{-1}$,
\begin{itemize}
\item If $s_0\in (0,\infty)$, $\ZygNorm{\Eh}{s+1}[B_{\R^r\times \C^n}(\eta_0)],\ZygNorm{\Fh}{s+1}[B_{\R^r\times \C^n}(\eta_0)]\lesssim_{\Zygad{s}} 1$, $\forall s>0$.
\item If $s_0=\omega$, $\ANorm{\Eh}{2n+r}{\eta_0}, \ANorm{\Fh}{2n+r}{\eta_0}\lesssim_{\Zygad{\omega}} 1$.
\end{itemize}
Furthermore, we have $\forall \zeta\in B_{\R^r\times \C^n}(\eta_0)$,
\begin{equation*}
[\Xh_{k_1}, \Xh_{k_2}](\zeta), [\Lh_{j_1}, \Lh_{j_2}](\zeta), [\Xh_{k},\Lh_j](\zeta)\in \Span_{\C}\{ \Xh_1(\zeta),\ldots, \Xh_r(\zeta), \Lh_1(\zeta),\ldots, \Lh_n(\zeta)\},
\end{equation*}
which follows from the corresponding assumption on the $X$s and $L$s (and the fact that $(I+M)^{-1}$ is an invertible matrix).

\Cref{Prop::NirenbergCommute::MainProp} applies to the vector fields $\Xh,\Lh$, and any constant which is $\Zygad{s}$-admissible
in the sense of that proposition is $\Zygad{s}$-admissible in the sense of this section.  We obtain $\Zygad{s_0}$-admissible
constants $\eta_3>0$, $K_1\geq 1$, a map
$\Phi_3:B_{\R^r\times \C^n}(\eta_3)\rightarrow B_{\R^r\times \C^n}(\eta_0)$,
and a matrix $\AMatrix_2:B_{\R^r\times \C^n}(\eta_3)\rightarrow \M^{(r+n)\times (r+n)}(\C)$ as in that proposition.
\Cref{Item::Nirenberg::Phi3Reg}, \cref{Item::Nirenberg::MainProp::0anddet}, and \cref{Item::Nirenberg::MainProp::Diffeo}
follow immediately from the corresponding results in \cref{Prop::NirenbergCommute::MainProp}.

Next we establish \cref{Item::Niremberg::MainProp::AEst}.  We have, from \cref{Prop::NirenbergCommute::MainProp},
\begin{equation*}
\begin{bmatrix}
\diff{u} \\ \diff{\wb}
\end{bmatrix}
=K_1^{-1} (I+\AMatrix_2) \begin{bmatrix} \Phi_3^{*} \Xh \\ \Phi_3^{*} \Lh\end{bmatrix}
=K_1^{-1} (I+\AMatrix_2) (I+M\circ \Phi_3)^{-1}\begin{bmatrix} \Phi_3^{*} X \\ \Phi_3^{*} L\end{bmatrix}
=:K_1^{-1} (I+\AMatrix_3) \begin{bmatrix} \Phi_3^{*} X \\ \Phi_3^{*} L\end{bmatrix},
\end{equation*}
where $I+\AMatrix_3 := (I+\AMatrix_2) (I+M\circ \Phi_3)^{-1}$.  Since $M(0)=0$, $\Phi_3(0)=0$, and $\AMatrix_2(0)=0$, we have
$\AMatrix_3(0)=0$.  Also, we have
\begin{itemize}
\item If $s_0\in (0,\infty)$, since $\ZygNorm{(I+M)^{-1}}{s+1}[B_{\R^r\times \C^n}(\eta_0)]\lesssim_{\Zygad{s}} 1$,
$\ZygNorm{\Phi_3}{s+2}[B_{\R^r\times \C^n}(\eta_3)]\lesssim_{\Zygad{s}} 1$, and $\Phi_3(B_{\R^r\times \C^n}(\eta_3))\subseteq B_{\R^r\times \C^n}(\eta_0)$,
it follows from \cref{Lemma::FuncSpaceRev::Composition} that $\ZygNorm{(I+M\circ \Phi_3)^{-1}}{s+1}[B_{\R^r\times \C^n}(\eta_3)]\lesssim_{\Zygad{s}} 1$.
Combining this with $\ZygNorm{\AMatrix_2}{s+1}[B_{\R^r\times \C^n}(\eta_3)]\lesssim_{\Zygad{s}} 1$ (see \cref{Prop::NirenbergCommute::MainProp}),
\cref{Prop::FuncSpaceRev::Algebra} implies $\ZygNorm{\AMatrix_3}{s+1}[B_{\R^r\times \C^n}(\eta_3)]\lesssim_{\Zygad{s}} 1$.

\item If $s_0=\omega$, since $\ANorm{(I+M)^{-1}}{2n+r}{\eta_0}\leq 2$ and since $\ANorm{\Phi_3}{r+2n}{\eta_3}\leq \eta_0$, it follows from
\cref{Lemma::FuncSpaceRev::ComposeAnal}
that $\ANorm{(I+M\circ \Phi_3)^{-1}}{r+2n}{\eta_3}\leq 2$.  Since $\ANorm{\AMatrix_2}{r+2n}{\eta_3}\lesssim_{\Zygad{\omega}} 1$ (see \cref{Prop::NirenbergCommute::MainProp}),
\cref{Prop::FuncSpaceRev::Algebra} implies $\ANorm{\AMatrix_3}{r+2n}{\eta_3}\lesssim_{\Zygad{\omega}} 1$.
\end{itemize}
The above comments complete the proof.


\section{Proof of the Main Result}\label{Section::Proofs}
In this section, we prove \cref{Thm::EMfld::MainThm}; \cref{Thm::Intro::MainThm} is an immediate consequence of \cref{Thm::EMfld::MainThm}.
Throughout this section, fix $s\in (0,\infty]\cup \{\omega\}$ and let $M$ be a $\ZygSpace{s+2}$ manifold.  As in the rest of the paper, we give $\R^r\times \C^n$ coordinates
$(t_1,\ldots, t_r, z_1,\ldots, z_n)$.

\begin{lemma}\label{Lemma::Proof::GoodAt0}
Let $\LVS$ be a $\ZygSpace{s+1}$ elliptic structure on $M$ of dimension $(r,n)$.  Then, $\forall \zeta_0\in M$, there exists a neighborhood
$V_0$ of $\zeta_0$, $\ZygSpace{s+1}$ sections $L_1,\ldots, L_n,X_1,\ldots, X_r$ of $\LVS$ over $V_0$, and a $\ZygSpacediff{s+2}$ diffeomorphism
$\Psi_0:B_{\R^r\times \C^n}(1)\rightarrow V_0$ such that:
\begin{enumerate}[(i)]
\item $\Psi_0(0)=\zeta_0$.
\item $\forall \zeta\in V_0$, $L_1(\zeta),\ldots, L_n(\zeta), X_1(\zeta),\ldots, X_r(\zeta)$ is a basis for $\LVS[\zeta]$.
\item $\Psi_0^{*}L_j(0)=\diff{\zb[j]}$, $\Psi_0^{*} X_k(0)=\diff{t_k}$, $1\leq j\leq n$, $1\leq k\leq r$.
\item For $1\leq j \leq n$, $1\leq k\leq r$,
\begin{itemize}
\item If $s\in (0,\infty]$, $\Psi_0^{*} L_j, \Psi_0^{*} X_k\in \ZygSpace{s+1}[B_{\R^r\times \C^n}(1)][\C^{2n+r}]$.
\item If $s=\omega$, $\Psi_0^{*} L_j, \Psi_0^{*} X_k \in \ASpace{2n+r}{1}[\C^{2n+r}]$.
\end{itemize}
\item\label{Item::Proof::CommutatorFirst} $\forall j_1, j_2, k_1, k_2, j,k$, $\forall \xi\in B_{\R^r\times \C^n}(1)$,
\begin{equation*}
\begin{split}
&[\Psi_0^{*}L_{j_1}, \Psi_0^{*}L_{j_2}](\xi), [\Psi_0^{*}X_k, \Psi_0^{*}L_j](\xi), [\Psi_0^{*}X_{k_1}, \Psi_0^{*}X_{k_2}](\xi)
\\&\in \Span_{\C} \{ \Psi_0^{*}L_1(\xi), \ldots, \Psi_0^{*}L_n(\xi), \Psi_0^{*}X_1(\xi),\ldots, \Psi_0^{*}X_r(\xi)  \}.
\end{split}
\end{equation*}
\end{enumerate}
\end{lemma}
\begin{proof}
Note that, by the definition of elliptic structures of dimension $(r,n)$, we have $\dim \LVS[\zeta]=n+r$, $\forall \zeta\in M$ and $\dim M=2n+r$ (see \cref{Rmk::Bundles::DimensionMfld}).
By \cref{Lemma::AppendVS::RealBasis} we may pick a basis $y_1,\ldots, y_r$ of $\LVS[\zeta_0]\cap \LVSb[\zeta_0]$ with $y_1,\ldots, y_r\in T_{\zeta_0} M$ (i.e., $y_1,\ldots, y_r$ are \textit{real}).
Extend $y_1,\ldots, y_r$ to a basis $l_1,\ldots, l_n, y_1,\ldots, y_r$ of $\LVS[\zeta_0]$.

By the definition of a $\ZygSpace{s+1}$ bundle, we may find a neighborhood $U_1$ of $\zeta_0$ and $\ZygSpace{s+1}$ sections $Z_1,\ldots, Z_K$ of $\LVS$ over $U_0$ such that
$\forall \zeta\in U_0$, $\Span_{\C} \{Z_1(\zeta),\ldots, Z_K(\zeta)\}= \LVS[\zeta]$.  Without loss of generality, reorder $Z_1,\ldots, Z_K$ so that $Z_1(\zeta_0),\ldots, Z_{n+r}(\zeta_0)$ form
a basis of $\LVS[\zeta_0]$.  By continuity, there exists a neighborhood $U_2\subseteq U_1$ of $\zeta_0$ such that
$\forall \zeta\in U_2$, $Z_1(\zeta),\ldots, Z_{n+r}(\zeta)$ are linearly independent.  We conclude $\forall \zeta\in U_2$, $Z_1(\zeta),\ldots, Z_{n+r}(\zeta)$ forms a basis for $\LVS[\zeta]$.

Let $M\in \M^{(n+r)\times (n+r)}(\C)$ be the invertible matrix such that
\begin{equation*}
M
\begin{bmatrix}
Z_1(\zeta_0) \\
\vdots\\
Z_{n+r}(\zeta_0)
\end{bmatrix}
=
\begin{bmatrix}
y_1\\
\vdots\\
y_r\\
l_1\\
\vdots\\
l_n
\end{bmatrix}.
\end{equation*}
For $\zeta\in U_2$ set
\begin{equation*}
\begin{bmatrix}
\Xh_1(\zeta)\\
\vdots\\
\Xh_r(\zeta)\\
\Lh_1(\zeta)\\
\vdots\\
\Lh_n(\zeta)
\end{bmatrix}:=
M
\begin{bmatrix}
Z_1(\zeta) \\
\vdots\\
Z_{n+r}(\zeta)
\end{bmatrix}.
\end{equation*}
Since $M$ is a (constant) invertible matrix, we have $\forall \zeta\in U_2$, $\Lh_1(\zeta),\ldots, \Lh_n(\zeta), \Xh_1(\zeta),\ldots, \Xh_r(\zeta)$ forms a basis for $\LVS[\zeta]$,
and $\Lh_1,\ldots, \Lh_n, \Xh_1,\ldots, \Xh_r$ are $\ZygSpace{s+1}$ sections of $\LVS$ over $U_2$.

By the definition of a $\ZygSpace{s+2}$ manifold (see also \cref{Rmk::FuncMflds::CoordChartsAreDiffeo}) there exists a $\ZygSpacediff{s+2}$ diffeomorphism
$\Psi_1:B_{\R^{2n+r}}(\epsilon_1)\rightarrow V_1$, where $V_1\subseteq U_2$ is a neighborhood of $\zeta_0$, $\Psi_1(0)=\zeta_0$.  Since $\Xh_1(\zeta_0)=y_1,\ldots, \Xh_r(\zeta_0)=y_r$ are real,
 since $\LVS[\zeta_0]+\LVSb[\zeta_0]=\C T_{\zeta_0}M$, and since $\Xh_1(\zeta_0),\ldots, \Xh_r(\zeta_0), \Lh_1(\zeta_0),\ldots, \Lh_n(\zeta_0)$ forms a basis for $\LVS[\zeta_0]$,
 we have 
 \begin{equation}\label{Eqn::Proof::TheySpanRealTangent}
 \Span_{\R} \{2 \Real(\Lh_1)(\zeta_0), \ldots, 2\Real(\Lh_n)(\zeta_0), 2\Imag(\Lh_1)(\zeta_0),\ldots, 2\Imag(\Lh_n)(\zeta_0),\Xh_1(\zeta_0),\ldots, \Xh_n(\zeta_0) \} = T_{\zeta_0} M.
 \end{equation}
 Pulling \cref{Eqn::Proof::TheySpanRealTangent} back via $\Psi_1$ we have 
 $$2 \Real(\Psi_1^{*} \Lh_1)(0), \ldots, 2\Real(\Psi_1^{*}\Lh_n)(0), 2\Imag(\Psi_1^{*}\Lh_1)(0),\ldots, 2\Imag(\Psi_1^{*}\Lh_n)(0),\Psi_1^{*}\Xh_1(0),\ldots, \Psi_1^{*}\Xh_n(0)$$
 forms a basis for $T_0 \R^{2n+r}$.
 
 We give $\R^{2n+r}\cong \R^{r}\times \R^{2n}$ coordinates $(t_1,\ldots, t_r, x_1,\ldots, x_{2n})$.  Let $C\in \M^{(r+2n)\times (r+2n)}(\R)$ denote the (constant) invertible matrix such that
 \begin{equation*}
 C\begin{bmatrix}
 \diff{t_1}\\
 \vdots\\
 \diff{t_r}\\
 \diff{x_1}\\
 \vdots\\
 \diff{x_{2n}}
 \end{bmatrix}
 =
 \begin{bmatrix}
 \Psi_1^{*} \Xh_1(0) \\
 \vdots\\
 \Psi_1^{*} \Xh_r(0)\\
 2\Real(\Psi_1^{*} \Lh_1)(0)\\
 \vdots\\
 2\Real(\Psi_1^{*} \Lh_n)(0)\\
 2\Imag(\Psi_1^{*} \Lh_1)(0)\\
 \vdots\\
  2\Imag(\Psi_1^{*} \Lh_n)(0)
 \end{bmatrix}.
 \end{equation*}
 Set $A=C^{\transpose}$ and we identify $A$ with the corresponding invertible linear transformation $\R^{r+2n}\rightarrow \R^{r+2n}$.  Then for $\epsilon_2>0$ sufficiently small, we set
 $\Psi_2:=\Psi_1\circ A : B_{\R^{r+2n}}(\epsilon_2)\rightarrow V_1$.  Then, $\Psi_2(0)=\Psi_1(0)=\zeta_0$, $\Psi_2$ is a $\ZygSpacediff{s+2}$ diffeomorphism onto its image (which is a neighborhood of $\zeta_0$),
 and if we identify $\R^{r+2n}\cong \R^r\times \C^n$ via the map $(t_1,\ldots, t_r, x_1,\ldots, x_{2n})\mapsto (t_1,\ldots, t_r, x_1+i x_n,\ldots, x_n+ix_{2n})$, then,
 \begin{equation*}
 \Psi_2^{*} \Lh_j(0) = \diff{\zb_j}, \quad \Psi_2^{*} \Xh_k(0) = \diff{t_k}, \quad 1\leq j\leq n, 1\leq k\leq r,
 \end{equation*}
 \begin{equation*}
 \Psi_2^{*}\Lh_j, \Psi_2^{*} \Xh_k \in \ZygSpacemap{s+1}[B_{\R^r\times \C^n}(\epsilon_2)][\C^{2n+r}].
 \end{equation*}
 
 Take $\epsilon_3\in (0,\epsilon_2)$ such that, $\forall 1\leq j\leq n, 1\leq k\leq r$,
 \begin{itemize}
 \item If $s_0\in (0,\infty]$, $\Psi_2^{*} \Lh_j, \Psi_2^{*}\Xh_k \in \ZygSpace{s+1}[B_{\R^{r}\times \C^n}(\epsilon_3)][\C^{2n+r}]$.
 \item If $s_0=\omega$, $\Psi_2^{*} \Lh_j, \Psi_2^{*}\Xh_k \in \ASpace{2n+r}{\epsilon_3}[\C^{2n+r}]$.
 \end{itemize}
 Define $D_{\epsilon_3}:\R^{r}\times \C^n \rightarrow \R^{r}\times \C^n$ by $D_{\epsilon_3} (t,z) = (\epsilon_3 t, \epsilon_3 z)$, and define
 $\Psi_0:B_{\R^r\times \C^n}(1)\rightarrow V_1$ by $\Psi_0 := \Psi_2\circ D_{\epsilon_3}$.  Letting $V_0=\Psi_0(B_{\R^r\times \C^n}(1))\subseteq V_1$ we
 have $\Psi_0:B_{\R^r\times \C^n}(1)\rightarrow V_0$ is a $\ZygSpacediff{s+2}$ diffeomorphism.  Set $L_j:=\epsilon_3 \Lh_j$ and $X_k:=\epsilon_3 \Xh_k$.
 With these choices, all of the conclusions of the lemma follow from the above remarks.
 
 We include a few additional comments regarding the proof of \cref{Item::Proof::CommutatorFirst}.  Since $\forall \zeta\in V_0\subseteq V_1\subseteq U_2$, we have $L_1(\zeta),\ldots, L_n(\zeta), X_1(\zeta),\ldots, X_r(\zeta)$
 form a basis for $\LVS[\zeta]$, we have $\forall \zeta\in V_0$, $\forall j_1, j_2, k_1, k_2, j,k$,
 \begin{equation*}
 [L_{j_1}, L_{j_2}](\zeta), [X_k, L_j](\zeta), [X_{k_1},X_{k_2}](\zeta)\in \LVS[\zeta]=\Span_{\C} \{ L_1(\zeta),\ldots, L_n(\zeta), X_1(\zeta),\ldots, X_r(\zeta)\}. 
 \end{equation*}
 Pulling this back via $\Psi_0$ yields \cref{Item::Proof::CommutatorFirst} and completes the proof.
\end{proof}

\begin{lemma}\label{Lemma::Proof::ChartAtPoint}
Let $\LVS$ be a $\ZygSpace{s+1}$ elliptic structure  on $M$ of dimension $(r,n)$.  Then, $\forall \zeta_0\in M$, there exists a neighborhood $V$ of $\zeta_0$,
$\ZygSpace{s+1}$ sections $L_1,\ldots, L_n, X_1,\ldots, X_r$ of $\LVS$ over $V$, and a $\ZygSpacediff{s+2}$ diffeomorphism $\Psi:B_{\R^r\times \C^n}(1)\rightarrow V$
such that
\begin{itemize}
\item $\Psi(0)=\zeta_0$.
\item $\forall \zeta\in V$, $L_1(\zeta),\ldots, L_n(\zeta), X_1(\zeta),\ldots, X_r(\zeta)$ is a basis for $\LVS[\zeta]$.
\item $\forall \xi\in B_{\R^r\times \C^n}(1)$,
\begin{equation*}
	\Span_{\C} \{ \Psi^{*}L_1(\xi),\ldots, \Psi^{*} L_n(\xi), \Psi^{*} X_1(\xi),\ldots, \Psi^{*} X_r(\xi)\} = \Span_{\C} \mleft\{ \diff{\zb[1]},\ldots, \diff{\zb[n]}, \diff{t_1},\ldots, \diff{t_n}\mright\}.
\end{equation*}
\end{itemize}
\end{lemma}
\begin{proof}
Let $L_1,\ldots, L_n, X_1,\ldots, X_r$ and $\Psi_0:B_{\R^r\times \C^n}(1)\rightarrow V_0$ be as in \cref{Lemma::Proof::GoodAt0}.
If $s\in (0,\infty)\cup \{\omega\}$, set $s_0:=s$.  If $s=\infty$, set $s_0:=1$.
The conclusions of \cref{Lemma::Proof::GoodAt0} show that \cref{Thm::Nirenberg::MainThm} applies (with this choice of $s_0$) to the vector fields
$\Psi_0^{*}L_1,\ldots, \Psi_0^{*} L_n, \Psi_0^{*}X_1,\ldots, \Psi_0^{*} X_r$ and yields $\Phi_4\in \ZygSpace{s+2}[B_{\R^r\times \C^n}(1)][\R^r\times \C^n]$ as in that theorem.
In particular, $\Phi_4$ is a diffeomorphism onto its image, $\Phi_4(0)=0$, and since $I+\AMatrix(\xi)$ from \cref{Thm::Nirenberg::MainThm} \cref{Item::Nirenberg::AMatrix} is invertible,
$\forall \xi\in B_{\R^r\times \C^n}(1)$,  \cref{Thm::Nirenberg::MainThm} \cref{Item::Nirenberg::AMatrix} shows $\forall \xi\in B_{\R^r\times \C^n}(1)$,
\begin{equation*}
\Span_{\C} \{ \Phi_4^{*} \Psi_0^{*} L_1(\xi),\ldots, \Phi_4^{*} \Psi_0^{*} L_n(\xi), \Phi_4^{*} \Psi_0^{*} X_1(\xi),\ldots, \Phi_4^{*} \Psi_0^{*} X_n(\xi)  \}
=\Span_{\C}\mleft\{ \diff{\zb[1]},\ldots, \diff{\zb[n]}, \diff{t_1},\ldots, \diff{t_r} \mright\}.
\end{equation*}
Setting $\Psi:=\Psi_0\circ \Phi_4$, the result follows with $V:=\Psi(B_{\R^r\times \C^n}(1))\subseteq V_0$, by using the above mentioned properties of $\Phi_4$ combined
with the conclusions of \cref{Lemma::Proof::GoodAt0}.
\end{proof}

\begin{proof}[Proof of \cref{Thm::EMfld::MainThm}]
\cref{Item::MainThm::EMfld}$\Rightarrow$\cref{Item::MainThm::Bundle}:  This is obvious.

\cref{Item::MainThm::Bundle}$\Rightarrow$\cref{Item::MainThm::EMfld}: 
Let $\LVS$ be a $\ZygSpace{s+1}$ elliptic structure on $M$ of dimension $(r,n)$.
We wish to construct a $\ZygSpace{s+2}$ E-atlas on $M$ of dimension $(r,n)$, compatible with its $\ZygSpace{s+2}$ structure, such that $\LVS$ is the $\ZygSpace{s+1}$
elliptic structure associated to this E-manifold structure.  For each $\zeta_0\in M$,  let $\Psi_{\zeta_0}:B_{\R^{r}\times \C^n}(1)\rightarrow V_{\zeta_0}$ be the function
$\Psi$ from \cref{Lemma::Proof::ChartAtPoint} with this choice of $\zeta_0$; so that $V_{\zeta_0}$ is a neighborhood of $\zeta_0$ and $\Psi_{\zeta_0}$ is a $\ZygSpacediff{s+2}$ diffeomorphism
satisfying the conclusions of that lemma.  In particular, it follows from that lemma that $\forall \zeta\in V_{\zeta_0}$,
\begin{equation}\label{Eqn::MainThm::BundleEquals}
  \LVS[\zeta] = \Span_{\C} \mleft\{ d\Psi_{\zeta_0}(\Psi_{\zeta_0}^{-1}(\zeta)) \diff{t_1},\ldots, d\Psi_{\zeta_0}(\Psi_{\zeta_0}^{-1}(\zeta)) \diff{t_r}, d\Psi_{\zeta_0}(\Psi_{\zeta_0}^{-1}(\zeta)) \diff{\zb[1]},\ldots, d\Psi_{\zeta_0}(\Psi_{\zeta_0}^{-1}(\zeta)) \diff{\zb[n]} \mright\}.
\end{equation}
We claim $\{ (\Psi_{\zeta_0}^{-1}, V_{\zeta_0}) : \zeta_0\in M\}$ is the desired atlas.  Indeed, that $\Psi_{\zeta_1}^{-1} \circ \Psi_{\zeta_2}$ is a $\ZygSpacemap{s+2}$ map follows from \cref{Lemma::FuncMfld::CompositionOfMaps,Lemma::FuncManfiold::InverseMap}.  To see that $\Psi_{\zeta_1}^{-1} \circ \Psi_{\zeta_2}$ is an E-map, note that, for $1\leq k\leq r$,
$$d\mleft(\Psi_{\zeta_1}^{-1} \circ \Psi_{\zeta_2} \mright)(\xi) \diff{t_k} = d\Psi_{\zeta_1}^{-1}( \Psi_{\zeta_2}(\xi)) d \Psi_{\zeta_2}(\xi) \diff{t_k}.$$
\Cref{Eqn::MainThm::BundleEquals} shows $d \Psi_{\zeta_2}(\xi) \diff{t_k} \in \LVS[\Psi_{\zeta_2}(\xi)]$, and applying \cref{Eqn::MainThm::BundleEquals} again
shows
$$d\Psi_{\zeta_1}^{-1}( \Psi_{\zeta_2}(\xi)) d \Psi_{\zeta_2}(\xi) \diff{t_k} \in \Span_{\C} \mleft\{\diff{t_1},\ldots, \diff{t_r}, \diff{\zb[1]},\ldots, \diff{\zb[n]}\mright\}.$$
Similarly, for $1\leq j\leq n$,
\begin{equation*}
d\mleft(\Psi_{\zeta_1}^{-1} \circ \Psi_{\zeta_2}\mright)(\xi) \diff{\zb[j]}\in \Span_{\C} \mleft\{\diff{t_1},\ldots, \diff{t_r}, \diff{\zb[1]},\ldots, \diff{\zb[n]}\mright\}.
\end{equation*}
It follows that $\Psi_{\zeta_1}^{-1} \circ \Psi_{\zeta_2}$ is an E-map.  Thus, since $\{ V_{\zeta_0} :\zeta_0\in M\}$ is an open cover for $M$
we have that $\{ (\Psi_{\zeta_0}^{-1}, V_{\zeta_0}) : \zeta_0\in M\}$ is a $\ZygSpace{s+2}$ E-atlas on $M$.  Since each $\Psi_{\zeta_0}:B_{\R^{r}\times \C^n}(1)\rightarrow V_{\zeta_0}$
is a $\ZygSpacediff{s+2}$ diffeomorphism (by \cref{Lemma::Proof::ChartAtPoint}, where $V_{\zeta_0}\subseteq M$ is given the original $\ZygSpace{s+2}$ manifold structure),
we see that the $\ZygSpace{s+2}$ E-manifold structure induced by the above atlas is compatible with the original $\ZygSpace{s+2}$ manifold structure on $M$.
That $\LVS$ is the $\ZygSpace{s+1}$ elliptic structure associated to this E-manifold structure follows from \cref{Eqn::MainThm::BundleEquals}.

Finally, we turn to the uniqueness of this E-manifold structure.  Suppose $M$ is given two $\ZygSpace{s+2}$ E-manifold structures, compatible with the $\ZygSpace{s+2}$ manifold structure, such that $\LVS$
is the $\ZygSpace{s+1}$ elliptic structure associated to both of these E-manifold structures.  That the identity map $M\rightarrow M$ is a $\ZygSpacediff{s+2}$ diffeomorphism follows immediately because
both copies of $M$ have the same underlying $\ZygSpace{s+2}$ manifold structure.  That the identity map is an E-map follows from \cref{Lemma::EMfld::RecongnizeEMap}.  This shows that
the identity map is 
a $\ZygSpacediff{s+2}$ E-diffeomorphism, which completes the proof.
\end{proof}


\appendix
\section{Linear Algebra}\label{Appendix::LinearAlgebra}
Let $\VVS$ be a real vector space and let $\VVS^{\C}=\VVS\otimes_{\R} \C$ be its complexification.  We consider $\VVS\hookrightarrow \VVS^{\C}$ as a real subspace
by identifying $v$ with $v\otimes 1$.  There are natural maps:
\begin{equation*}
\Real:\VVS^{\C}\rightarrow \VVS, \quad \Imag:\VVS^{\C}\rightarrow \VVS,\quad \text{complex conjugation}:\VVS^{\C}\rightarrow \VVS^{\C},
\end{equation*}
defined as follows.  Every $v\in \VVS^{\C}$ can be written uniquely as $v=v_1\otimes 1+ v_2\otimes i$, with $v_1,v_2\in V$.  Then,
$\Real(v):=v_1$, $\Imag(v):=v_2$, and $\overline{v}:=v_1\otimes 1 - v_2\otimes i$.

\begin{lemma}\label{Lemma::AppendCR::dimFormula}
Let $\LVS\subseteq \VVS^{\C}$ be a finite dimensional complex subspace.  Then,
$\dim (\LVS+\overline{\LVS}) + \dim (\LVS\bigcap \overline{\LVS}) = 2\dim(\LVS).$
\end{lemma}
\begin{proof}
It is a standard fact that $\dim (\LVS+\overline{\LVS}) + \dim (\LVS\bigcap \overline{\LVS}) = \dim(\LVS)+\dim(\overline{\LVS})$.
Using that $w\mapsto \overline{w}$, $\LVS\rightarrow \overline{\LVS}$ is an anti-linear isomorphism, the result follows.
\end{proof}

\begin{lemma}\label{Lemma::AppendVS::RealBasis}
Let $\XVS\subseteq \VVS^{\C}$ be a finite dimensional subspace of dimension $r$, and suppose $\XVSb=\XVS$.  Then there exist $x_1,\ldots, x_r\in \XVS \cap V$
such that $x_1,\ldots, x_r$ is a basis for $\XVS$.
\end{lemma}
\begin{proof}
Let $l_1,\ldots, l_r$ be a basis for $\XVS$.  Since $\XVS=\XVSb$, $\Real(l_j), \Imag(l_j)\in \XVS$, and clearly
$\Real(l_1),\ldots, \Real(l_r), \Imag(l_1),\ldots, \Imag(l_r)$ form a spanning set for $\XVS$.  Extracting a basis from this spanning set yields the result.
\end{proof}

\section{Elliptic PDEs}
In this section, we state quantitative versions of some standard results regarding nonlinear ellipic PDEs.
All of the results in this section are well-known, and we make no effort to state these results in the greatest possible generality:
we content ourselves with the simplest settings which are sufficient for our purposes.

	\subsection{Real Analyticity for a Nonlinear Elliptic Equation}
It is a classical result that the solutions to real analytic, nonlinear elliptic PDEs are themselves real analytic; see, e.g., \cite{MorreyOnTheAnalyticityOfTheSolutionsOfAnalyticNonLinearEllipticSystems}.
We require a quantitative version of (a special case of) this fact, which follows from standard proofs.

Let $\sE$ be a constant coefficient, first order, linear partial differential operator
\begin{equation*}
\sE:\CjSpace{\infty}[\R^n][\C^{m_1}]\rightarrow \CjSpace{\infty}[\R^n][\C^{m_2}],
\end{equation*}
where $m_2\geq m_1$.  We may think of $\sE$ as an $m_2\times m_1$ matrix of constant coefficient partial differential operators of order $\leq 1$.

Let $\Gamma:\C^{m_1}\times \C^{nm_1}\rightarrow \C^{m_2}$ be a bilinear map.  Fix $R>0$ and we consider
the equation for $b:B_{\R^n}(R)\rightarrow \C^{m_1}$ given by
\begin{equation}\label{Eqn::AppendRealAnal::MainEqn}
\sE b = \Gamma(b, \grad b).
\end{equation}

\begin{prop}\label{Prop::AppendRealAnal::MainProp}
Fix $s_0>1$ and suppose $\sE$ is elliptic.  Then, $\exists \gamma=\gamma(\sE, \Gamma, R, s_0)>0$, $\eta_0=\eta_0(\sE,\Gamma,R,s_0)>0$
such that the following holds.
If $b\in \ZygSpace{s_0}[B_{\R^n}(R)][\C^{m_1}]$ is a solution to \cref{Eqn::AppendRealAnal::MainEqn} and
$\ZygNorm{b}{s_0}[B_{\R^n}(R)]\leq \gamma$, then $b\in \sBSpace{n}{m_1}{\eta_0}$ and
$\sBNorm{b}{n}{m_1}{\eta_0}\leq C$, where $C=C(\sE, \Gamma, R, s_0)$.  See \cref{Section::FuncRev::RealAnal} for the definition of
$\sBSpace{n}{m_1}{\eta_0}$.
\end{prop}

We outline a proof of \cref{Prop::AppendRealAnal::MainProp} by following the proof from \cite{MorreyOnTheAnalyticityOfTheSolutionsOfAnalyticNonLinearEllipticSystems},
which becomes somewhat simpler in this special case and is therefore easier to extract the needed quantitative estimates.
In what follows, we write $A\lesssim B$ to mean $A\leq CB$, where $C$ can be chosen to depend only on $\sE$, $\Gamma$, $R$, and $s_0$.
Throughout the rest of this section, we take the setting of \cref{Prop::AppendRealAnal::MainProp}; in particular, we are given a solution
$b\in \ZygSpace{s_0}[B_{\R^n}(R)][\C^{m_1}]$ to \cref{Eqn::AppendRealAnal::MainEqn} as in that proposition.  Our goal is to pick $\gamma$ and $\eta_0$
so that the conclusions of the proposition hold.

Without loss of generality, by possibly shrinking $s_0$, we may assume $s_0=1+\mu$, where $\mu\in (0,1)$.  Thus, the space
$\ZygSpace{s_0}[B]$ coincides with the H\"older space $\HSpace{1}{\mu}[B]$ for any ball $B$,\footnote{That $\ZygSpace{1+\mu}(\Omega)=\HSpace{1}{\mu}(\Omega)$ for a bounded, Lipschitz domain $\Omega$ (and $\mu\ne 0,1$) is classical
and follows easily from \cite[Theorem 1.118 (i)]{TriebelTheoryOfFunctionSpacesIII}.}
 which allows us to use the results from \cite{MorreyOnTheAnalyticityOfTheSolutionsOfAnalyticNonLinearEllipticSystems}
which deal with H\"older spaces.  For the rest of the section, we continue to use the notation $\ZygSpace{j+\mu}$ for $j\in \N$,
but (just in this section) the reader is free to interpret it either as $\ZygSpace{j+\mu}$ or $\HSpace{j}{\mu}$; indeed in this section we only deal with
$\mu\in (0,1)$ fixed and $\ZygSpace{j+\mu}[\Omega]$, $\HSpace{j}{\mu}[\Omega]$ for bounded Lipschitz domains $\Omega$ in which case
these two spaces have equivalent norms.

First we need a quantitative version of the classical fact that the solution $b$ is smooth.  This is discussed in an appendix to \cite{StovallStreetII}.
There it is shown that $\exists \gamma_1=\gamma_1(\sE,\Gamma)>0$ such that if $\ZygNorm{b}{1+\mu}[B_{\R^n}(R)]\leq \gamma_1$,
then $b\in \ZygSpace{2+\mu}[B_{\R^n}(R/2)][\C^{m_1}]$ with $\ZygNorm{b}{2+\mu}[B_{\R^n}(R/2)]\lesssim \ZygNorm{b}{1+\mu}[B_{\R^n}(R)]$.
We will choose $\gamma\leq \gamma_1$, so we may henceforth assume $b\in \ZygSpace{2+\mu}[B_{\R^n}(R/2)][\C^{m_1}]$ with $\ZygNorm{b}{2+\mu}[B_{\R^n}(R/2)]\lesssim \gamma$.

For $\eta,h>0$, set $D^n(\eta;h):=\{x+iy : x,y\in \R^n, |x|<\eta, |y|<h(\eta-|x|)\}$ and set, for $s>0$,
\begin{equation*}
\DSpace{n}{m_1}{\eta}{h}{s}:=
\left\{
f:B_{\R^n}(\eta)\rightarrow \C^{m_1}\: \big|\:
f\text{ is real analytic and extends to a holomorphic function }E(f)\in \ZygSpace{s}[D^n(\eta;h)][\C^{m_1}]
\right\}.
\end{equation*}
With the norm
\begin{equation*}
\DNorm{f}{n}{m_1}{\eta}{h}{s}:=\ZygNorm{E(f)}{s}[D^n(\eta;h)][\C^{m_1}],
\end{equation*}
$\DSpace{n}{m_1}{\eta}{h}{s}$ is a Banach space.

\begin{lemma}\label{Lemma::AppendReal::ExistsC1}
There exists a bounded linear map
\begin{equation*}
\sP:\ZygSpace{\mu}[B_{\R^n}(R/2)][\C^{m_1}]\rightarrow \ZygSpace{2+\mu}[B_{\R^n}(R/2)][\C^{m_1}]
\end{equation*}
such that $\sE^{*}\sE \sP=I$ and $\exists h=h(\sE,R)>0$ such that $\sP$ restricts to a bounded map
\begin{equation*}
\sP:\DSpace{n}{m_1}{R/2}{h}{\mu}\rightarrow \DSpace{n}{m_1}{R/2}{h}{2+\mu}
\end{equation*}
and such that if we set $V_0:=\sP \sE^{*} \Gamma(b,\grad b)$ and $H:=b-V_0$,
then $\ZygNorm{V_0}{2+\mu}[B_{\R^n}(R/2)][\C^{m_1}]\leq C_1 \gamma$,
$\ZygNorm{H}{2+\mu}[B_{\R^n}(R/2)][\C^{m_1}]\leq C_1 \gamma$,
and $H\in \DSpace{n}{m_1}{R/2}{h}{2+\mu}$ with $\DNorm{H}{n}{m_1}{R/2}{h}{2+\mu}\leq C_1 \gamma$.
Here, $C_1=C_1(\sE, \Gamma, R, s_0)>0$.
\end{lemma}
\begin{proof}[Comments on the proof]
This is essentially a special case of Theorems A, B, and C of \cite{MorreyOnTheAnalyticityOfTheSolutionsOfAnalyticNonLinearEllipticSystems}; here we are applying these theorems
to the elliptic operator $\sE^{*}\sE$ and using that $\sE^{*}\sE H=0$ by the definitions.  In \cite{MorreyOnTheAnalyticityOfTheSolutionsOfAnalyticNonLinearEllipticSystems}, these theorems
were stated on the subspace of functions which vanish at $0$, though this is not an essential point.  Moreover, in the special case we are interested in, $\sE^{*}\sE$ is essentially the Laplacian (see \cref{Eqn::AppendElliptic::ComputeSquare} for $\sE^{*} \sE$ in the case we are interested in).  In this case, the above result follows from standard methods.
\end{proof}

Define $\sT(V):=\sP\sE^{*} \Gamma(H+V,\grad(H+V))$; by the definition of $V_0$, $\sT(V_0)=V_0$.

\begin{lemma}\label{Lemma::AppendRealAnal::ChooseGamma}
Let $C_1>0$ be as in \cref{Lemma::AppendReal::ExistsC1}.  If $\gamma=\gamma(\sE,\Gamma,R,s_0)>0$ is sufficiently small
and $\ZygNorm{V_1}{2+\mu}[B_{\R^n}(R/2)][\C^{m_1}], \ZygNorm{V_2}{2+\mu}[B_{\R^n}(R/2)][\C^{m_1}]\leq C_1\gamma$
then,
\begin{equation*}
\ZygNorm{\sT(V_1)}{2+\mu}[B_{\R^{n}}(R/2)][\C^{m_1}]\leq C_1 \gamma,\quad
\ZygNorm{\sT(V_1)-\sT(V_2)}{2+\mu}[B_{\R^n}(R/2)][\C^{m_1}] \leq \frac{1}{2} \ZygNorm{V_1-V_2}{2+\mu}[B_{\R^n}(R/2)][\C^{m_1}].
\end{equation*}
The same results holds for $\ZygSpace{2+\mu}[B_{\R^n}(R/2)][\C^{m_1}]$ replaced by $\DSpace{n}{m_1}{R/2}{h}{2+\mu}$, throughout.
\end{lemma}
\begin{proof}
Since $\ZygNorm{V_1}{2+\mu}[B_{\R^{n}}(R/2)][\C^{m_1}]\leq C_1 \gamma$ and $\ZygNorm{H}{2+\mu}[B_{\R^n}(R/2)][\C^{m_1}]\leq C_1 \gamma$, it
follows from \cref{Prop::FuncSpaceRev::Algebra} that
$\ZygNorm{\Gamma(H+V_1,\grad(H+V_1))}{1+\mu}[B_{\R^n}(R/2)][\C^{m_1}]\lesssim (C_1 \gamma)^2$.  \Cref{Lemma::AppendReal::ExistsC1} implies
$\ZygNorm{\sT(V_1)}{2+\mu}[B_{\R^n}(R/2)][\C^{m_1}]\lesssim (C_1 \gamma)^2$; and so if $\gamma$ is sufficiently small it follows that
$\ZygNorm{\sT(V_1)}{2+\mu}[B_{\R^n}(R/2)][\C^{m_1}]\leq C_1 \gamma$.
Similarly, again using  \cref{Prop::FuncSpaceRev::Algebra}, we have
\begin{equation*}
\ZygNorm{\Gamma(V_1-V_2, \grad(H+V_1))}{1+\mu}[B_{\R^n}(R/2)][\C^{m_1}], \ZygNorm{\Gamma(H+V_2, \grad(V_1-V_2))}{1+\mu}[B_{\R^n}(R/2)][\C^{m_1}]\lesssim \gamma \ZygNorm{V_1-V_2}{2+\mu}[B_{\R^n}(R/2)][\C^{m_1}].
\end{equation*}
Since $\sT(V_1)-\sT(V_2)=\sP \sE^{*} ( \Gamma(V_1-V_2, \grad(H+V_1)) - \Gamma(H+V_2, \grad(V_1-V_2)) )$ it follows from \cref{Lemma::AppendReal::ExistsC1} that
\begin{equation*}
\ZygNorm{\sT(V_1)-\sT(V_2)}{2+\mu}[B_{\R^n}(R/2)][\C^{m_1}] \lesssim \gamma  \ZygNorm{V_1-V_2}{2+\mu}[B_{\R^n}(R/2)][\C^{m_1}].
\end{equation*}
Taking $\gamma$ sufficiently small, we have
\begin{equation*}
\ZygNorm{\sT(V_1)-\sT(V_2)}{2+\mu}[B_{\R^n}(R/2)][\C^{m_1}] \leq \frac{1}{2}  \ZygNorm{V_1-V_2}{2+\mu}[B_{\R^n}(R/2)][\C^{m_1}],
\end{equation*}
as desired.
The same proof works with $\ZygSpace{j+\mu}[B_{\R^n}(R/2)][\C^{m_1}]$ replaced by $\DSpace{n}{m_1}{R/2}{h}{j+\mu}$, throughout.
\end{proof}

\begin{proof}[Proof of \cref{Prop::AppendRealAnal::MainProp}]
By taking $\gamma>0$ sufficiently small, as in \cref{Lemma::AppendRealAnal::ChooseGamma}, we see that $V_0$ is the unique fixed point of the strict
contraction $\sT$, acting on the complete metric space $\{ V: \ZygNorm{V}{2+\mu}[B_{\R^{n}}(R/2)][\C^{m_1}]\leq C_1 \gamma\}$.
This fixed point agrees with the fixed point of $\sT$ when acting on the complete metric space
$\{ V : \DNorm{V}{n}{m_1}{R/2}{h}{2+\mu}\leq C_1 \gamma\}$ (on which is it a strict contraction by \cref{Lemma::AppendRealAnal::ChooseGamma}).
We conclude $\DNorm{V_0}{n}{m_1}{R/2}{h}{2+\mu}\leq C_1 \gamma\lesssim 1$.
Since $\DNorm{H}{n}{m_1}{R/2}{h}{2+\mu}\leq C_1 \gamma\lesssim 1$, by
\cref{Lemma::AppendReal::ExistsC1}, and since $b=H+V_0$, we have
$\DNorm{b}{n}{m_1}{R/2}{h}{2+\mu}\leq 2C_1 \gamma\lesssim 1$.  Taking $\eta_0=\eta_0(R/2,h)>0$ sufficiently small we have $B_{\R^n}(\eta_0)\subseteq D^n(R/2;h)$ and therefore,
\begin{equation*}
\sBNorm{b}{n}{m_1}{\eta_0}\leq  \DNorm{b}{n}{m_1}{R/2}{h}{2+\mu}\lesssim 1,
\end{equation*}
completing the proof.
\end{proof}

	
	\subsection{Existence for a Nonlinear Elliptic Equation}
Fix $D>0$, $m_1,m_2\in \N$.  For functions $A:B_{\R^n}(D)\rightarrow \C^{m_1}$ and $B:B_{\R^n}(D)\rightarrow \C^{m_2}$ we write
\begin{equation*}
\Deriv^1 A=(\partial_x^{\alpha} A)_{|\alpha|\leq 1}, \quad \Deriv^2 B=(\partial_x^{\alpha} B)_{|\alpha|\leq 2}, \quad \Deriv_2 B=(\partial_x^{\alpha} B)_{|\alpha|=2},
\end{equation*}
so that, for example, $\Deriv^2 B$ is the vector of all partial derivatives of $B$ up to order $2$, and $\Deriv_2 B$ is the vector of all
partial derivatives of order exactly $2$.

Fix a $C^\infty$ function $g$.  We consider the equation
\begin{equation}\label{Eqn::AppendExist::MainEqn}
g(\Deriv^1 A(x), \Deriv^2 B(x))=0.
\end{equation}
Here, $g$ is a $C^\infty$ function defined on a neighborhood of the origin, takes values in $\C^{m_2}$, and satisfies $g(0,0)=0$.
Our goal is to give conditions on $g$ so that given $A$ (sufficiently small), we can find $B=B(A)$ so  that \cref{Eqn::AppendExist::MainEqn} holds;
and we wish to further understand how the regularity of $B$ depends on the regularity of $A$, in a quantitative way.

Though it is not necessary for the results which follow, we assume \cref{Eqn::AppendExist::MainEqn} is quasilinear in $B$, which is
sufficient for our purposes and simplifies the proof.  That is, we assume
\begin{equation}\label{Eqn::AppendExist::Quasi}
g(\Deriv^1 A(x), \Deriv^2 B(x)) = g_1(A(x), \Deriv^1 B(x)) \Deriv_2 B(x) + g_2(\Deriv^1 A(x), \Deriv^1 B(x)),
\end{equation}
where $g_1$ and $g_2$ are smooth on a neighborhood of the origin, $g_1$ takes values in matrices of an appropriate size, and $g_2(0,0)=0$.

Finally, let $\sE_2$ denote the second order partial differential operator
\begin{equation*}
\sE_2 B := g_1(0,0) \Deriv_2 B,
\end{equation*}
so that $\sE_2$ is an $m_2\times m_2$ matrix of constant coefficient partial differential operators of order $\leq 2$.

\begin{prop}\label{Prop::AppendExist::MainProp}
Suppose $\sE_2$ is elliptic.  Fix $s_0>0$ and a neighborhood $N\subseteq \ZygSpace{2+s_0}[B_{\R^n}(D)][\C^{m_2}]$ of $0$.
Then, there exists a neighborhood $W\subseteq \ZygSpace{1+s_0}[B_{\R^n}(D)][\C^{m_2}]$ of $0$ and a map
$\BofA:W\rightarrow N$ such that
$g(\Deriv^1 A(x), \Deriv^2 \BofA(A)(x)) =0$ for $x\in B^n(D)$, $A\in W$.
This map satisfies $\Deriv^1 \BofA(A)(0)=0$, $\forall A\in W$, and
$\ZygNorm{\BofA(A)}{2+s_0}[B_{\R^n}(D)][\C^{m_2}]\leq C \ZygNorm{A}{1+s_0}[B_{\R^n}(D)][\C^{m_1}],$
where $C$ does not depend on $A\in W$.  Finally, for $\eta\in (0,D)$, let $R_\eta$ denote the restriction map
$R_\eta: f\mapsto f\big|_{B_{\R^n}(\eta)}$.  Then, for $s\geq s_0$, $\eta\in (0,D)$,
$R_\eta \circ \BofA : \ZygSpace{1+s}[B_{\R^n}(D)][\C^{m_1}]\cap W \rightarrow \ZygSpace{2+s}[B_{\R^n}(\eta)][\C^{m_2}],$
and 
$\ZygNorm{R_\eta\circ \BofA(A)}{2+s}[B_{\R^n}(\eta)][\C^{m_2}]\leq C_{s,\eta},$
where $C_{s,\eta}$ can be chosen to depend on an upper bound for $\ZygNorm{A}{1+s}[B_{\R^n}(D)][\C^{m_1}]$ and does not
depend on $A\in W$ in any other way.  It can depend on any of the other ingredients in the problem.
\end{prop}

See \cite{StovallStreetII} for a discussion of this proposition.
	
	\subsection{An Elliptic Operator}
In this section, we discuss a particular first order, overdetermined, constant coefficient, linear, elliptic operator which is needed
in this paper. 
For $t\in \R^r$ and $z\in \C^n$, we consider functions
$A(t,z):\R^r\times \C^n\rightarrow \C^r$ and
$B(t,z):\R^r\times \C^n\rightarrow \C^n$.
We define
\begin{equation*}
\begin{split}
\sE(A,B) =
\Bigg(&
\left(\frac{\partial A_{k_1}}{\partial t_{k_2}} - \frac{\partial A_{k_2}}{\partial t_{k_1}}\right)_{1\leq k_1<k_2\leq r},
\left(\frac{\partial A_{k}}{\partial \zb[j]} - \frac{\partial B_j}{\partial t_k}\right)_{\substack{1\leq k\leq r \\ 1\leq j\leq n}},\\
&\left(\frac{\partial B_{j_1}}{\partial \zb[j_2]} - \frac{\partial B_{j_2}}{\partial \zb[j_1]}
\right)_{1\leq j_1<j_2\leq n},
\sum_{k=1}^r \frac{\partial A_k}{\partial t_k} + \sum_{j=1}^n \frac{\partial B_j}{\partial z_j}
\Bigg).
\end{split}
\end{equation*}

\begin{lemma}\label{Lemma::AppendElliptic::OneOperator}
$\sE$ is elliptic.
\end{lemma}
\begin{proof}
It is straightforward to directly compute $\sE^{*}\sE$
to see
\begin{equation}\label{Eqn::AppendElliptic::ComputeSquare}
\sE^{*} \sE (A,B) = -\left(\sum_{k=1}^{r} \frac{\partial^2}{\partial t_k^2} + \sum_{j=1}^n \frac{\partial^2}{\partial z_j \partial \zb[j]} \right)(A,B),
\end{equation}
and the result follows.
\end{proof}

There is another way to interpret \cref{Eqn::AppendElliptic::ComputeSquare}.  Indeed, we identify
$(A,B)$ with the one form $\Psi:=A_1 dt_1+\cdots+A_r dt_r + B_1d\zb[1]+\cdots+B_n d\zb[n]$.
We let $d$ denote the usual de Rham complex acting in the $t$ variable and
$\dbar$ denote the usual $\dbar$-complex acting in the $z$ variable.
We let $d^{*}$ and $\dbar^{*}$ denote the adjoints of these two complexes.
Then $\sE(A,B)$ can be identified with $(d+\dbar, -(d+\dbar)^{*}) \Psi$.   And so $\sE^{*}\sE$ can be identified
with $(d+\dbar)^{*}(d+\dbar) + (d+\dbar)(d+\dbar)^{*}=d^{*}d+dd^{*} + \dbar^{*}\dbar + \dbar\dbar^{*} + (\dbar^{*} d +d \dbar^{*})  + (d^{*} \dbar + \dbar d^{*}) = d^{*}d+dd^{*} + \dbar^{*}\dbar + \dbar\dbar^{*}$.


\bibliographystyle{amsalpha}

\bibliography{nirenberg}

\center{\it{University of Wisconsin-Madison, Department of Mathematics, 480 Lincoln Dr., Madison, WI, 53706}}

\center{\it{street@math.wisc.edu}}

\center{MSC 2010:  58A30 (Primary), 53C15 (Secondary)}


\end{document}